\documentclass[11pt]{amsart}
\usepackage{amsfonts,amsmath,latexsym,amssymb,verbatim,amsbsy,amsthm}
\usepackage{graphicx}
\usepackage{caption}
\usepackage{float}
\usepackage{wrapfig}
\usepackage[shortlabels]{enumitem}

\usepackage[top=1in, bottom=1in, left=1in, right=1in]{geometry}

\usepackage[dvipsnames]{xcolor}

\usepackage[colorlinks=true, pdfstartview=FitV, linkcolor=RoyalBlue,citecolor=ForestGreen, urlcolor=blue]{hyperref}

\theoremstyle{plain}
\newtheorem{THEOREM}{Theorem}[section]

\newtheorem{theorem}[THEOREM]{Theorem}
\newtheorem{corollary}[THEOREM]{Corollary}
\newtheorem{lemma}[THEOREM]{Lemma}
\newtheorem{proposition}[THEOREM]{Proposition}

\theoremstyle{definition}

\theoremstyle{remark}

\newtheorem{remark}[THEOREM]{Remark}

\newcommand{\ep}{\varepsilon}

\DeclareMathOperator{\supp}{supp} %
\def \b {\beta}
\def \g {\gamma}
\def \d {\delta}
\def \k {\kappa}
\def \e {\varepsilon}

\def \l {\lambda}
\def \n {\nabla}
\def \s {\sigma}

\def \w {\omega}

\def \D {\Delta}

\def \bu {{\bf u}}
\def \bv {{\bf v}}

\def \bx {{\bf x}}
\def \by {{\bf y}}

\def \bx {{\bf x}}

\def \by {{\bf y}}
\def \cK {\mathcal{K}}
\def \cL {\mathcal{L}}

\def \cN {\mathcal{N}}

\def \cR {\mathcal{R}}

\def \ff {\mathsf{f}}
\def \hh {h}
\def \HH {H}

\def \II {\mathsf{I}}
\def \JJ {\mathsf{J}}
\usepackage{amsmath}

\newcommand{\N}{\ensuremath{\mathbb{N}}}   %%% naturals
\newcommand{\R}{\ensuremath{\mathbb{R}}}   %%% reals
\newcommand{\T}{\ensuremath{\mathbb{T}}}   %%% torus

\newcommand{\pa}{\partial}
\newcommand{\ldosv}{L^2_{\varphi'_{\kappa}}}
\def \sign {\mathrm{sgn}}

\def \p {\partial}

\def \dx  {\, \mbox{d}x}

\def \dy  {\, \mbox{d}y}
\def \dz  {\, \mbox{d}z}

\def \ddt  {\frac{\mbox{d\,\,}}{\mbox{d}t}}

\begin{document}

\title[2D EULER: TRAVELING WAVES NEAR COUETTE]{Traveling waves  near Couette Flow \\for the 2D Euler equation}

\author{\'Angel Castro} \author{Daniel Lear}

\address{Instituto de Ciencias Matematicas, Consejo Superior de Investigaciones Cient\'ificas, Madrid.}
\email{angel\_castro@icmat.es}

\address{Department of Mathematics, Statistics, and Computer Science, University of Illinois, Chicago.}
\email{lear@uic.edu}

\date{\today}



%\thanks{\textbf{Acknowledgment.}  AC is supported by the Spanish Ministry of Science and Innovation, through
%the “Severo Ochoa Programme for Centres of Excellence in R\&D (CEX2019-000904-S)”,
%and by grants  Europa Excelencia program ERC2018-092824 and  RED2018-102650-T funded by MCIN/AEI/10.13039/501100011033. AC and DL are supported by grants MTM2017-89976-P and  PID2020-114703GB-I00 funded by MCIN}

\begin{abstract}
In this paper we reveal the existence of  a large family of new, nontrivial and smooth traveling waves for the 2D Euler equation at an arbitrarily small distance  from the Couette flow  in $H^s$, with $s<3/2$, at the level of the vorticity. The speed of these waves is of order 1 with respect to this distance. This result strongly contrasts  with the setting of very high regularity in Gevrey spaces (see \cite{BM}), where the  problem exhibits an inviscid damping mechanism that leads to relaxation of perturbations  back to nearby shear flows. It also complements the fact that there not exist nontrivial  traveling waves in the $H^{\frac{3}{2}+}$  neighborhoods of Couette flow (see \cite{LZ}).
\end{abstract}

\maketitle

%\addtocontents{toc}{\protect\setcounter{tocdepth}{1}}
\tableofcontents

%\newpage
\section{Introduction and main result}
In this paper we deal with the incompressible Euler equations in the two-dimensional
space $D:=\T\times\R$. In the vorticity formulation 2D Euler is given by the transport equation
\begin{equation}\label{e:Euler2D}
(\bx,t)\in D\times\mathbb{R}^{+},\qquad \left\{
\begin{split}
&\pa_t w +\bv \cdot \n w =  0, \\
&\bv=\n^{\perp}\D^{-1}w,\\
& w|_{t=0}=  w_0,
\end{split}\right.
\end{equation}
where $\bv=(v_1,v_2)$ denotes the velocity field and $w =\pa_x v_2 -\pa_y v_1$ its vorticity. The starting point of this paper is the fact that the Couette flow, given by $\bv=(y,0)$ and $w = -1$,  is a steady state for \eqref{e:Euler2D}. We are interested in perturbations of the form $\bv(\bx,t):=(y,0)+ \bu(\bx,t)$ with total vorticity $w(\bx,t):=-1+\omega(\bx,t),$ where $\omega=\text{curl} (\bu).$ The equations for a perturbation around the Couette flow are given by
\begin{equation}\label{e:E2Dperturbado}
\left\{
\begin{split}
&\pa_t \omega +y \pa_x \omega +\bu \cdot \n \omega =  0, \\
&\bu=\n^{\perp}\D^{-1} \omega,\\
&\omega|_{t=0}=  \omega_0.
\end{split}\right.
\end{equation}
Here, the operator $\D^{-1}$ is given by expression \eqref{e:streamfunction}.\medskip

Stability of \textit{shear flows}, to which the Couette flow  belongs, is a widely investigated problem in the  context of hydrodynamic stability, a field where many  questions still remain unanswered. We refer to \cite{BGMsurvey,G} for an overview of this topic and a detailed list of references.

The Couette flow is one of the simplest shear flows  if not the simplest one, however, it poses several long-standing puzzles in hydrodynamics.
The issue that Couette flow is known to be spectrally stable for all Reynolds
numbers in contradiction with instabilities observed in experiments is now often referred  as the
\textit{Sommerfeld paradox} or \textit{turbulence paradox}. There have been many attempts in the literature to find an explanation to this paradox starting in the nineteenth century with Stokes, Helmholtz, Reynolds, Rayleigh, Kelvin, Orr, Sommerfeld and many others.\medskip

The linear analysis for the Couette flow has been obtained via an eigenvalue analysis, however, the classical stability analysis in general does not agree with the numerical and physical observations.  A key observation was made by Trefethen et al. \cite{Trefethen-Trefethen-Reddy-Driscoll}, where it has been shown that a pure eigenvalues analysis could hide several problems. These problems are due to the fact that the operators involved in the linearization around a shear flow are in general non-normal. Nowadays, the idea that the interaction between nonlinear effects and non-normal transient growth can lead to instabilities is classical in fluid mechanics, see \cite{Trefethen,Trefethen-Trefethen-Reddy-Driscoll}. We refer to \cite{LL} and the references therein for a modern approach and a mathematical treatment to the problem.

There were many attempts to find an explanation to the Sommerfeld paradox. One of the first attempt might be due to  Orr, who studied the linear
stability directly by considering the initial value problem. In this case, the problem reduces to
a simple transport equation for the vorticity. By recovering the velocity via the Biot-Savart law, Orr observed that  the velocity may experience a transient growth, suggesting that this phenomenon may be a possible source of instabilities in the nonlinear problem. In fact, one can see assuming enough regularity on the initial perturbation that the vertical velocity tends to zero when times goes to infinity. This convergence back to equilibrium, despite time reversibility and the lack of dissipative mechanisms, is now referred as \textit{inviscid damping}.
This phenomenon share analogies with the \textit{Landau damping} in the kinetic theory of plasma physics, see the paper of Mohout--Villani \cite{Mouhot-Villani} and references therein.\medskip

Despite the understanding of the linear setting, the nonlinear problem is substantially harder and remained unresolved until the breakthrough of Bedrossian--Masmoudi \cite{BM}. They established that sufficiently small perturbations in the Gevrey spaces $\mathcal{G}^s$, with $s>1/2$, converge strongly in $L^2$ (of velocity) to a  shear flow near Couette as times goes to infinity in $\T\times\R$.

As it was shown in \cite{DM} by Deng--Masmoudi, the Gevrey regularity requirement $s>1/2$   is  crucial. Indeed, Deng and Masmoudi prove long time instabilities in $\mathcal{G}^{1/2^{-}}$.

 In very recent works Ionescu--Jia \cite{Ionescu-Jia,Ionescu-Jia-monotonic} and Masmoudi--Zhao \cite{MZ} proved that nonlinear asymptotic stability holds true also for perturbations around the Couette flow and more general monotonic shear flows, with compactly supported vorticity, in the periodic channel $\T\times[-1,1].$\medskip

Motivated by these results, the linearized equations around more  general shear flows solutions were
investigated intensely in the last few years, and linear inviscid damping and decay was proven
in many cases of physical interest, see for example \cite{BZV,BMlinear,CZ,DZ1,DZ2,GNRS,Jia1,Jia2,Wei-Zhang-Zhao,Wei-Zhang-Zhao_1,Wei-Zhang-Zhao_2,Zillinger1,Zillinger2}.
For the analogous problem in the viscous setting or even in the compressible case, we refer to \cite{BZ,BGM,BMV,ZEW2,L,ZEW,ADM} for a comprehensive but not exhaustive list of references.  Similar analysis have been also carried out recently for the Boussinesq system, see for example \cite{bia1,bia2,deng,masmoudiboussinesq,Zillinger3,Zillinger4}.

For the most up-to-date overview of this subject, see  \cite{BGMsurvey} and references therein.\medskip

By all the above, an important direction to explore further is the role of nonlinear instabilities, and moreover, how the regularity of initial data may or may not affect the dynamics.

Nonlinear long time dynamics near Couette flow is related to the existence of nontrivial invariant structures as steady-states, traveling waves, etc. Non existence of nontrivial invariant structures near Couette flow is necessary for nonlinear inviscid damping. Conversely, their existence means nonlinear inviscid damping is not true, and long time dynamics near  Couette flow may be richer.\medskip

In this direction we refer to the work of Lin--Zeng \cite{LZ}, where the authors construct nontrivial steady flows arbitrarily close to Couette flow in  $H^s$, with $s<3/2$, in a finite channel  $\T\times[-1,1]$. In addition, the non existence of nontrivial traveling waves close to the Couette in $H^s$, with $s>3/2$, is proved in this paper. Here and in the rest of the paper ``trivial'' means independent of  $x$.

%\begin{remark}
%Here and in the rest ``trivial'' means for us independent of the horizontal variable $x$.
%\end{remark}
%\begin{remark}
%We work in the setting $\T\times\R$ but the same analysis can be done in the finite channel.
%\end{remark}
It is important to emphasize that the steady cat's eyes structure near
Couette flow, constructed in the aforementioned  paper \cite{LZ}, can be used to obtain  traveling waves using the symmetries of the system \eqref{e:Euler2D}.
Note that, by Galilean invariance, if $\bv(\bx,t)$ solves 2D Euler equation then $\overline{\bv}(\bx,t):=\bv(x+\lambda t,y,t)-(\lambda,0)$ also solves it. Therefore, if  $\bv(\bx)$ is a steady solution from \cite{LZ}, we get that $\overline{\bv}(\bx,t):=\bv(x+\lambda t,y)-(\lambda,0)$ is a traveling wave solution of 2D Euler.

Since steady solutions in \cite{LZ} satisfy $v_1(\bx)=y+O(\epsilon)$,  where $\epsilon$ measure the distance to the Couette flow, the size of $\overline{v}_1-y$ will  be $\lambda$. So, in order to obtain a traveling wave $\epsilon-$close to the Couette flow, the speed of the wave must be of the same order, i.e., $\lambda=O(\epsilon).$

In the present paper we are concerned with the existence of nontrivial and smooth traveling waves close to Couette flow in the $H^{3/2^{-}}$ topology with speed $\lambda$ of order 1 and with $v_1=y+O(\epsilon)$. Then our solutions are of a different nature that those of \cite{LZ}.

It is necessary to mention here also the result of Li--Lin \cite{LL}, where the authors prove the existence of traveling waves bifurcating from the sequence, $
v_n(y):= y +\frac{A}{n}\sin(4n\pi y)$ for $ \frac{1}{8\pi}< A < \frac{1}{4\pi}.$
As $n\to \infty$, the oscillatory shears approach the Couette flow, i.e., $v_n(y) \to y$ in $L^2$
and $L^{\infty}$. On the other hand, in the vorticity variable, the oscillatory shears do not
approach the Couette  since $\pa_y v_n(y) = 1+4 A \pi \cos(4n \pi  y) \not\to 1$ in any Lebesgue norm.
\medskip

Recently, there has been a growing interest in the study of existence or not of invariant structures as steady-states and traveling waves for the incompressible 2D Euler equations near other shear flows and for related equations. For example, we refer to \cite{LYZ,LWZZ} for the 2D Euler equation with Coriolis force or \cite{ZEW2} for the case of Kolmogorov and Poiseuille flows.\medskip

For the sake of clarity we shall now give an elementary statement of our main result.

\begin{theorem}\label{thmbasic} For  any $0\leq s<3/2$ and $\epsilon>0$,  the perturbed 2D Euler system \eqref{e:E2Dperturbado} admits a nontrivial  traveling wave solution that satisfies the smallness condition $$\|w+1\|_{H^s(\T\times\R)}\equiv\|\omega\|_{H^s(\T\times\R)}< \epsilon,$$
with $\omega\in C^\infty(\T\times\R)$ compactly supported, and whose speed is $O(1)$ with respect to $\epsilon$.
\end{theorem}\medskip

This result will be the consequence of Theorems \ref{thm5}, \ref{thm6} and \ref{thm7} where one can get a deeper insight on the properties of these solutions.\medskip

An important tool behind the proof will be the bifucartion theory, which has been very useful to prove the existence of solutions with a global structure in several equations arising in the field of fluid mechanics.  Following the approach of Burbea \cite{B}, there has been several works concerning the existence of  single or multiple patches
moving without changing shape, not only for 2D Euler equation but also for the generalized Surface Quasi-Geostrophic equations (SQG$)_{\beta}$ where $0<\beta<2$. We refer to \cite{CCG,CCG1,HH,HMV} for single rotating patches, \cite{DHH,DHMV} for doubly connected
V-states, \cite{HM,has} for corotating and counter-rotating vortex pairs and \cite{javi1} for steady states. See \cite{javi2} for further properties of rotating solutions and \cite{javi3,javi4} for the case of the vortex-sheet problem. See \cite{cor1,cor2,cor3} for related constructions.

The existence of smooth rotating vortices is  more intricate due to the higher dimension of the space on which the linear part of the equation acts. In spite of this, in \cite{CCG2,CCG3} smooth rotating solutions were constructed for SGQ equation and for 2D Euler equation, respectively. See also \cite{delpino} for  the construction of a different type of smooth rotating solutions for SQG and \cite{gravejat} for the existence of traveling waves  also for SQG.  The existence of non-smooth rotating vortices with non-uniform densities can be found in  \cite{GHS} for 2D Euler.\medskip

In a broad sense, these results connect with that developed  in \cite{BZV}, where the authors analyze the incompressible 2D Euler equation linearized around a radially symmetric, strictly monotone decreasing vorticity distribution. For sufficiently regular data some interesting phenomena appear, such as \textit{vortex axysimmetrization} (the vorticity weakly converges back to radial symmetry) and \textit{vorticity depletion} (faster inviscid damping rates than those possible with passive scalar evolution). Finally, for perturbations around coherent vortex structures we also refer to \cite{CZ,GW,Ionescu-Jia-vortex1,Ionescu-Jia-vortex2,Wei-Zhang-Zhao} and references therein.

\subsection{Sketch of the proof}
We will assume that the level sets of the vorticity are given by the family of graphs $(x,f(\bx,t))$, i.e.,
\[
\omega(x,f(\bx,t),t)=\varpi(y).
\]
Then, all the information of the problem is encoded in the time-evolution of level curves $f(\bx,t)$ and in the profile function $\varpi.$ After some algebraic manipulations the problem reduces to the equation:
\begin{equation}\label{e:sketchlevelcurveseq}
\ddt f(\bx,t)+f(\bx,t)\pa_xf(\bx,t) =\mathbf{U}_{\varpi}[f](\bx,t), \qquad \text{for } \bx\in \T\times \supp(\varpi'),
\end{equation}
where $\mathbf{U}_{\varpi}[\cdot]$ is a nonlocal and nonlinear functional defined below in \eqref{operator:u(f)}.

In order to look for traveling wave solutions we will take
\begin{equation*}\label{e:sketchansatz}
f(\bx,t):=y+\mathsf{f}(x+\l t,y)
\end{equation*}
which will lead to an equation for $(\lambda,\ff(\bx))$ of the form,
\begin{equation}\label{e:sketchfunctionaleq}
F_{\varpi}[\l,\ff](\bx)=0,  \qquad \text{for } \bx\in \T\times \supp(\varpi').
\end{equation}
with $F_\varpi[\lambda,\ff]$ given in \eqref{Fdefi}.  This equation we will be solved by using the Crandall-Rabinowitz theorem (C-R). Note that $F_{\varpi}[\lambda,0]=0$  for all $\lambda\in\R$ and one can expect to bifurcate from zero for some values of $\lambda$. We also be able to get enough information on $\ff(\bx)\equiv \ff(x,y)$ to guarantee that it depend on the horizontal variable $x$ in a nontrivial way.

The size of our solution $\omega$ in $H^s(\T\times \R)$, $0<s<\frac{3}{2}$, which coincide with the distance of our solutions $w=-1+\omega$ to the Couette flow, will be given by the size of the profile $\varpi$ in $H^s(\R)$. Thus, we have to found a small enough profile $\varpi$ for which we can solve \eqref{e:sketchfunctionaleq}. A full description of $\varpi$  will be given in Section \ref{profile}. At this point let us says that it will be a $\kappa-$regularization of the profile of size $\ep$ we see in  Figure \ref{fig:trapeciokappa}:

\begin{figure}[h]
  \centering
    \includegraphics[scale=1]{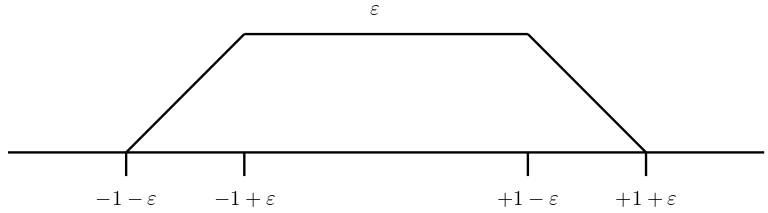}
    \caption{Profile $\varpi_{\ep,\, 0}$.}
  \label{fig:trapeciokappa}
\end{figure}

The core of the paper will be to study the spectral properties of the operator $$\mathcal{L}[\lambda]:=D_\ff F_{\varpi_{\ep,\kappa}}[\lambda,0].$$

This analysis will be made by an asymptotic analysis on $\ep$.

\subsection{Organization }
The remainder of the paper is organized as follows. In  Section \ref{s:formulation}, we shall write the equation for the level curves of the vorticity through the  Biot-Savart law. In Section \ref{s:bifurcation}, we shall
introduce and review some background material on the bifurcation theory and  Crandall-Rabinowitz theorem.
In Section \ref{s:regularity}, we will define the spaces we will work with in order to apply the C-R theorem and study the regularity of the nonlinear functional. In Section \ref{s:analysislinear}, we conduct the spectral study and check that our equation satisfies the hypotheses of the C-R theorem. In Section \ref{s:distance}, we obtain quantitative bounds for the $H^{s}$-norm in terms of all the parameters involved in the problem. We address the full regularity of the traveling wave solution in Section \ref{s:fullregularity}. Finally, in the last section, in order to facilitate the presentation, we collect in the Appendix the proofs of all technical lemmas used in the paper.

\section{Formulation of the problem}\label{s:formulation}

In this section  it will be obtained  an equation whose solutions yield traveling waves  of \eqref{e:E2Dperturbado}.

\subsection{The 2D Euler as an equation for the level curves of the vorticity}

 We will obtain the equation for the level curves of the vorticity of a solution of the 2D Euler equation \eqref{e:E2Dperturbado}.  We will use the stream function $\psi$, which satisfies
\begin{equation}\label{e:streamformulation}
\bu=\n^{\perp}\psi, \qquad \D\psi=\omega.
\end{equation}
The Green function for the Laplacian in the strip $\T\times\R$  is given (see \cite{CGO}) by the convolution with the kernel
\begin{equation*}
K(\bar{\bx})=\frac{1}{4\pi}\log[\cosh(\bar{y})-\cos(\bar{x})].
\end{equation*}
Consequently, the stream function $\psi=K\ast \omega$ is given by the expression
\begin{equation}\label{e:streamfunction}
\psi(\bx,t)=\frac{1}{4\pi}\int_{\T\times\R} \log\left(\cosh\left(y-\bar{y}\right)-\cos(x-\bar{x})\right)\omega(\bar{\bx},t)d\bar{\bx}.
\end{equation}

We will find solutions of \eqref{e:E2Dperturbado} by looking to the level curves of $\omega$. Assuming that  these level curves can be parameterized by the family of  graphs  $(x,\,f(\bx,t))$, with $\pa_y f(\bx,t)>0$, we have that
\begin{align}\label{levels}\omega(x, f(\bx,t),t)=\varpi(y),\end{align}
for some smooth and even profile function $\varpi$. We remark that  equation \eqref{e:E2Dperturbado} is a transport type equation, then the profile $\varpi$ can be taken  independent on $t$ without any loss of generality.

By straightforward computations, we get
\begin{equation}\label{e:gradw}
\n\omega(x,f(\bx,t),t)=\frac{\varpi'(y)}{\pa_y f(\bx,t)}(-\pa_x f(\bx,t),1).
\end{equation}
It is easily seen that taking a time derivative on \eqref{levels} and using \eqref{e:E2Dperturbado}  yields
\begin{equation}\label{inter1}
\pa_t f(\bx,t) \pa_y \omega (x,f(\bx,t),t)=f(\bx,t)\pa_x \omega(x,f(\bx,t),t)+ \bu(x,f(\bx,t),t)\cdot \n\omega (x,f(\bx,t),t).
\end{equation}
Plugging \eqref{e:gradw} into \eqref{inter1} we obtain
\begin{equation}\label{inter2}
\left[\p_t f (\bx,t)+f(\bx,t)\pa_x f(\bx,t)-\mathbf{U}[f](\bx,t)\right]\frac{\varpi'(y)}{\pa_y f(\bx,t)}=0,
\end{equation}
where
\begin{equation}\label{e:u(f)}
\mathbf{U}[f](\bx,t):=\bu(x,f(\bx,t),t)\cdot(-\pa_x f(\bx,t),1).
\end{equation}
Next we obtain an expression for $\mathbf{U}[f](\bx,t)$. Taking $\nabla^\perp$ on \eqref{e:streamfunction} we have that
\begin{align*}
\bu(\bx,t)=&\frac{1}{4\pi}\int_{\T\times\R}\log\left(\cosh(y-\bar{y})-\cos(x-\bar{x})\right)\nabla^\perp\omega(\bar{\bx},t)\mbox{d}\bar{\bx}\\
=&\frac{1}{4\pi}\int_{\supp(\nabla^{\perp}\omega)}\log\left(\cosh(y-\bar{y})-\cos(x-\bar{x})\right)\nabla^\perp\omega(\bar{\bx},t)\mbox{d}\bar{\bx}.
\end{align*}
Making the change of variable $\bar{\bx}=(\tilde{x},f(\tilde{\bx},t))$ and using expression \eqref{e:gradw} we have that
\begin{align}\label{u}
\bu(x,y,t)=-\frac{1}{4\pi}\int_{\T\times \supp(\varpi')}\log\left(\cosh(y-f(\tilde{\bx},t))-\cos(x-\tilde{x})\right)(1,\p_x f(\tilde{\bx},t))\varpi'(\tilde{y})\mbox{d}\tilde{\bx},
\end{align}
and
\begin{align*}
\bu(x,f(\bx,t),t)=-\frac{1}{4\pi}\int_{\T\times \supp(\varpi')}\log\left(\cosh(f(\bx,t)-f(\tilde{\bx},t))-\cos(x-\tilde{x})\right)(1,\p_x f(\tilde{\bx},t))\varpi'(\tilde{y})\mbox{d}\tilde{\bx}.
\end{align*}
Thus
\begin{align}\label{operator:u(f)}
&\mathbf{U}[f](\bx,t)\\
&=\frac{1}{4\pi}\int_{\T\times \supp(\varpi')}\log\left(\cosh(f(\bx,t)-f(\bar{\bx},t))-\cos(x-\bar{x})\right)(\pa_x f(\bx,t)-\partial_1f(\bar{\bx},t))\varpi'(\bar{y})\mbox{d}\bar{\bx}\nonumber\\
&=\frac{1}{4\pi}\int_{\T\times \supp(\varpi'(y-\cdot))} \varpi'(y-\bar{y})\log\left(\cosh(\D_{\bar{\bx}}[f](\bx,t))-\cos(\bar{x})\right) \D_{\bar{\bx}}[\pa_x f](\bx,t)\mbox{d}\bar{\bx}\nonumber,
\end{align}
where in the sequel we shall use the notation
\[
\D_{\bar{\bx}}[f](\bx,t):=f(\bx,t)-f(\bx-\bar{\bx},t).
\]

Since equation \eqref{inter2} is multiplied by $\varpi'(y)$ and in order to compute $\mathbf{U}[f](\bx,t)$ in $\T\times \supp(\varpi')$ one just needs to know $f(\bx,t)$ in $\T\times\supp(\varpi'),$ we have to solve
\begin{equation}\label{mainequation}
\p_t f(\bx,t)+f(\bx,t)\pa_x f(\bx,t)=\mathbf{U}[f](\bx,t),\qquad \bx=(x,y)\in \T\times \stackrel{\circ}\supp(\varpi').
\end{equation}
Conversely, if $f(\cdot,t)\in C^\infty(\T\times \supp(\varpi'))$, with $\pa_yf>0$, satisfies the previous equation \eqref{mainequation} on $\T\times \stackrel\circ\supp(\varpi')$ for some $\varpi\in C^\infty_c(\R)$, we can prove that the function $\omega(\cdot,t)$ defined by
\begin{align*}
\omega(x,f(\bx,t),t)=\varpi'(y) \quad \text{on} \quad  \T\times \supp(\varpi'),
\end{align*}
and extended by suitable constants to the complementary of $$\{ (x,y)\in \T\times \R\,:\, (x,y)=(\bar{x},f(\bar{\bx},t))\quad\text{with}\quad (\bar{x},\bar{y})\in \T\times \supp(\varpi')\}$$
 is a smooth solution of the original problem \eqref{e:E2Dperturbado} on $\T\times\R$.

\subsection{The equation for the traveling wave}

We will  look for solutions of the form
\begin{equation}\label{e:ansatz}
f(\bx,t)=y+\mathsf{f}(x+\l t,y).
\end{equation}
Putting the above ansatz  \eqref{e:ansatz} into  \eqref{mainequation} our problem reduces to solve the time independent equation
\[
\l \pa_x \ff(\bx)+(y+\ff(\bx))\pa_x\ff(\bx)=\mathbf{U}[y+\ff](\bx),
\]
or equivalently
\begin{equation}\label{e:goal}
F[\l,\ff](\bx)=0, \qquad \bx\in \T\times \stackrel{\circ}\supp(\varpi'),
\end{equation}
where
\begin{multline}\label{Fdefi}
F[\l,\ff](\bx):=\l \pa_x \ff(\bx)+(y +\ff(\bx))\pa_x\ff(\bx)\\
-\frac{1}{4\pi}\int_{\T\times\supp (\varpi')} \varpi'(\bar{y})\log\left[\cosh(y-\bar{y} + \ff(\bx)-\ff(\bar{\bx}))-\cos(x-\bar{x})\right] (\pa_x\ff(\bx)-\pa_x\ff(\bar{\bx}))\mbox{d}\bar{\bx},
\end{multline}

It is easy to check that  $\ff(\bx)=0$ is a trivial solution of \eqref{e:goal} whose velocity and vorticiy are given by $\bu(\bx,t)=(u_1(y),0)$ and $\omega(\bx,t)=-\pa_y u_1(y)$ where $-\pa_y u_1(y)=\varpi(y)$. Consequently, $f(\bx,t)=y$ is a solution of \eqref{mainequation} whose velocity is a degenerate shear flow. To prove our goal, we will need to show the existence of nontrivial solutions of the equation \eqref{mainequation}.
More specifically, we will prove the existence of traveling wave solutions of \eqref{mainequation} bifurcating from the trivial one $f(\bx,t)=y$.

\begin{remark} The solutions given  by $\omega(x,y+\ff(x+\lambda t,t)=\varpi(y)$  give rise, by expression \eqref{u} and an appropriate change of variables, to solutions of 2D Euler
\begin{align*}
\pa_t \bu +y \pa_x \bu +(\bu\cdot\nabla)\bu+\nabla p=0,\quad \text{in $\T\times \R$,}
\end{align*}
of the form $$\bu(\bx,t)=(u_1(x,y,t), u_2(x,y,t))=(\bar{u}_1(x+\lambda t,y), \bar{u}_2(x+\lambda t,y)).$$
\end{remark}

\subsection{The profile function $\varpi_{\ep,\kappa}$.}\label{profile}

To continue, we define the profile function $\varpi\in C_c^\infty(\R)$ that will be used in Section \ref{s:analysislinear} to solve the functional equation \eqref{e:goal}.

To start with let us construct first a profile $\varpi\in W^{1,\infty}(\R)\cap H^{\frac{3}{2}^-}(\R)$. This profile will be even and compactly supported. Let us consider the auxiliary function
\begin{align*}
\varphi(z)=\frac{1-z}{2},\quad  z\in [-1,1].
\end{align*}
Let $\e>0,$ we define $\varpi_\e(y)$ in the following way
\begin{align*}
\varpi_\e(y)=\left\{\begin{array}{cr} 0 & y>1+\e, \\
\e \varphi\left(\frac{y-1}{\e}\right) & 1-\e \leq y \leq 1+\e,\\
\e & 0<y<1-\e.
\end{array}\right.
\end{align*}
Notice that  $\varpi_\e(y)$ is extended as an even function for $y<0$.

Since we are interested in smooth solutions, we need to consider a smooth profile function $\varpi$. We will use a $\kappa$-regularization of $\varphi$, note that $\varphi'=-\frac{1}{2}$. The function $\varphi_\kappa$ will be defined by
$$\varphi_\kappa'(z)=-\frac{\psi'_\kappa(z)}{\int_{-1}^1\psi_{\kappa}'(\bar{z})\mbox{d}\bar{z}}, \quad  \text{for } |z|<1,$$
with
$$\psi'_\kappa(z)=\int_{-1}^1 \chi_{[-1+\kappa,1-\kappa]}(\bar{z})\Theta_\kappa(z-\bar{z})\mbox{d}\bar{z}=\int_{-1+\kappa}^{1-\kappa}\Theta_{\kappa}(z-\bar{z})\mbox{d}\bar{z},$$
where $\Theta$ is a mollifier, i.e a smooth, positive function, with $\supp(\Theta)\subset (-1,1)$ and $\int_{-1}^1\Theta(\bar{z})\mbox{d}\bar{z}=1$. Then, for each $\k>0$ we define $$\Theta_{\kappa}(z)=\frac{1}{\kappa}\Theta\left(\frac{z}{\kappa}\right).$$
Notice that, since $\Theta$ is smooth then $\psi'_{\kappa}$ is also smooth. Moreover, since $\supp(\Theta_{\kappa})\subset (-\kappa,\kappa)$ we have that $\psi'_{\kappa}( z )=0$ for $|z|>1$ and  it is easy to check that
$\psi_{\kappa}'(z)=1$ for $-1+2\kappa<z<1-2\kappa$.  Then, $\psi'_{\kappa}\to 1$ as $\kappa\to 0$. This convergence holds in $L^1([-1,1])$ and actually $\psi'_{\kappa}\to \chi_{[-1,1]}$ in $H^{\frac{3}{2}-}(\R)$. Furthermore, one can check that
\begin{align*}
\int_{-1}^1\left|\psi'_{\kappa}(\bar{z})-1\right|\mbox{d}\bar{z}=\int_{-1}^{-1+2\k}\left|\psi'_\kappa(\bar{z})-1\right|\mbox{d}\bar{z} + \int_{1-2\kappa}^1\left|\psi'_\kappa(\bar{z})-1\right|\mbox{d}\bar{z}\leq C \kappa,\\
\left|\int_{-1}^1\psi'_\kappa(\bar{z})\mbox{d}\bar{z}-2\right|=\left|\int_{-1}^1\left(\psi'_\kappa(\bar{z})-1\right)\mbox{d}\bar{z}\right|\leq \int_{-1}^1\left|\psi'_{\kappa}(\bar{z})-1\right|\mbox{d}\bar{z} \leq C\kappa,
\end{align*}
and then
\begin{align*}
\int_{-1}^1\left|\frac{\psi'_{\kappa}(\bar{z})}{\int_{-1}^1\psi'_\kappa(\bar{z})\mbox{d}\bar{z}}-\frac{1}{2}\right|\mbox{d}\bar{z}=\frac{1}{\left|2\int_{-1}^1\psi'_\kappa(\bar{z})\mbox{d}\bar{z} \right|}\int_{-1}^1\left|2\left(\psi'_{\kappa}(\bar{z})-1\right)+\left(2-\int_{-1}^1\psi'_\kappa(\bar{z})\mbox{d}\bar{z}\right)\right|\mbox{d}\bar{z}\leq C\kappa.
\end{align*}
We will take
\begin{align}\label{varphikappa}
\varphi_\kappa(z)=1+\int_{-1}^z\varphi'_\kappa(\bar{z})\mbox{d}\bar{z}=1-\frac{\int_{-1}^z \psi'_{\kappa}(\bar{z})\mbox{d}\bar{z}}{\int_{-1}^1\psi_{\kappa}'(\bar{z})\mbox{d}\bar{z}}.
\end{align}
We summarize the properties of this function in the following lemma.
\begin{lemma}\label{propiedadesvarphi}The function $\varphi_{\kappa}$ given by \eqref{varphikappa} satisfies
\begin{enumerate}
\item $\varphi_{\kappa}(-1)=1$, $\varphi_{\kappa}(1)=0$, $(\pa_z^{n}\varphi_{\kappa})(\pm 1)=0$, for $n=1,2,\ldots$
\item $\varphi_{\kappa}'(z)<0$, for $|z|<1$.
\item $\|\varphi'_{\kappa}+\frac{1}{2}\|_{L^1([-1,1])}\leq C\kappa$.
\end{enumerate}
\end{lemma}

The profile $\varpi$ that will be used in this manuscript is given by
\begin{align}\label{varpiepkappa}
\varpi_{\ep,\kappa}(y):=
\left\{\begin{array}{cr} 0 & y>1+\e, \\
\ep \varphi_{\kappa}\left(\frac{y-1}{\e}\right)\qquad & 1-\e \leq y \leq 1+\e, \\
\e & 0<y<1-\e.\end{array}\right.
\end{align}
Notice that  $\varpi_\e(y)$ is extended as an even function, i.e. for any $y<0$ we take $\varpi_{\ep,\kappa}(y)=\varpi_{\ep,\kappa}(-y)$.
Let us point out that the support of $\omega'_{\ep,\kappa}$ will be the domain
\begin{align*}
I_\ep=[-1-\ep,-1+\ep]\cup [1-\ep, 1+\ep].
\end{align*}

The rest of the paper consists in finding a nontrivial solution of \eqref{e:goal} with $\varpi\equiv\varpi_{\ep,\kappa}$ in \eqref{Fdefi} for parameters $\ep$ and $\kappa$ small enough. It will be done using the bifurcation theory through Crandall-Rabinowitz theorem \cite{CR}. For the completeness of the paper we recall this basic theorem and it will referred to as sometimes by C-R theorem.

%\newpage
\section{Bifurcation theory and Crandall-Rabinowitz}\label{s:bifurcation}
Before going into details, we shall first fix some notations that we will used later. For a linear mapping $\cL$ we will denote by $\cN(\cL)$ and $\cR(\cL)$  the kernel and range of $\cL$ respectively.   If $Y$ is a vector space and $S$ is a subspace, then $Y/ S$ denotes the quotient space.\medskip

Now, we intend to  give some formal explanations and general picture of the bifurcation theory. This brief discussion will be closed by stating the details of the C-R theorem.  Roughly speaking, the main objective  of this theory is to look for the solutions of the functional equation
$$
F[\l,\ff]=0,
$$
where $F:\R\times X\to Y$ is a smooth continuous function between Banach spaces  $X$ and $Y$. We assume in addition \mbox{that $\ff=0$} is a  trivial solution for any $\lambda\in\R$, that is, $F[\lambda, 0]=0$. Whether close to the  trivial solution  $(\l_\star, 0)$ one  can find a branch of nontrivial ones is the main aim of this theory. If this is the case  we say that there is a bifurcation at the point $(\l_\star, 0)$.
As the Implicit Function Theorem tells us, the first idea is to study  the linear operator  $\mathcal{L}_\l:= D_{\ff}F[\l, 0]:X\to Y$.
In principle, the involved Banach spaces $X$ and $Y$ are infinite-dimensional  and thus the bifurcation analysis  is in general complex. However, if the linearized operator around this point $(\l_\star, 0)$ generates a Fredholm type operator, then one can use the so-called Lyapunov-Schmidt reduction  in order to reduce the infinite-dimensional problem to a finite-dimensional one. Finally, for this last problem we just need some concrete transversal conditions so that the Implicit Function Theorem can be applied.

To sum up, this is the classical result proved by  Crandall and Rabinowitz which is a basic tool in the bifurcation theory and that will be used in this manuscript. Now, we recall here the statement of this theorem from \cite{CR} for expository purposes.

\begin{theorem}\label{th:CR} Let $X, Y$ be two Banach spaces, $V$ a neighborhood of $0$ in $X$ and let
$
F : \R \times V \to Y
$
with the following  properties:
\begin{enumerate}
\item $F [\lambda, 0] = 0$ for any $\lambda\in \R$.
\item The partial derivatives $D_\l F$, $D_{\ff}F$ and $D^2_{\l,\ff}F$ exist and are continuous.
\item There exists $\lambda_\star$ such that if $\mathcal{L}_\star= D_{\ff} F[\l_\star,0]$ then  $\cN(\mathcal{L}_\star)$ and $Y/\cR(\mathcal{L}_\star)$ are one-dimensional.
\item {\it Transversality assumption}: $D^2_{\l,\ff}F[\l_\star, 0]\hh_\star \not\in \cR(\mathcal{L}_\star)$, where
$$
\cN(\mathcal{L}_\star) = \text{span}\{\hh_\star\}.
$$
\end{enumerate}
If $Z$ is any complement of $\cN(\mathcal{L}_\star)$ in $X$, then there is a neighborhood $U$ of $(\l_\star,0)$ in $\R \times X$, an interval $(-\sigma_0,\sigma_0)$, and continuous functions $\varphi: (-\sigma_0,\sigma_0) \to \R$, $\psi: (-\sigma_0,\sigma_0) \to Z$ such that $\varphi(0) = 0$, $\psi(0) = 0$ and
\begin{align*}
F^{-1}(0)\cap U=&\Big\{\big(\lambda_\star+\varphi(\sigma), \sigma h_\star+\sigma \psi(\sigma)\big)\,;\,\vert \sigma\vert<\sigma_0\Big\} \cup \Big\{(\lambda,0)\,;\, (\lambda,0)\in U\Big\}.
\end{align*}

\end{theorem}

The bulk of the paper  consists in checking all the assumptions of Theorem \ref{th:CR}. This will be done in details in
the next sections.

\section{Functional setting and regularity}\label{s:regularity}

Note that our functional can be written as
\begin{equation}\label{e:functional}
F[\l,\ff](\bx)=  \lambda \pa_x \ff(\bx) +(y+\ff(\bx))\pa_x \ff (\bx) -\frac{1}{4\pi}\int_{D_{\ep}(y)}\varpi'(y-\bar{y})K[\ff](\bx,\bar{\bx})\D_{\bar{\bx}}[\pa_x \ff](\bx) \mbox{d}\bar{\bx},
\end{equation}
where the kernel is given by the following expression
\[
K[g](\bx,\bar{\bx}):=\log\left[\cosh(\bar{y}+\D_{\bar{\bx}}[g](\bx))-\cos(\bar{x})\right].
\]
Recall that the finite difference is given by
\[
\D_{\bar{\bx}}[g](\bx)=g(\bx)-g(\bx-\bar{\bx}),
\]
and the domain of integration is just
\[
D_{\ep}(y)=\T\times I_{\ep}(y),\quad \text{with} \quad I_\ep(y)= y+I_\ep.
\]
For simplicity, we will use the notation $D_\ep=D_{\ep}(0)=\T\times I_\ep.$

\begin{remark}
In all this section we will assume that $\varpi\equiv\varpi_{\e,\k}\in C^\infty(\R)$ as in Section \ref{profile}.
\end{remark}

\subsection{The functional setting}

In order to apply C-R theorem we need first to fix the function spaces.  We should look for Banach spaces $X$ and $Y$ such that $F : \R \times X \rightarrow Y$ is well-defined and satisfies the required assumptions.

Our first step is to define the spaces we will work with in order to apply the Crandall-Rabinowitz theorem.
The spaces $X$ and $Y$ will be given by
\begin{align}\label{spaceX}
X(D_{\ep}):=\left\{g\in H^{4,3}(D_{\ep}): g \text{ is even in $x$  with $\int_{-\pi}^{\pi} g(x,y)\dx=0$}\right\},
\end{align}
and
\begin{align}\label{sapceY}
Y(D_{\ep}):=\left\{g\in H^3(D_{\ep}): g  \text{ is odd in $x$}\right\}.
\end{align}
Here $H^{4,3}(D_{\ep})$ is the Sobolev-Leibnitz space of $2\pi$-periocic functions in the $x$-variable with norm
\begin{align*}
\|g\|_{H^{4,3}(D_{\ep})}:= \sum_{i=0}^4 \sum_{j=0}^{3-i}\|\pa_x^i \pa_y^j g\|_{L^2(D_{\ep})}.
\end{align*}
After that the main purpose will be to prove next lemma.
\begin{lemma} For all $0<\ep<1,$ there exist $0<\d(\ep)\ll 1$ small enough such that
\begin{align*}
F\,:\, \R\times \mathbb{B}_{\d}(X(D_{\ep})) & \to Y(D_{\ep}),\\
\quad (\l,\ff) & \to F[\l,\ff]
\end{align*}
where
\begin{align*}
\mathbb{B}_{\d}(X(D_{\ep})) :=\{ g\in X(D_{\ep})\,:\, \|g\|_{H^{4,3}(D_{\ep})}< \d\}.
\end{align*}
\end{lemma}
\begin{proof}
Let $\ff\in X(D_{\ep}),$ the evenness in the $x$-direction translates into the oddness of $F[\l,\ff]$ just by definition of the functional \eqref{e:functional}. Moreover, the non-integral part of the RHS of \eqref{e:functional} maps $H^{4,3}(D_\ep)$ into $H^{3}(D_\ep)$ trivially as a direct consequence  of the algebraic property of the Sobolev space $H^3(D_\ep)$. Therefore,  we will focus our attention into the integral part of \eqref{e:functional}.
In addition, as the profile function $\varpi$ is smooth, our functional is well-defined  if  each of the next terms are bounded as follow:
\begin{align}\label{resumen1lema}
%\left\Vert \int_{D_{\ep}(y)}K[\ff](\bx,\bar{\bx})\D_{\bar{\bx}}[\pa_x \ff](\bx) \mbox{d}\bar{\bx}\right\Vert_{L^2(D_{\ep})}^2 &\leq C(\d,\|g\|_{H^{4,3}(D_{\ep})})\\
\left\Vert \int_{D_{\ep}(y)}\p_{\bx}^i K[\ff](\bx,\bar{\bx})\D_{\bar{\bx}}[\p_{\bx}^{3-i}\pa_x \ff](\bx) \mbox{d}\bar{\bx}\right\Vert_{L^2(D_{\ep})}^2&\leq C(\ep,\|\ff\|_{H^{4,3}(D_{\ep})}) \qquad (0\leq i\leq 3).
\end{align}
Note that by definition of $\D_{\bar{\bx}}[g](\bx)$ we have
\begin{equation}\label{L2difference}
\int_{D_{\ep}}\left(\int_{D_{\ep}(y)}|\D_{\bar{\bx}}[g](\bx)|^2 \mbox{d}\bar{\bx}\right)\mbox{d}\bx\leq 4|D_{\ep}| \|g\|_{L^{2}(D_{\ep})}^2.
\end{equation}
For the case $i=0$ or $i=1$, as $\ff\in H^{4,3}(D_{\ep})$ implies $\p_{\bx}^2\pa_x\ff,\p_{\bx}^3\pa_x\ff \in L^2(D_{\ep}),$ we  obtain
\begin{align*}
\left\Vert \int_{D_{\ep}(y)}K[\ff](\bx,\bar{\bx})\D_{\bar{\bx}}[\p_{\bx}^3\pa_x \ff](\bx)\mbox{d}\bar{\bx} \right\Vert_{L^2(D_{\ep})}^2 \leq 4|D_{\ep}|^2 \|\ff\|_{H^{4,3}(D_{\ep})}^2 \sup_{\bx\in D_{\ep}}\left( \sup_{\bar{\bx}\in D_{\ep}(y)}\left|K[\ff](\bx,\bar{\bx})\right|^2 \right),
\end{align*}
and
\begin{multline*}
\left\Vert \int_{D_{\ep}(y)}\p_{\bx} K[\ff](\bx,\bar{\bx})\D_{\bar{\bx}}[\p_{\bx}^2\pa_x \ff](\bx)\mbox{d}\bar{\bx} \right\Vert_{L^2(D_{\ep})}^2 \\
  \leq 4|D_{\ep}| \|\ff\|_{H^{4,3}(D_{\ep})}^2  \sup_{\bx\in D_{\ep}}\left( \int_{D_{\ep}(y)}\left|\p_{\bx} K[\ff](\bx,\bar{\bx})\right|^2\mbox{d}\bar{\bx} \right).
\end{multline*}
For the case $i=2$, as $\ff\in H^{4,3}(D_{\ep})$ we get $\p_{\bx}\pa_x\ff \in H^2(D_{\ep})\subset C^{\g}(D_{\ep})$ for $0<\g<1,$ which give us
\begin{multline*}
\left\Vert \int_{D_{\ep}(y)}\p_{\bx}^2 K[\ff](\bx,\bar{\bx}) \D_{\bar{\bx}}[\p_{\bx}\pa_x \ff](\bx)\mbox{d}\bar{\bx} \right\Vert_{L^2(D_{\ep})}^2 \\
\hspace*{-0.5 cm} \leq |D_{\ep}| \|\ff\|_{H^{4,3}(D_{\ep})}^2 \int_{D_{\ep}} \left(\int_{D_\ep(y)}\left|\p_{\bx}^2 K[\ff](\bx,\bar{\bx}) \right|^2 |\bar{\bx}|^{2\g}\mbox{d}\bar{\bx} \right)\mbox{d}\bx.
\end{multline*}
For the last case $i=3$, as $\pa_x\ff\in H^3(D_{\ep})\subset C^1(D_{\ep})$, we obtain
\begin{multline*}
\left\Vert \int_{D_{\ep}(y)}\p_{\bx}^3  K[\ff](\bx,\bar{\bx}) \D_{\bar{\bx}}[\pa_x \ff](\bx)\mbox{d}\bar{\bx} \right\Vert_{L^2(D_{\ep})}^2 \\
\leq |D_{\ep}|  \|\ff\|_{H^{4,3}(D_{\ep})}^2 \int_{D_{\ep}}\left(\int_{D_\ep(y)}\left|\p_{\bx}^3 K[\ff](\bx,\bar{\bx}) \right|^2 |\bar{\bx}|^{2}\mbox{d}\bar{\bx} \right)\mbox{d}\bx.
\end{multline*}
Finally, for the sake of brevity and clarity we will refer to Lemma \ref{l:DK[f]}, where the last term of each of the above expressions is bounded.
\end{proof}

\subsection{Hypothesis 1 and 2 }

On one hand, the hypothesis 1 in the C-R theorem is trivial to check. On the other hand, the hypothesis 2 has to do with the regularity of the functional $F[\lambda,\ff]$ with respect to $\ff$ and $\lambda$.

We present the regularity of the functional in the following proposition.
\begin{proposition} For all $0<\ep<1,$ there exist $0<\d(\e)\ll 1$ small enough such that the following holds true:
\begin{enumerate}
	 \item The functional $F:\R \times \mathbb{B}_{\d}(X(D_{\ep}))\rightarrow Y(D_{\ep})$ is of class $C^1.$
	 \item The partial derivative $D^2_{\l,\ff} F:\R \times \mathbb{B}_{\d}(X(D_{\ep}))\rightarrow Y(D_{\ep})$ is continuous.
\end{enumerate}
\end{proposition}
\begin{proof}
The proof of this proposition is rather standard. We include the main details here for sake of completeness.

The continuity of the functional is trivial for any derivative involving $D_\l$. Consequently, the result reduces to check that
$D_{\ff} F:\R\times \mathbb{B}_{\d}(X(D_{\ep}))\rightarrow Y(D_{\ep})$ is continuous at the origin.
This will be done by showing first the existence of the G\^ateaux derivative and second its continuity in the strong topology.
A refined analysis concerning its connection with Fr\'{e}chet derivative will be developed in the next section.
The G\^ateaux differential of $F[\l,\cdot]$ at $\ff\in \mathbb{B}_{\d}(X(D_{\ep}))$ in the direction $\hh\in X(D_{\ep})$
is defined as
\[
D_\ff F[\l,\ff]\hh:=\frac{\mbox{d}}{{\mbox{d}\tau}}F[\l,\ff+\tau \hh]|_{\tau=0}=\lim_{\tau\to0}\frac{F[\l,\ff+\tau\hh]-F[\l,\ff]}{\tau}.
\]
A formal straightforward computation show that this derivative is given by
\begin{align}\label{def:dF}
D_\ff F[\l,\ff]\hh(\bx)=&\l \pa_x \hh(\bx)+(y +\ff(\bx))\pa_x\hh(\bx)+ \hh(\bx)\pa_x \ff(\bx)\\
&-\frac{1}{4\pi}\int_{D_{\ep}(y)}\varpi'(y-\bar{y})K[\ff](\bx,\bar{\bx})\D_{\bar{\bx}}[\pa_x \hh](\bx)\mbox{d}\bar{\bx} \nonumber \\
&-\frac{1}{4\pi}\int_{D_{\ep}(y)}\varpi'(y-\bar{y})(\pa_x K)[\ff](\bx,\bar{\bx})\D_{\bar{\bx}}[\hh](\bx)\mbox{d}\bar{\bx}. \nonumber
\end{align}
%\begin{multline}
%D\cF_{\l}(\ff)[h](\bx)=\l \pa_x \hh(\bx)+(y +\ff(\bx))\pa_x\hh(\bx)+ \hh(\bx)\pa_x \ff(\bx)\\
%-\frac{1}{4\pi}\int_{D_{\ep}(y)}\varpi'(y-\bar{y})K[\ff](\bx,\bar{\bx})\D_{\bar{\bx}}[\pa_x \hh](\bx)\dz \nonumber \\
%-\frac{1}{4\pi}\int_{D_{\ep}(y)}\varpi'(y-\bar{y})(\pa_x K)[\ff](\bx,\bar{\bx})\D_{\bar{\bx}}[\hh](\bx)\dz. \nonumber
%\end{multline}
To prove this rigorously, we need to get
\[
\lim_{\tau\to 0} \left\Vert \frac{F[\l,\ff+\tau\hh]-F[\l,\ff]}{\tau} -D_\ff F[\l,\ff]\hh\right\Vert_{H^3(D_{\ep})}=0.
\]
By virtue of \eqref{def:dF} it is enough to prove that
\begin{align}
\lim_{\tau\to 0} \left\Vert \int_{D_{\ep}(y)} \varpi'(y-\bar{y})\cK[\ff+\tau \hh,\ff](\bx,\bar{\bx})\D_{\bar{\bx}}[\pa_x \hh](\bx)\mbox{d}\bar{\bx} \right\Vert_{H^3(D_{\ep})}&=0,\label{limit2} \\
\lim_{\tau\to 0} \left\Vert \int_{D_{\ep}(y)} \varpi'(y-\bar{y})\left(\frac{\cK[\ff+\tau \hh,\ff](\bx,\bar{\bx})}{\tau}-\Psi_1[\ff](\bx,\bar{\bx})\D_{\bar{\bx}}[\hh](\bx)\right)\D_{\bar{\bx}}[\pa_x \ff](\bx)\mbox{d}\bar{\bx} \right\Vert_{H^3(D_{\ep})}&=0,\label{limit1}
\end{align}
where the new kernel is
\begin{equation}\label{bikernel}
\cK[\ff',\ff''](\bx,\bar{\bx}):=K[\ff'](\bx,\bar{\bx})-K[\ff''](\bx,\bar{\bx})=\log\left[\frac{\cosh\left(\bar{y}+\D_{\bar{\bx}}[\ff'](\bx)\right)-\cos(\bar{x})}{\cosh\left(\bar{y}+\D_{\bar{\bx}}[\ff''](\bx)\right)-\cos(\bar{x})}\right],
\end{equation}
and the auxiliary functions $\Psi_1,\Psi_2$ are given respectively by
\begin{align*}
\Psi_1[g](\bx,\bar{\bx}):=\frac{\sinh(\bar{y}+\D_{\bar{\bx}}[g](\bx))}{\cosh(\bar{y}+\D_{\bar{\bx}}[g](\bx))-\cos(\bar{x})},\\
\Psi_2[g](\bx,\bar{\bx}):=\frac{\cosh(\bar{y}+\D_{\bar{\bx}}[g](\bx))}{\cosh(\bar{y}+\D_{\bar{\bx}}[g](\bx))-\cos(\bar{x})}.
\end{align*}

In order to do the manuscript more readable, we redirect the reader to the next subsection for precise proofs of each of the above limits \eqref{limit2}, \eqref{limit1}. As we can imagine the computations are very long and tedious but share lot of similarities.

\subsubsection{Computation of \eqref{limit2}}
As the profile $\varpi$ is a smooth and compactly supported function, condition $\eqref{limit2}$ reduces to check
\[
\lim_{\tau\to 0} \left\Vert \int_{D_{\ep}(y)}\cK[\ff+\tau \hh,\ff](\bx,\bar{\bx})\D_{\bar{\bx}}[\pa_x \hh](\bx)\mbox{d}\bar{\bx} \right\Vert_{H^3(D_{\ep})}=0.
\]
\textbf{Remark:} There are not boundary terms to handle due to the  support of $\varpi$ and definition of $D_{\e}(y).$
Consequently, the above reduces to prove that each of the terms of $H^3(D_{\ep})-$norm tends to zero.
In the same spirit of \eqref{resumen1lema}, the proof of \eqref{limit2} reduces to check  that
\begin{equation}\label{resumen2lema}
\left\Vert \int_{D_{\ep}(y)}\p_{\bx}^i \cK[\ff+\tau \hh,\ff](\bx,\bar{\bx})\D_{\bar{\bx}}[\p_{\bx}^{3-i}\pa_x \hh](\bx)\mbox{d}\bar{\bx} \right\Vert_{L^2(D_{\ep})}^2 \leq C(\ep,\|\ff\|_{H^{4,3}(D_{\ep})})\tau^2 \qquad (0\leq i \leq 3).
\end{equation}
To do that, we will use repeatedly inequality \eqref{L2difference} and the fact that, without lost of generality, the direction $\hh\in H^{4,3}(D_\ep)$ can be taken with norm $\|\hh\|_{H^{4,3}(D_{\ep})}=1$.

For the case $i=0$ or $i=1$, as $\hh\in H^{4,3}(D_{\ep})$ implies $\p_{\bx}^3\pa_x\hh,\p_{\bx}^2\pa_x\hh \in L^2(D_{\ep})$, we  obtain
\begin{align*}
\left\Vert \int_{D_{\ep}(y)}\cK[\ff+\tau \hh,\ff](\bx,\bar{\bx})\D_{\bar{\bx}}[\p_{\bx}^3\pa_x \hh](\bx)\mbox{d}\bar{\bx} \right\Vert_{L^2(D_{\ep})}^2 \leq 4|D_{\ep}|^2 \sup_{\bx\in D_{\ep}}\left( \sup_{\bar{\bx}\in D_{\ep}(y)}\left|\cK[\ff+\tau \hh,\ff](\bx,\bar{\bx})\right|^2 \right),
\end{align*}
and
\begin{multline*}
\left\Vert \int_{D_{\ep}(y)}\p_{\bx} \cK[\ff+\tau \hh,\ff](\bx,\bar{\bx})\D_{\bar{\bx}}[\p_{\bx}^2\pa_x \hh](\bx)\mbox{d}\bar{\bx} \right\Vert_{L^2(D_{\ep})}^2\\
  \leq 4|D_{\ep}|  \sup_{\bx\in D_{\ep}}\left( \int_{D_{\ep}(y)}\left|\p_{\bx} \cK[\ff+\tau \hh,\ff](\bx,\bar{\bx})\right|^2\mbox{d}\bar{\bx} \right).
\end{multline*}
For the case $i=2$, as $\hh\in H^{4,3}(D_{\ep})$ we get $\p_{\bx}\pa_x\hh \in H^2(D_{\ep})\subset C^{\g}(D_{\ep})$ for $0<\g<1,$ which give us
\begin{multline*}
\left\Vert \int_{D_{\ep}(y)}\p_{\bx}^2 \cK[\ff+\tau \hh,\ff](\bx,\bar{\bx}) \D_{\bar{\bx}}[\p_{\bx}\pa_x \hh](\bx)\mbox{d}\bar{\bx} \right\Vert_{L^2(D_{\ep})}^2\\
 \leq |D_{\ep}| \int_{D_{\ep}}\left(\int_{D_\ep(y)}\left|\p_{\bx}^2 \cK[\ff+\tau \hh,\ff](\bx,\bar{\bx}) \right|^2 |\bar{\bx}|^{2\g}\mbox{d}\bar{\bx} \right)\mbox{d}\bx.
\end{multline*}
For the last case $i=3$, as $\pa_x\hh\in H^3(D_{\ep})\subset C^1(D_{\ep})$, we obtain
%Finally, as $|\D_{\bar{\bx}}[\pa_x \hh](\bx)|\leq \|\n \pa_x\hh\|_{L^{\infty}(D_{\ep})}|\bar{\bx}|$ and $L^{\infty}(D_{\ep})\hookrightarrow H^2(D_{\ep})$ we obtain
\begin{multline*}
\left\Vert \int_{D_{\ep}(y)}\p_{\bx}^3  \cK[\ff+\tau \hh,\ff](\bx,\bar{\bx}) \D_{\bar{\bx}}[\pa_x \hh](\bx)\mbox{d}\bar{\bx} \right\Vert_{L^2(D_{\ep})}^2\\
\leq |D_{\ep}| \int_{D_{\ep}}\left(\int_{D_\ep(y)}\left|\p_{\bx}^3 \cK[\ff+\tau \hh,\ff](\bx,\bar{\bx}) \right|^2 |\bar{\bx}|^{2}\mbox{d}\bar{\bx} \right)\mbox{d}\bx.
\end{multline*}
As $\ff\in \mathbb{B}_{\d}(X(D_{\ep}))$ and our goal is to compute the limit as $\tau\to 0,$ we can assume without loss of generality that $0<\tau\ll 1$ is small enough such that $\ff+\tau\hh\in \mathbb{B}_{\d}(X(D_{\ep}))$. Consequently, the last term of each of the above expressions can be handle applying auxiliary Lemma \ref{l:DK[f',f'']} with
\begin{align*}
\ff'&:=\ff+\tau\hh,\\
\ff''&:=\ff.
\end{align*}
Finally, as by hypothesis $\|\hh\|_{H^3(D_\ep)}=1$,  we  get
\[
\left\Vert \int_{D_{\ep}(y)}\cK[\ff+\tau \hh,\ff](\bx,\bar{\bx})\D_{\bar{\bx}}[\pa_x \hh](\bx)\mbox{d}\bar{\bx} \right\Vert_{H^3(D_{\ep})} \leq C(\e,\|\ff\|_{H^3(D_{\ep})})\tau^2,
\]
and taking the limit as $\tau\to 0$ we have proved \eqref{limit2} .

\subsubsection{Computation of \eqref{limit1}}
Similarly, since the profile $\varpi$ is a smooth and compactly supported function, condition $\eqref{limit1}$ reduces to check
\[
\lim_{\tau\to 0} \left\Vert \int_{D_{\ep}(y)}\left(\frac{\cK[\ff+\tau \hh,\ff](\bx,\bar{\bx})}{\tau}-\Psi_1[\ff](\bx,\bar{\bx})\D_{\bar{\bx}}[\hh](\bx)\right)\D_{\bar{\bx}}[\pa_x \ff](\bx)\mbox{d}\bar{\bx} \right\Vert_{H^3(D_{\ep})}=0.
\]
As before, there are not boundary terms to handle due to the  support of $\varpi$ and definition $D_{\e}(y).$
Consequently, the above reduces to prove that each of the terms of $H^3(D_{\ep})-$norm tends to zero.
In the same spirit of \eqref{resumen1lema} or \eqref{limit2}, the proof of \eqref{limit1} reduces to check  that
\begin{equation}\label{resumen3lema}
\left\Vert \int_{D_{\ep}(y)}\p_{\bx}^i\left\lbrace \frac{\cK[\ff+\tau \hh,\ff](\bx,\bar{\bx})}{\tau}-\Psi_1[\ff](\bx,\bar{\bx})\D_{\bar{\bx}}[\hh](\bx)\right\rbrace\D_{\bar{\bx}}[\p_{\bx}^{3-i}\pa_x \ff](\bx)\mbox{d}\bar{\bx} \right\Vert_{L^2(D_{\ep})}^2 \hspace*{-0.5 cm}\leq C(\e,\|\ff\|_{H^{4,3}(D_{\ep})})\tau^2,
\end{equation}
for $0\leq i \leq 3$ and $\ff \in\mathbb{B}_{\d}(H^{4,3}(D_{\ep}))$ with $\d(\e)$ small enough and $\hh\in H^{4,3}(D_{\ep})$ with $\|\hh\|_{H^{4,3}(D_{\ep})}=1.$
Now proceeding as before, we use repeatedly inequality \eqref{L2difference} and the fact that $\|\hh\|_{H^{4,3}(D_{\ep})}=1$.

For the case $i=0$ or $i=1$, as $\ff\in H^{4,3}(D_{\ep})$ implies $\pa_x\ff,\p_{\bx}^2\pa_x\ff \in L^2(D_{\ep}),$ we  obtain
\begin{align*}
&\left\Vert \int_{D_{\ep}(y)}\left(\frac{\cK[\ff+\tau \hh,\ff](\bx,\bar{\bx})}{\tau}-\Psi_1[\ff](\bx,\bar{\bx})\D_{\bar{\bx}}[\hh](\bx)\right)\D_{\bar{\bx}}[\pa_x \ff](\bx)\mbox{d}\bar{\bx} \right\Vert_{L^2(D_{\ep})}^2\\
&\qquad \qquad \leq 4|D_{\ep}|^2 \|\ff\|_{H^{4,3}(D_{\ep})}^2 \sup_{\bx\in D_{\ep}}\left( \sup_{\bar{\bx}\in D_{\ep}(y)}\left|\frac{\cK[\ff+\tau \hh,\ff](\bx,\bar{\bx})}{\tau}-\Psi_1[\ff](\bx,\bar{\bx})\D_{\bar{\bx}}[\hh](\bx)\right|^2 \right),
\end{align*}
and
\begin{align*}
&\left\Vert \int_{D_{\ep}(y)}\p_{\bx}\left\lbrace \frac{\cK[\ff+\tau \hh,\ff](\bx,\bar{\bx})}{\tau}-\Psi_1[\ff](\bx,\bar{\bx})\D_{\bar{\bx}}[\hh](\bx)\right\rbrace\D_{\bar{\bx}}[\p_{\bx}^2\pa_x \ff](\bx)\mbox{d}\bar{\bx} \right\Vert_{L^2(D_{\ep})}^2\\
&\qquad \qquad  \leq 4|D_{\ep}| \|\ff\|_{H^{4,3}(D_{\ep})}^2 \sup_{\bx\in D_{\ep}}\left( \int_{D_{\ep}(y)}\left|\p_{\bx}\left\lbrace\frac{\cK[\ff+\tau \hh,\ff](\bx,\bar{\bx})}{\tau}-\Psi_1[\ff](\bx,\bar{\bx})\D_{\bar{\bx}}[\hh](\bx)\right\rbrace\right|^2\mbox{d}\bar{\bx} \right).
\end{align*}
For the case $i=2$, as $\ff\in H^{4,3}(D_{\ep})$ we get $\p_{\bx}\pa_x\ff \in H^2(D_{\ep})\subset C^{\g}(D_{\ep})$ for $0<\g<1,$ which give us
\begin{align*}
&\left\Vert \int_{D_{\ep}(y)}\p_{\bx}^2\left\lbrace \frac{\cK[\ff+\tau \hh,\ff](\bx,\bar{\bx})}{\tau}-\Psi_1[\ff](\bx,\bar{\bx})\D_{\bar{\bx}}[\hh](\bx)\right\rbrace\D_{\bar{\bx}}[\p_{\bx}\pa_x \ff](\bx)\mbox{d}\bar{\bx} \right\Vert_{L^2(D_{\ep})}^2\\
&\qquad  \leq |D_{\ep}|\|\ff\|^2_{H^{4,3}(D_{\ep})} \int_{D_{\ep}}\left(\int_{D_\ep(y)}\left|\p_{\bx}^2\left\lbrace \frac{\cK[\ff+\tau \hh,\ff](\bx,\bar{\bx})}{\tau}-\Psi_1[\ff](\bx,\bar{\bx})\D_{\bar{\bx}}[\hh](\bx)\right\rbrace\right|^2 |\bar{\bx}|^{2\g}\mbox{d}\bar{\bx} \right)\mbox{d}\bx.
\end{align*}
For the last case $i=3$, as $\pa_x\ff\in H^3(D_{\ep})\subset C^1(D_{\ep})$, we obtain
%Finally, as $|\D_{\bar{\bx}}[\pa_x \ff](\bx)|\leq \|\n \pa_x\ff\|_{L^{\infty}(D_{\ep})}|\bar{\bx}|$ and $L^{\infty}(D_{\ep})\hookrightarrow H^2(D_{\ep})$ we obtain
\begin{align*}
&\left\Vert \int_{D_{\ep}(y)}\p_{\bx}^3\left\lbrace \frac{\cK[\ff+\tau \hh,\ff](\bx,\bar{\bx})}{\tau}-\Psi_1[\ff](\bx,\bar{\bx})\D_{\bar{\bx}}[\hh](\bx)\right\rbrace\D_{\bar{\bx}}[\pa_x \ff](\bx)\mbox{d}\bar{\bx} \right\Vert_{L^2(D_{\ep})}^2\\
&\qquad  \leq |D_{\ep}|\|\ff\|^2_{H^{4,3}(D_{\ep})} \int_{D_{\ep}}\left(\int_{D_\ep(y)}\left|\p_{\bx}^3\left\lbrace \frac{\cK[\ff+\tau \hh,\ff](\bx,\bar{\bx})}{\tau}-\Psi_1[\ff](\bx,\bar{\bx})\D_{\bar{\bx}}[\hh](\bx)\right\rbrace\right|^2 |\bar{\bx}|^{2}\mbox{d}\bar{\bx} \right)\mbox{d}\bx.
\end{align*}
For the sake of clarity, the last term of each of the above expressions can be handle using auxiliary Lemma \ref{l:K/TAU-PSI}.
After that, collecting all we get
\[
 \left\Vert \int_{D_{\ep}(y)}\left(\frac{\cK[\ff+\tau \hh,\ff](\bx,\bar{\bx})}{\tau}-\Psi_1[\ff](\bx,\bar{\bx})\D_{\bar{\bx}}[\hh](\bx)\right)\D_{\bar{\bx}}[\pa_x \ff](\bx)\mbox{d}\bar{\bx} \right\Vert_{H^3(D_{\ep})}^2\leq C(\e,\|\ff\|_{H^{4,3}(D_{\ep})})\tau^2,
\]
and taking the limit as $\tau\to 0$ we have proved \eqref{limit1}.\\

This shows the existence of G\^ateaux derivative and now we intend to prove the continuity of the map  $\ff\to D_\ff F[\l,\ff]$
from $X(D_{\ep})$ to the space $\mathsf{L}(X(D_{\ep}),Y(D_{\ep}))$ of all bounded linear operators from $X(D_{\ep})$ to $Y(D_{\ep})$.  This is a consequence of the following estimate:
\begin{equation}\label{continuity}
\|D_\ff F[\l,\ff']\hh-D_\ff  F[\l,\ff'']\hh\|_{H^{3}(D_{\ep})}\lesssim \|\ff'-\ff''\|_{H^{4,3}(D_{\ep})},
\end{equation}
for any pair $\ff',\ff'' \in \mathbb{B}_{\d}(X(D_{\ep}))$ and $h\in X(D_{\ep}).$
As the profile function $\varpi$ is smooth it is easy to check that condition \eqref{continuity} holds if and only if the following bounds are satisfied
\begin{align}
\left\Vert \int_{D_{\ep}(y)} \cK[\ff',\ff''](\bx,\bar{\bx})\D_{\bar{\bx}}[\pa_x \hh](\bx) \mbox{d}\bar{\bx} \right\Vert_{H^3(D_{\ep})} &\lesssim \|\ff'-\ff''\|_{H^{4,3}(D_{\ep})}, \label{Frechet1}\\
\left\Vert \int_{D_{\ep}(y)} (\pa_xK[\ff']-\pa_xK[\ff''])(\bx,\bar{\bx})\D_{\bar{\bx}}[ \hh](\bx) \mbox{d}\bar{\bx} \right\Vert_{H^3(D_{\ep})} &\lesssim \|\ff'-\ff''\|_{H^{4,3}(D_{\ep})}. \label{Frechet2}
\end{align}
Note that \eqref{Frechet1} follows by direct application of Lemma \ref{l:DK[f',f'']}. In order to obtain \eqref{Frechet2}, we just need to note that adding and subtracting appropriate terms we have
\[
(\pa_xK[\ff']-\pa_xK[\ff''])(\bx,\bar{\bx})=\Psi_1[\ff'](\bx,\bar{\bx}) \D_{\bar{\bx}}[\pa_x(\ff'-\ff'')](\bx) +(\Psi_1[\ff']-\Psi_1[\ff''])(\bx,\bar{\bx})\D_{\bar{\bx}}[\pa_x\ff''](\bx).
\]
The first term of the above expression is trivially bounded and for the other one we just need to use auxiliary Lemma \ref{l:Psi_1[f']-Psi_1[f'']}. Consequently, we have obtained that the G\^ateaux derivatives are continuous
with respect to the strong topology and hence they are in fact Fr\'echet derivatives. Therefore, we can conclude that the Fr\'{e}chet derivative exists and coincides with the G\^ateaux derivative. See \cite{Gateaux-Frechet} for more details.
\end{proof}

\section{Analysis of the linear part}\label{s:analysislinear}

Hypothesis 3 and 4 in the C-R theorem have to do with the linear part of equation \eqref{e:goal}.

Recall that the linearization of \eqref{e:functional} around $\ff=0$, thanks to the expression \eqref{def:dF}, is given by
\[
\cL[\l]\hh:=D_\ff F[\l,0]\hh.
\]
That is,
\begin{align}\label{e:Linear}
\cL[\l]\hh(\bx)= \, &(\l+y) \pa_x \hh(\bx)\\
&-\frac{1}{4\pi}\int_{\T\times I_\ep}\varpi_{\e,\kappa}'(\bar{y})\log\left[\cosh(\bar{y}-\bar{y})-\cos(\bar{x}-\bar{x})\right](\pa_x\hh(\bx)-\pa_x\hh(\bar{\bx}))\mbox{d}\bar{\bx}.\nonumber
\end{align}

In order to study both the kernel and image of $\cL[\lambda]$ we will introduce some modification which allows us to realize an asymptotic analysis on $\ep$.

\subsection{Decomposition of the linear operator}\label{s:simplifications}

In first place, we define a primitive function $\Omega_{\e,\kappa}$ of the profile function $\varpi_{\e,\kappa}$
by
\begin{align*}
\Omega_{\ep,\kappa}(y):=\int_{0}^{y} \varpi_{\ep,\kappa}(\bar{y})\mbox{d}\bar{y}.
\end{align*}
Then
\begin{equation}\label{e:simplification0}
\frac{1}{4\pi}\int_{\T\times I_\ep}\varpi_{\e,\k}'(\bar{y})\log\left[\cosh(y-\bar{y})-\cos(x-\bar{x})\right]\mbox{d}\bar{\bx}=\Omega_{\e,\k}(y).
\end{equation}
\begin{proof}[Proof of \eqref{e:simplification0}]
Since
\[
\frac{1}{4\pi}\int_{\T}\log\left[\cosh(\bar{y})-\cos(\bar{x})\right]\mbox{d}\bar{x}=\frac{1}{2}\left(|\bar{y}|-\log 2 \right),
\]
we have that
\begin{align}\label{last}
\frac{1}{4\pi}&\int_{\T\times I_\ep }\varpi_{\e,\k}'(\bar{y})\log\left[\cosh(y-\bar{y})-\cos(x-\bar{x})\right]\mbox{d}\bar{\bx}=\frac{1}{2}\int_{\R}|y-\bar{y}|\varpi_{\e,\k}'(\bar{y}) \mbox{d}\bar{y}\\
&=\frac{1}{2}\int_{\R}\sign(y-\bar{y})\varpi_{\ep,\kappa}(\bar{y})d\bar{y}=\frac{1}{2}\int_{\R}\sign(y-\bar{y})\Omega_{\ep,\kappa}'(\bar{y})\mbox{d}\bar{y}\nonumber\\
&=\Omega_{\ep,\kappa}(y)
+\lim_{R\to 0}\left( \sign(y-R)\Omega_{\ep,\kappa}(R)-\sign(y+R)\Omega_{\ep,\kappa}(-R)\right). \nonumber
\end{align}
Since $\varpi_{\ep,\kappa}$ is even, their primitive $\Omega_{\ep,\kappa}$ is odd. Thus the limit in the last line of \eqref{last} is zero.
\end{proof}

Using \eqref{e:simplification0} one learns that

\begin{align}\label{e:Linear2}
\cL[\l]\hh(\bx)=\,&(\l+y-\Omega_{\ep,\kappa}(y)) \pa_x \hh(\bx)\\
&+\frac{1}{4\pi}\int_{\T\times I_\ep}\varpi_{\e,\kappa}'(\bar{y})\log\left[\cosh(\bar{y}-\bar{y})-\cos(\bar{x}-\bar{x})\right]\pa_x\hh(\bar{\bx}))\mbox{d}\bar{\bx}. \nonumber
\end{align}

Now, recalling that $\hh\in X(D_\ep)$, we can use the expansion
\begin{equation}\label{e:hserie}
\hh(\bx)=\sum_{n=1}^{\infty}\hh_n(y)\cos(n x),
\end{equation}
with
\begin{align}
\sum_{n=1}^\infty\sum_{j=0}^3\sum_{k=0}^{4-j}n^{2k}\|\pa^j_y\hh_n\|_{L^2(\R)}^2<\infty\label{littlenorm},
\end{align}
which give us that
\begin{multline}\label{e:coefficientequation}
\cL[\l]\left(\sum_{n=1}^\infty \hh_n(y)\cos(nx)\right)\\
=\sum_{n=1}^\infty (-1)n\sin(nx)\left( (\lambda+y-\Omega_{\ep,\kappa}(y))h_n(y)-\frac{1}{2}\int_{I_\ep}\varpi'_{\ep,\kappa}(\bar{y})\hh_n(\bar{y})e^{-n|y-\bar{y}|}\mbox{d}\bar{y}\right).
\end{multline}

\begin{proof}[Proof of \eqref{e:coefficientequation}]
We  compute $\cL[\lambda]$ acting on the mode $\hh_n(y)\cos(nx)$:
\begin{align*}
\cL[\lambda](\hh_n(y)\cos(nx))=&-n \sin(n x)(\l+y-\Omega_{\ep,\kappa}(y)) \hh_n(y)\\
&-\frac{n}{4\pi}\int_{\T\times I_\ep}\varpi_{\e,\kappa}'(\bar{y})\log\left(\cosh(y-\bar{y})-\cos(x-\bar{x})\right)\hh_n(\bar{y})\sin(n \bar{x})\mbox{d}\bar{\bx} .
\end{align*}
Now, we integrate in the $\bar{x}-$variable on the above double integral, which give us
\begin{align*}
\frac{1}{4\pi}\int_{\T}\log\left[\cosh(y-\bar{y})-\cos(x-\bar{x})\right]\sin(n \bar{x})\mbox{d}\bar{x} &= \frac{1}{4\pi}\int_{\T}\log\left[\cosh(y-\bar{y})-\cos(\bar{x})\right]\sin(n (x-\bar{x}))\mbox{d}\bar{x}\\
&=\frac{\sin(n x)}{4\pi}\int_{\T}\log\left[\cosh(y-\bar{y})-\cos(\bar{x})\right]\cos(n \bar{x})\mbox{d}\bar{x},
\end{align*}
where the last step is due to trigonometric identities and the parity of the integrands.
In addition, last integral can be written as
\begin{multline*}
\frac{\sin(n x)}{n} \frac{1}{4\pi}\int_{\T}\log\left[\cosh(y-\bar{y})-\cos(\bar{x})\right](\sin(n \bar{x}))' \mbox{d}\bar{x}\\
=-\frac{\sin(n x)}{n} \frac{1}{4\pi}\int_{\T}\frac{ \sin(\bar{x}) \sin(n \bar{x})}{\cosh(y-\bar{y})-\cos(\bar{x})}\mbox{d}\bar{x}=-\frac{\sin(n x)}{n} \frac{1}{2}e^{-n|y-\bar{y}|}.
\end{multline*}
This last integral can be computed, for example, by using residues. These last computations yield
\[
\cL[\lambda]\left(\hh_n(y)\cos(n x)\right)=-n\sin(n x)\left((\lambda+y-\Omega_{\ep,\kappa}(y))h_n(y)-\frac{1}{2n}\int_{I_\ep}\omega_{\ep,\kappa}'(\bar{y})\hh(\bar{y})e^{-n|y-\bar{y}|}\mbox{d}\bar{y}\right).
\]
Combining everything we obtain that $\cL[\l]\hh(\bx)$ admits an expansion given by
\[
\sum_{n=1}^{\infty}(-1)n \sin(n x)\left[(\l+y) \hh_n(y)-\Omega_{\e,\kappa}(y)\hh_n(y)-\frac{1}{2n}\int_{I_\ep} \omega_{\e,\kappa}'(\bar{y}) \hh_n(\bar{y}) e^{-n|y-\bar{y}|}\mbox{d}\bar{y}\right].
\]
\end{proof}

To sum up, for any $h\in H^3(I_\ep)$ the operators $\cL_n[\lambda]$ will be define as
\begin{align*}
\cL_n[\lambda]h(y):=-n(\lambda+y-\Omega_{\ep,\kappa}(y))h(y)+\frac{1}{2}\int_{I_\ep}\varpi'_{\ep,\kappa}(\bar{y})h(\bar{y})e^{-n|y-\bar{y}|}\mbox{d}\bar{y},
\end{align*}
in such a way that, in $X(D_\ep)$, the full operator is
\begin{align*}
\cL[\lambda] =\sum_{n=1}^\infty  \sin(nx)\cL_n[\lambda]\Pi_n,
\end{align*}
where $\Pi_n$ is just the projector onto $\cos(nx)$.
\subsection{Rescaling of the decomposition} Since $I_\ep=[-1-\e,-1+\e]\cup[1-\e,1+\e],$ we will make a reflection to consider just the domain $y\in [1-\ep, 1+\ep].$ Recall that $\varpi'_{\ep,\kappa}$ and $\Omega_{\ep,\kappa}$ are odd functions, then we have that
\[
\cL_n[\lambda]h_n (+y)\\=-n(\lambda+y-\Omega_{\ep,\kappa}(y))h_n(+y)+\frac{1}{2}\int_{\R} \varpi'_{\ep,\kappa}(\bar{y})h_n(\bar{y})e^{-n|y-\bar{y}|}\mbox{d}\bar{y},\quad \text{on }  y\in \left(1-\ep,1+\ep\right),
\]
and
\[
\cL_n[\lambda]h_n (-y)=-n(\lambda-y+\Omega_{\ep,\kappa}(y))h_n(-y)+\frac{1}{2}\int_{\R} \varpi'_{\ep,\kappa}(\bar{y})h_n(\bar{y})e^{-n|y+\bar{y}|}\mbox{d}\bar{y},\quad \text{on } y\in \left(1-\ep,1+\ep\right).
\]
In addition, the integral part can be write as
\begin{align*}
\int_{\R} \varpi'_{\ep,\kappa}(\bar{y})h_n(\bar{y})e^{-n|y-\bar{y}|}\mbox{d}\bar{y}=\int_{1-\ep}^{1+\ep}\varpi'_{\ep,\kappa}(\bar{y})\left(h_n(y)e^{-n|y-\bar{y}|}-h_n(-y)e^{-n|y+\bar{y}|}\right)\mbox{d}\bar{y},
\end{align*}
and
\begin{align*}
\int_{\R} \varpi'_{\ep,\kappa}(\bar{y})h_n(\bar{y})e^{-n|y+\bar{y}|}\mbox{d}\bar{y}=\int_{1-\ep}^{1+\ep}\varpi'_{\ep,\kappa}(\bar{y})\left(h_n(y)e^{-n|y+\bar{y}|}-h_n(-y)e^{-n|y-\bar{y}|}\right)\mbox{d}\bar{y}.
\end{align*}
Thus, if we define on $y\in \left(1-\ep, 1+\ep\right)$ the auxiliary functions
\[
\begin{minipage}{.4\linewidth}
  \centering
  $\begin{array}{r@{{}\mathrel{}{}}l}
    h^+_n(y)&:=h_n(+y),\\
    \cL_n^+[\lambda]\left(\begin{array}{cc}h_n^+ \\ h_n^-\end{array}\right)(y)&:=\cL_n[\lambda] h_n(+y),
  \end{array}$
\end{minipage}
\hspace{2 cm}
\begin{minipage}{.4\linewidth}
  \centering
  $\begin{array}{r@{{}\mathrel{}{}}l}
 h^-_n(y)&:=h_n(-y),\\
    \cL_n^-[\lambda]\left(\begin{array}{cc}h_n^+ \\ h_n^-\end{array}\right)(y)&:=\cL_n[\lambda] h_n(-y),
  \end{array}$
\end{minipage}
\]
we find on $y\in \left(1-\ep, 1+\ep\right)$ that
\[
\cL_n^+[\lambda]\left(\begin{array}{cc}h_n^+ \\ h_n^-\end{array}\right)(y)\\
=-n(\lambda+y-\Omega_{\ep,\kappa}(y))h_n^+(y)+\frac{1}{2}\int_{1-\ep}^{1+\ep}\varpi'_{\ep,\kappa}(\bar{y})\left(h^+_n(\bar{y})e^{-n|y-\bar{y}|}-h^-_n(\bar{y})e^{-n|y+\bar{y}|}\right)\mbox{d}\bar{y},
\]
and
\[
\cL_n^-[\lambda]\left(\begin{array}{cc}h_n^+ \\h_n^-\end{array}\right)(y)\\=-n(\lambda-y+\Omega_{\ep,\kappa}(y))h_n^-(y)+\frac{1}{2}\int_{1-\ep}^{1+\ep}\varpi'_{\ep,\kappa}(\bar{y})\left(h^+_n(\bar{y})e^{-n|y+\bar{y}|}-h^-_n(\bar{y})e^{-n|y-\bar{y}|}\right)\mbox{d}\bar{y}.
\]

Now, we make the change of variables $z=\frac{\bar{y}-1}{\ep}$ or $\bar{y}=1+\ep z$  inside of the integrals to get
\[
\frac{1}{2}\int_{-1}^{1}\ep\varpi'_{\ep,\kappa}\left(1+\ep z\right)\left(h^+_n\left(1+\ep z\right)e^{-n|y-\left(1+\ep z \right)|}-h^-_n\left(1+\ep z\right)e^{-n|y+\left(1+\ep z\right)|}\right)\dz,
\]
and
\[
\frac{1}{2}\int_{-1}^{1}\ep\varpi'_{\ep,\kappa}\left(1+\ep z\right)\left(h^+_n\left(1+\ep z\right)e^{-n|y+\left(1+\ep z\right)|}-h^-_n\left(1+\ep z\right)e^{-n|y-\left(1+\ep z\right)|}\right)\dz
\]
We define on $z\in(-1,1)$ the auxiliary functions
\[
\begin{minipage}{.5\linewidth}
  \centering
  $\begin{array}{r@{{}\mathrel{}{}}l}
    w_n^+(z)&:=h_n^+\left(1+\ep z\right),\\
    L^+_n[\lambda]\left(\begin{array}{cc} w_n^+ \\w_n^-\end{array}\right)(z)&:=\cL_n^+[\lambda]\left(\begin{array}{cc}h_n^+ \\h_n^-\end{array}\right)(1+\ep z),
  \end{array}$
\end{minipage}
\begin{minipage}{.5\linewidth}
  \centering
  $\begin{array}{r@{{}\mathrel{}{}}l}
 w_n^-(z)&:=h_n^-\left(1+\ep z\right),\\
     L^-_n[\lambda]\left(\begin{array}{cc} w_n^+ \\w_n^-\end{array}\right)(z)&:=\cL_n^-[\lambda]\left(\begin{array}{cc}h_n^+ \\h_n^-\end{array}\right)(1+\ep z).
  \end{array}$
\end{minipage}
\]
Now, recalling that $\varpi_{\ep,\kappa}(y)=\ep\varphi_\kappa((y-1)/\ep)$, we find that
\begin{align*}
L^+_n[\lambda]\left(\begin{array}{cc}w_n^+ \\w_n^-\end{array}\right)=&-n\left(\lambda+1+\ep z-\Omega_{\ep,\kappa}\left(1+\ep z\right)\right)w_n^+(z)\\
&+\frac{1}{2}\int_{-1}^{1}\ep\varphi'_\kappa(\bar{z})\left(w^+_n(\bar{z})e^{-n\ep|z-\bar{z}|}-w^-_n\left(\bar{z}\right)e^{-n\left|2+\ep(z+\bar{z})\right|}\right)\mbox{d}\bar{z},
\end{align*}
and
\begin{align*}
L^-_n[\lambda]\left(\begin{array}{cc}w_n^+ \\w_n^-\end{array}\right)=&-n\left(\lambda-1-\ep z+\Omega_{\ep,\kappa}\left(1+\ep z\right)\right)w_n^-(z)\\
&+\frac{1}{2}\int_{-1}^{1}\ep\varphi_\kappa'(\bar{z})\left(w^+_n\left(\bar{z}\right)e^{-n|2+\ep(z+\bar{z})|}-w^-_n\left(\bar{z}\right)e^{-n\ep|z-\bar{z}|}\right)\mbox{d}\bar{z}.
\end{align*}
In addition, one can compute that
\begin{align*}
&\Omega_{\ep,\kappa}(y)=\int_{0}^y\varpi_{\ep,\kappa}(\bar{y})\mbox{d}\bar{y}=\int_{0}^{1-\ep}\varpi_{\ep,\kappa}(\bar{y})\mbox{d}\bar{y}+\int_{1-\ep}^y\varpi_{\ep,\kappa}(\bar{y})\mbox{d}\bar{y}=\ep(1-\ep)+\int_{1-\ep}^y\varpi_{\ep,\kappa}(\bar{y})\mbox{d}\bar{y}.
\end{align*}
Thus, the primitive can be write in a more convenient way as
\begin{align*}
\Omega_{\ep,\kappa}(1+\ep z)&=\ep(1-\ep)+\int_{-1}^{1+\ep z}\ep\varphi_{\kappa}\left(\frac{\bar{y}-1}{\ep}\right)\mbox{d}\bar{y}\\
&=\ep(1-\ep)+\ep^2\int_{-1}^z\varphi_\kappa(\bar{z})\mbox{d}\bar{z}=\ep +\ep^2\left(-1+\int_{-1}^z\varphi_{\kappa}(\bar{z})\mbox{d}\bar{z}\right).
\end{align*}
Therefore, we have
\begin{align*}
&\Omega_{\ep,\kappa}\left(1+\ep z\right)= \ep+\ep^2\Phi_{\kappa}(z),
\end{align*}
where
\begin{align*}
\Phi_\kappa(z):=-1+\int_{-1}^z\varphi_{\kappa}(\bar{z})\mbox{d}\bar{z}.
\end{align*}

Combining all we have proved that
\begin{align}\label{l+}
L^+_n[\lambda]\left(\begin{array}{cc}w_n^+ \\w_n^-\end{array}\right)=&-n\left(\lambda+1-\ep+\ep z-\ep^2\Phi_\k(z)\right)w_n^+(z)\\
&+\frac{\ep}{2}\int_{-1}^{1}\varphi'_\kappa(\bar{z})\left(w^+_n(\bar{z})e^{-n\ep|z-\bar{z}|}-w^-_n\left(\bar{z}\right)e^{-n\left|2+\ep(z+\bar{z})\right|}\right)\mbox{d}\bar{z},\nonumber
\end{align}
\begin{align}\label{l-}
L^-_n[\lambda]\left(\begin{array}{cc}w_n^+ \\w_n^-\end{array}\right)=&-n\left(\lambda-1+\ep-\ep z+\ep^2\Phi_\k(z)\right)w_n^-(z)\\
&+\frac{\ep}{2}\int_{-1}^{1}\varphi'_\kappa(\bar{z})\left(w^+_n\left(\bar{z}\right)e^{-n|2+\ep(z+\bar{z})|}-w^-_n\left(\bar{z}\right)e^{-n\ep|z-\bar{z}|}\right)\mbox{d}\bar{z}.\nonumber
\end{align}

\subsection{One dimensionality of the kernel of the linear operator}
The following section consists on two well-differentiated parts. On one hand,  we will prove that there exists an element in the kernel of the linear operator. On the other hand, we will prove that the kernel is the span of this element. To sum up, the main result of this section is to prove the following result:
\begin{theorem}\label{eyufijadom}
For any $M>1,$ there exist positive $\kappa_0=\kappa_0(M)$ and positive $\ep_0=\ep_0(M)$ such that for all $0<\ep<\ep_0$, $0\leq\kappa<\kappa_0$ and  $m\in \N$ such that $m<M$, we can find $\lambda_{\ep,\kappa,m}\in \R$ and a $\frac{2\pi}{m}-$periodic and non-identically zero function $h_{\ep,\kappa,m}\in X(D_\ep)$  solving
\begin{align*}
\cL_{\ep,\kappa}[\lambda_{\ep,\kappa,m}] h_{\ep,\kappa,m}=0,
\end{align*}
where the functional $\cL_{\ep,\kappa}[\lambda]$ is given in \eqref{e:Linear}.
In addition, the kernel $\cL_{\ep,\kappa}[\lambda_{\ep,\kappa,m}]$ on $X(D_\ep)$ is the span of $h_{\ep,\kappa,m}$ and the regularity of the solution is in fact $h_{\ep,\kappa,m}\in C^\infty(D_{\ep}$).

 Importantly, $h_{\ep,\kappa,m}(x,y)$, with $x\in\T$ and $y\in I_{\ep}$, depends non trivially on $x$.
\end{theorem}

\begin{remark} In the statement of the theorem it has been made explicit the dependence of $\cL[\lambda]$ on the parameters $\ep$ and $\kappa$ although it has not been  in \eqref{e:Linear}. It is easy to check that $\cL[\lambda]$ in \eqref{e:Linear} depends on $\ep$ and $\kappa$ through $\Omega_{\ep,\kappa}$ and $\varpi_{\ep,\kappa}$.
\end{remark}

\begin{remark} It is important to emphasize that parameter $\ep$ does not depend on $\kappa$. \end{remark}

\begin{remark}\label{remark}The speed of the traveling wave $\lambda_{\ep,\kappa,m}$ satisfies the expansion in terms of $\ep$-parameter:
\begin{align}\label{speedsuperindex}
\lambda_{\ep,\kappa,m}=1+\lambda^1_{\kappa,m}\ep+\lambda^2_{\ep,\kappa,m}\ep^2,
\end{align}
where
\begin{align*}
\lambda^1_{\k,m}=O(1),
\end{align*}
and
\begin{align*}
|\lambda^2_{\ep,\kappa,m}|\leq C(M).
\end{align*}

\end{remark}

\begin{remark}
In the rest of the section \ref{s:analysislinear} we avoid to use subscripts $\ep,\k,m$ in expression \eqref{speedsuperindex}, which we will ignore in favor of readability. We write with abuse of notation $\lambda=1+\lambda_1\ep+\lambda_2^\ep \ep^2.$
\end{remark}
\begin{remark}
For the degenerate case $(\k=0)$, the same result works line by line for all $m\in \N.$ That is, all the computations can be taken independent of $M.$
\end{remark}

\subsubsection{Proof of existence }\label{s:existence}

In order to show the existence of $(\lambda_{\ep,\kappa,m},h_{\ep,\kappa,m})$ solving
$$\mathcal{L}[\lambda_{\ep,\kappa,m}]h_{\ep,\kappa,m}=0,$$
we fix  $m\in \N$ with $m<M$, and take
\begin{align}
w^+_n(z)=\left\{\begin{array}{cc} 0 & n\neq m,\\ a(z) & n=m, \end{array}\right.\label{w+solucion}\\
w^-_n(z)=\left\{\begin{array}{cc} 0 & n\neq m,\\ b(z) & n=m, \end{array}\right.\label{w-solucion}
\end{align}
where $a$ and $b$ depend on $\ep,\kappa$ and $m$ but we do not make this dependence explicit for sake of simplicity. Thus, we have to find $(\lambda,a,b)\in \R\times H^3([-1,1]\times H^3([-1,1])$ solving
\begin{align}\label{l+cero}
&L^+_m[\lambda]\left(\begin{array}{cc}a \\b\end{array}\right)=0,\\
\label{l-cero}&L^-_m[\lambda]\left(\begin{array}{cc}a \\ b\end{array}\right)=0,\end{align}
where $L^{\pm}_m$ are given in \eqref{l+} and \eqref{l-}. After that, the solution $(\lambda_{\ep,\kappa,m},\,h_{\ep,\kappa,m})$ will be given by
\begin{align*}
&\lambda_{\ep,\kappa,m}=\lambda,\\
&h_{\ep,\kappa,m}(x,y)=a\left(\frac{y-1}{\ep}\right)\cos(mx)\quad \text{for $y\in[+1-\ep,+1+\ep]$},\\
&h_{\ep,\kappa,m}(x,y)=b\left(\frac{-y-1}{\ep}\right)\cos(mx)\quad \text{for $y\in [-1-\ep, -1+\ep]$}.
\end{align*}

In order to solve \eqref{l+cero} and \eqref{l-cero} we introduce the ansatzs
\begin{align*}
a(z)&= \hspace{1.25cm} a_1(z)\ep+a_2^\ep(z) \ep^2,\\
b(z)&=b_0(z)+b^\ep_1(z)\ep,\\
\end{align*}
and
\[
\lambda=1+\lambda_1\ep+\lambda_2^\ep \ep^2,
\]
%\begin{align*}
%&a=a_1\ep+a_2^\ep \ep^2\\
%&b=b_0+b^\ep_1\ep\\
%&\lambda=1+\lambda_1\ep+\lambda_2^\ep \ep^2,
%\end{align*}
where $a_1$, $b_0$ and $\lambda_1$ will depend on $\kappa$ and $m$ but they will not depend on $\ep$. In addition, the remaining terms $a^\ep_2 $, $b^\ep_1)$ and $\l^\ep_2$ depend on $\ep$, $\kappa$ and $m$.
Imposing the equations
\begin{align}\label{eqb}
&(1+\lambda_1-z)b_0(z)=-\frac{1}{2m}\int_{-1}^1\varphi_{\kappa}'(\bar{z})b_0(\bar{z})\mbox{d}\bar{z},\\
&2a_1=-\frac{1}{2m}e^{-2m}\int_{-1}^1\varphi_{\kappa}'(\bar{z})b_0(\bar{z})\mbox{d}\bar{z},\label{eqa}
\end{align}
we get for $(\lambda_2^\ep, a^\ep_2, b^\ep_1)$ the system
\begin{align}\label{eqaep}
&2a^\ep_{2}(z)+\frac{e^{-2m}}{2m}\int_{-1}^1 \varphi'_\kappa(\bar{z})b^\ep_1(\bar{z})e^{-m\ep(z+\bar{z})}\mbox{d}\bar{z}=A_0[a_1,  b_0](z)+\ep A_1[a^\ep_2,\lambda_2^\ep; a_1,b_0,\lambda_1](z),\\
&(1+\lambda_1-z)b_1^\ep(z)+\frac{1}{2m}\int_{-1}^1\varphi'_{\kappa}(\bar{z})b_1^\ep(\bar{z})\mbox{d}\bar{z}+\lambda^\ep_2 b_0(z)=B_0[a_1, b_0](z)+\ep B_1[a_2^\ep, b_1^\ep,\lambda^\ep_2](z),\label{eqbep}
\end{align}
where
\begin{align*}
A_0[a_1,\, b_0](z)&=\frac{1}{2m}\int_{-1}^1 \varphi'_\kappa(\bar{z})a_1(\bar{z})e^{-m\ep|z-\bar{z}|}\mbox{d}\bar{z}
-\frac{e^{-2m}}{2m}\int_{-1}^1\varphi'_{\kappa}(\bar{z})b_0(\bar{z})\frac{e^{-m\ep(z+\bar{z})}-1}{\ep}\mbox{d}\bar{z}
+\ep a_1 \Phi_{\kappa}(z),\\
B_0[a_1,\, b_0](z)&=-\Phi_{\kappa}(z)b_0(z)+\frac{e^{-2m}}{2m}\int_{-1}^1 \varphi'_{\kappa}(\bar{z})a_1(\bar{z})e^{-m\ep(z+\bar{z})}\mbox{d}\bar{z},\\
\end{align*}
and
\begin{align*}
A_1[a^\ep_2,\, \lambda_2^\ep;\, a_1,\, \lambda_1](z)=-&(-1+\lambda_1+z)a^\ep_2(z)-a_1\lambda_2^\ep-\ep(\lambda^\ep_2-\Phi_{\kappa}(z))a^\ep_2(z)\\
+&\frac{1}{2m}\int_{-1}^1\varphi'_{\kappa}(\bar{z})a^\ep_2(\bar{z})e^{-m\ep|z-\bar{z}|}\mbox{d}\bar{z},\\
B_1[a^\ep_2,\,b_1^\ep,\, \lambda_2^\ep](z)=-& (\Phi_{\kappa}(z)+\lambda_2^\ep)b^\ep_1(z)+\frac{e^{-2m}}{2m}\int_{-1}^1\varphi'_\kappa(\bar{z})a_2^\ep(\bar{z})e^{-m\ep(z+\bar{z})}\mbox{d}\bar{z}\\
-&\frac{1}{2m}\int_{-1}^1\varphi'_{\kappa}(\bar{z})b_1^\ep(\bar{z})\frac{e^{-m\ep|z-\bar{z}|}-1}{\ep}\mbox{d}\bar{z}.
\end{align*}
%\begin{align*}
%A_0[a_1,\, b_0]=&\frac{1}{2m}\int_{-1}^1 \varphi'_\kappa(z')a_1(z')e^{-m\ep|z-z'|}dz'\\
%-&\frac{1}{2m}e^{-2m}\int_{-1}^1\varphi'_{\kappa}(z')b_0(z')\frac{e^{-m\ep(x+x')}-1}{\ep}dz'\\
%+&\ep \Phi_{\kappa}a_1\\
%A_1[a^\ep_2,\, \lambda_2^\ep;\, a_1,\, \lambda_1]=-&(-1+\lambda_1+x)a^\ep_2-a_1\lambda_2^\ep-\ep(\lambda^\ep_2-\Phi_{\kappa})a^\ep_2\\
%+&\frac{1}{2m}\int_{-1}^1\varphi'_{\kappa}(z')a^\ep_2(z')e^{-m\ep|z-z'|}dz',\\
%B_0[a_1,\, b_0]=-&\Phi_{\kappa}b_0+\frac{e^{-2m}}{2m}\int_{-1}^1 \varphi'_{\kappa}(z')a_1(z')e^{-m\ep(z+z')}dz'\\
%B_1[a^\ep_2,\,b_1^\ep,\, \lambda_2^\ep]=-& (\Phi_{\kappa}+\lambda_2^\ep)b^\ep_1+\frac{e^{-2m}}{2m}\int_{-1}^1\varphi'_\kappa(z')a_2^\ep(z')e^{-m\ep(z+z')}dz'\\
%-&\frac{1}{2m}\int_{-1}^1\varphi'_{\kappa}(z')b_1^\ep(z')\frac{e^{-m\ep|z-z'|}-1}{\ep}dz'.
%\end{align*}
Here, we note that
\begin{align*}
\|A_0\|_{L^2}\leq &C\left(|a_1|+\|b_0\|_{L^2}\right),\\
\|A_1\|_{L^2}\leq &C\left(1+|\lambda_1|\right)\|a_2^\ep\|_{L^2}+ C|a_1||\lambda^\ep_2|+C |\lambda^\ep_2| \|a_2^\ep\|_{L^2}+C(1+\ep)\|a_2^\ep\|_{L^2},\\
\|B_0\|_{L^2}\leq & C\left(|a_1|+\|b_0\|_{L^2}\right),\\
\|B_1\|_{L^2}\leq & C \|b_1^\ep\|_{L^2}+ C \|b_1^\ep \|_{L^2}|\lambda_2^\ep|+ C\left(\|a_2^\ep\|_{L^2}+\|b_1^\ep\|_{L^2}\right),
\end{align*}
%\begin{align*}
%||A_0||_{L^2}\leq &C\left(||a_1||_{L^2}+||b_0||_{L^2}\right),\\
%||A_1||_{L^2}\leq &C(1+|\lambda_1|)||a_2^\ep||_{L^2}+ C||a_1||_{L^2}|\lambda^\ep_2|+C |\lambda^\ep_2|||a_2^\ep||_{L^2}\\&+C(1+\ep)||a_2^\ep||_{L^2},\\
%||B_0||_{L^2}\leq & C\left(||a_1||_{L^2}+||b_0||_{L^2}\right),\\
%||B_1||_{L^2}\leq & C ||b_1^\ep||_{L^2}+ C||b_1^\ep||_{L^2}|\lambda_2^\ep|+ C\left(||a_2^\ep||_{L^2}+||b_1^\ep||_{L^2}\right),
%\end{align*}
where the constant $C$ does not depend on neither $\ep$, $\kappa$ or $m$, and $\|\cdot\|_{L^2}=\|\cdot\|_{L^2([-1,1])}$.

So first we will solve \eqref{eqb} on $b_0$ by choosing $\lambda_1$ in a suitable way. This will give $a_1$ in \eqref{eqa}. Then, equations \eqref{eqaep} and \eqref{eqbep} will be written to be able to determine $a_2^\ep$, $b_1^\ep$ and $\lambda_2^\ep$ by a contraction argument in $\ep$.\\

We deal with equation \eqref{eqb} in the following proposition.
\begin{proposition}\label{propo1} For any $M>1,$ there exists $\kappa_0=\kappa_0(M)$ such that for all $1\leq m<M$ and $0\leq \kappa<\kappa_0$, we can find a solution $(\lambda_1,\, b_0)\in \R\times C^\infty([-1,1])$  of equation \eqref{eqb}. In addition,  fixed $m$ and $\kappa$, this solution is unique (modulo multiplication by constant) and satisfies
\begin{align}
b_0(z)=\frac{1}{1+\lambda_1-z},\label{b0formula} \\
\lambda_1=\frac{2}{e^{4m}-1}+\emph{error}\label{lambda1formula},
\end{align}
where
\begin{align*}
|\emph{error}|\leq C(M)\kappa.
\end{align*}
Moreover, we get
\begin{align*}
\|b_0\|_{L^2([-1,1])}\leq C(M), && \|\pa_z^{k)}b_0\|_{L^\infty([-1,\,1])}\leq C(k,M),\quad k=0,1,2,... && a_1=\frac{1}{2}e^{-2m}.
\end{align*}
%We emphasis that the constant $C(M)$ depend on $M$  but it does not depend on neither $m$ or $\kappa$, for $0\leq \kappa<\kappa_0$ and $1\leq m<M$.
\end{proposition}
\begin{remark}
It is important to emphasize that constant $C(M)$ depend on $M$  but it does not depend on neither $m$ or $\kappa$, for $0\leq \kappa<\kappa_0$ and $1\leq m<M$.
\end{remark}

Before prove the previous result, we pick $\l_1$ appropriately. Notice that from \eqref{eqb} we see that $b_0$ must be of the form
\begin{align}\label{ansatzb}
b_0(z)=\frac{C}{1+\lambda_1-z},
\end{align}
for some free constant $C$. Plugging \eqref{ansatzb} into \eqref{eqb} one finds that
\begin{align*}
1=\frac{1}{2m}\int_{-1}^1\frac{-\varphi_\kappa'(\bar{z})}{1+\lambda_1-\bar{z}}\mbox{d}\bar{z},
\end{align*}
which is an equation for $\lambda_1$.
\begin{lemma}\label{lambda1} For any $M>1,$ there exist $\kappa'_0=\kappa'_0(M)$ such that for all $0\leq \kappa \leq \kappa'_0$ and $1\leq m<M$, we can find $\lambda_1>0$ (depending on $\kappa$ and $m$) such that
\begin{align}\label{lambda1eq}
1=\frac{1}{2m}\int_{-1}^1\frac{-\varphi_{\kappa}'(\bar{z})}{1+\lambda_1-\bar{z}}\mbox{d}\bar{z}.
\end{align}
In addition, fixed $m$ and $\kappa$ this solution is unique and satisfies
$$\lambda_1\geq\lambda^*(M)>0,$$
where $\lambda^*(M)$ just depends on $M$ (it does not depend on either $\kappa$ or $m$).
\end{lemma}
\begin{proof}
We will use the properties of $\varphi_\kappa$ stated in Lemma \ref{propiedadesvarphi}. Firstly, we define the function
\begin{align*}
I(\lambda):=\frac{1}{2m}\int_{-1}^1\frac{-\varphi'_\kappa(\bar{z})}{1+\lambda-\bar{z}}\mbox{d}\bar{z},
\end{align*}
for $\lambda\geq 0$. We have that $I(\lambda)$ is  positive, continuous, decreasing with $I'(\lambda)<0$ and satisfying
\begin{equation}\label{limlambdainf}
\lim_{\lambda\to +\infty}I(\lambda)=0.
\end{equation}
In addition, using Lemma \ref{propiedadesvarphi} we obtain the following lower bound:
\begin{align*}
I(\lambda)&=\frac{1}{2m}\int_{-1}^1\frac{-\varphi_\kappa'(\bar{z})}{1+\lambda-\bar{z}}\mbox{d}\bar{z}\geq \frac{1}{2M}\int_{-1}^1\frac{-\varphi'_{\kappa}(\bar{z})}{1+\lambda-\bar{z}}\mbox{d}\bar{z}\geq \frac{1}{2M}\int_{-1}^{1-\kappa}\frac{-\varphi'_{\kappa}(\bar{z})}{1+\lambda-\bar{z}}\mbox{d}\bar{z}\\
&=\frac{1}{2M}\int_{-1}^{1-\k}\frac{-\varphi'_{\kappa}(\bar{z})-\tfrac{1}{2}}{1+\lambda-\bar{z}}\mbox{d}\bar{z}+\frac{1}{4M}\int_{-1}^{1-\kappa}\frac{1}{1+\lambda-\bar{z}}\mbox{d}\bar{z}
\\ &\geq -\frac{1}{2M}\frac{1}{\lambda+\kappa}\|\varphi'_{\kappa}+\tfrac{1}{2}\|_{L^1([-1,1-\k])}+\frac{1}{4M}\left(\log(2+\lambda)-\log(\lambda+\kappa)\right)\\
&\geq -\frac{C}{2M}\frac{\kappa}{\lambda+\kappa}+\frac{1}{4M}\log(2+\lambda)-\frac{1}{4M}\log(\lambda+\kappa),\\
&\geq -\frac{C}{2M}-\frac{1}{4M}\log(\lambda+\kappa).
\end{align*}
Let $\gamma$ be small enough  such that
$$1\leq -\frac{C}{2M}-\frac{1}{4M}\log(\gamma).$$
Notice that $\gamma$ only depends on $M$. Then, taking $\kappa'_0=\gamma/2$ and $\lambda^*=\gamma/2$, we have proved that
\begin{align*}
I(\lambda^*)\geq -\frac{C}{2M}-\frac{1}{4M}\log(\gamma/2+\kappa)\geq -\frac{C}{2M}-\frac{1}{4M}\log(\gamma/2+\kappa'_0)\geq 1.
\end{align*}
for any $0\leq \kappa \leq \kappa'_0=\g/2$ and $1\leq m\leq M$.

Combining the lower bound with \eqref{limlambdainf} and the fact that $I(\l)$ is strictly decreasing, we can conclude that there exists a unique $\lambda_1\geq \lambda^*$ solving \eqref{lambda1eq} for all $0\leq \kappa<\kappa'_0$ and $1\leq m\leq M$.
\end{proof}

After choose a suitable $\lambda_1>\lambda^*(M)$ by Lemma \ref{lambda1}, we obtain $b_0$ through equation \eqref{ansatzb}. Now, we have all the ingredients to check the conclusions of Proposition \ref{propo1}.
\begin{proof}[Proof of Proposition \ref{propo1}]
Notice that
$$b_0(z)=\frac{1}{1+\lambda_1-z},$$
solves \eqref{eqb} and
$\|b_0\|_{L^2([-1,1])}=\sqrt{\frac{2}{(2+\lambda_1)\lambda_1}}<\sqrt{\frac{1}{\lambda_1}}\leq \sqrt{\frac{1}{\lambda^*(M)}}$, with $\lambda^*(M)$ given by Lemma \ref{lambda1}. We also have that $\|\pa_z^{k)}b_0\|_{L^\infty([-1,1])}\leq C(k,M)$. From \eqref{eqa} and \eqref{lambda1eq} we have that $a_1=\frac{1}{2}e^{-2m}$.

In order to get \eqref{lambda1formula} we proceed as follow. From the identity \eqref{lambda1eq} we get
\begin{align}\label{despejando}
1=\frac{1}{2m}\int_{-1}^1\frac{-\varphi'_\kappa(\bar{z})-1/2}{1+\lambda_1-\bar{z}}\mbox{d}\bar{z}+\frac{1}{4m}\int_{-1}^1\frac{1}{1+\lambda_1-\bar{z}}\mbox{d}\bar{z}=\frac{1}{4m}\left(\Lambda_1+\log\left(\frac{2+\l_1}{\l_1}\right)\right),
\end{align}
where
\[
\Lambda_1:=2\int_{-1}^1\frac{-\varphi'_\kappa(\bar{z})-1/2}{1+\lambda_1-\bar{z}}\mbox{d}\bar{z}.
\]
Solving \eqref{despejando} we obtain that
\[
\l_1=\frac{2}{e^{4m-\Lambda_1}-1},
\]
and consequently
\begin{align*}
\left|\lambda_1-\frac{2}{e^{4m}-1}\right|\leq 2e^{4m}\frac{\left|1-e^{-\Lambda_1}\right|}{\left|e^{4m-\Lambda_1}-1\right|\left(e^{4m}-1\right)}.
\end{align*}
One can estimate each of the terms of the last factor. Notice that numerator can be written as
\begin{align*}
|1-e^{-\Lambda_1}|=\left|\int_{-1}^0\left(\frac{d}{ds}e^{s\Lambda_1}\right) \mbox{d}s\right|= \left|\Lambda_1\int_{0}^1 e^{-s\Lambda_1}\mbox{d}s\right|\leq |\Lambda_1|e^{|\Lambda_1|}.
\end{align*}
In addition
\begin{align*}
|\Lambda_1|\leq 2\int_{-1}^1\frac{|\varphi'_\kappa(\bar{z})+\tfrac{1}{2}|}{1+\lambda_1-\bar{z}}\mbox{d}\bar{z}\leq \frac{C\kappa}{\lambda_1}\leq \frac{C\kappa}{\lambda^*(M)},
\end{align*}
where $C$ is a universal constant coming from Lemma \ref{propiedadesvarphi}. For the denominator we just note that
\begin{align*}
e^{4m-\Lambda_1}-1\geq e^{4-\frac{C\kappa}{\lambda^*(M)}}-1\geq e^2-1>1 \qquad \text{if} \quad \tfrac{C\kappa}{\lambda^*(M)}<2.
\end{align*}
Then, taking $\kappa<\kappa_0\equiv \min(\kappa'_0, \kappa''_0)$, where $\kappa'_0$ is given by Lemma \ref{lambda1} and $\kappa''_0=\frac{2\lambda^*(M)}{C}$, we find that
\begin{align*}
\left|\lambda_1-\frac{2}{e^{4m}-1}\right|\leq C(M)\kappa.
\end{align*}
for all $0\leq \kappa <\kappa_0$ and $1\leq m <M$.
\end{proof}

To sum up, Proposition \ref{propo1} shows that
\begin{align*}
\|A_0\|_{L^2}\leq &C(M),\\
\|A_1\|_{L^2}\leq &C(M)\left(\|a_2^\ep\|_{L^2}+ |\lambda_2^\ep|+ |\lambda^\ep_2| \|a_2^\ep\|_{L^2}\right),\\
\|B_0\|_{L^2}\leq & C(M),\\
\|B_1\|_{L^2}\leq & C(M)\left( \|b_1^\ep\|_{L^2}+ \|b_1^\ep\|_{L^2}|\lambda_2^\ep|+ \|a_2^\ep\|_{L^2}\right).
\end{align*}
In the following two lemmas we are going to learn how to invert the left hand side of \eqref{eqbep}.

\begin{lemma}\label{primero}Let $\lambda_1$ be as in Lemma \ref{lambda1} and let $F\in L^2([-1,1])$ satisfying  $(F, \varphi'_{\kappa}b_0)_{L^2}=0.$ Then
$$f(z):=F(z)b_0(z)=\frac{F(z)}{1+\lambda_1-z},$$
solves
$$(1+\lambda_1-z)f(z)+\frac{1}{2m}\int_{-1}^1\varphi'_{\kappa}(\bar{z})f(\bar{z})\mbox{d}\bar{z}=F(z).$$
\end{lemma}
\begin{proof}Direct computation.
\end{proof}

\begin{lemma} \label{segundo}Let $\lambda_1$ be as in Lemma \ref{lambda1} and let $G\in L^2([-1,1])$. Then, if
\begin{align*}
\lambda:=\frac{\int_{-1}^1G(\bar{z})b_0(\bar{z})\varphi'_{\kappa}(\bar{z})\mbox{d}\bar{z}}{\int_{-1}^1 b_0(\bar{z}) b_0(\bar{z})\varphi'_\kappa(\bar{z})\mbox{d}\bar{z}},
\end{align*}
the function
\begin{align*}
f(z):=(G(z)-\lambda b_0(z))b_0(z)=\frac{G(z)-\lambda b_0(z)}{1+\lambda_1-z},
\end{align*}
solves
\begin{align}\label{eqsegundo}
(1+\lambda_1-z)f(z)+\frac{1}{2m}\int_{-1}^1\varphi_{\kappa}'(\bar{z})f(\bar{z})\mbox{d}\bar{z}+\lambda b_0(z)=G(z).
\end{align}
In addition,
\begin{align}\label{lambdasegundo}
|\lambda|\leq C(M) \|G\|_{L^2([-1,1])}.
\end{align}
\end{lemma}
\begin{proof}It is easy to check that $F(z):=G(z)-\lambda b_0$ satisfies $(F,\varphi'_{\kappa}b_0)_{L^2}=0$ by  the definition of $\lambda$. Then, we can apply directly Lemma \ref{primero} to obtain \eqref{eqsegundo}. To conclude \eqref{lambdasegundo}, we note that
\[
\left| \int_{-1}^1G(\bar{z})b_0(\bar{z})\varphi'_{\kappa}(\bar{z})\mbox{d}\bar{z}\right|\leq \|G\|_{L^2([-1,1])}\|b_0\|_{L^2([-1,1])}\|\varphi'_\k\|_{L^{\infty}([-1,1])}\leq C(M)\|G\|_{L^2([-1,1])},
\]
and since $b_0(z)>(2+\lambda_1)^{-1}$ we get
\begin{align*}
\left|\int_{-1}^1 b_0(\bar{z})b_0(\bar{z}) \varphi'_{\kappa}(\bar{z})\mbox{d}z\right|\geq \frac{1}{2+\lambda_1}\int_{-1}^1-\varphi'_{\kappa}(z)b_0(z)dz=\frac{2m}{2+\lambda_1},
\end{align*}
where in the last step we have used \eqref{lambda1eq}. Combining both we have
\[
|\lambda|\leq \frac{2+\lambda_1}{2m} C(M)\|G\|_{L^2([-1,1])}\leq C(M)\|G\|_{L^2([-1,1])}.
\]
Here we have used that $\lambda_1<C$ where $C$ is a universal constant. This is easy to check from
\begin{align*}
I(\mu)=\frac{1}{2m}\int_{-1}^1\frac{-\varphi'_\kappa(\bar{z})}{1+\mu-\bar{z}}\mbox{d}\bar{z}\leq C \log \left(1+\frac{2}{\mu}\right),
\end{align*}
since for $\mu$ big enough we get that $I(\mu)<1$. Recall that $\l_1$ solves the equation $I(\mu)=1$.
\end{proof}

Lemma \ref{segundo} implies that if we solve
\begin{align}\label{eqaep2}
&2a^\ep_{2}(z)+\frac{e^{-2m}}{2m}\int_{-1}^1 \varphi'_\kappa(\bar{z})b^\ep_1(\bar{z})e^{-m\ep(z+\bar{z})}\mbox{d}\bar{z}=A_0[a_1,b_0](z)+\ep A_1[a^\ep_2,\lambda_2^\ep; a_1,b_0,\lambda_1](z),\\
&b_1^\ep(z)= \left(B_0[a_1,b_0](z)+\ep B_1[a_2^\ep,b_1^\ep,\lambda_2^\ep](z)-\lambda_2^\ep b_0(z)\right)b_0(z),
\label{eqbep2}\\&\lambda_{2}^\ep= \frac{\int_{-1}^1(B_0[a_1,b_0](\bar{z})+\ep B_1[a_2^\ep,b_1^\ep,\lambda_2^\ep](\bar{z})) b_0(\bar{z}) \varphi'_{\kappa}(\bar{z}) \mbox{d}\bar{z}}{\int_{-1}^1 b_0(\bar{z}) b_0(\bar{z}) \varphi'_\kappa(\bar{z}) \mbox{d}\bar{z}},\label{eqlambda2}
\end{align}
we find a solution of \eqref{eqaep} and \eqref{eqbep}. Now, to obtain a solution of \eqref{eqaep2}, \eqref{eqbep2} and \eqref{eqlambda2} we firstly introduce \eqref{eqbep2} and \eqref{eqlambda2} in the second term of the left hand side of \eqref{eqaep2} and \eqref{eqlambda2} in the right hand side of \eqref{eqbep2}. In this way, a solution of

\begin{align}\label{eqaep3}
&2a^\ep_{2}(z)=\tilde{A}_0[a_1,b_0](z)+\ep \tilde{A}_1[a^\ep_2,b^\ep_1,\lambda_2^\ep; a_1,b_0,\lambda_1](z),\\
&b_1^\ep(z)=\tilde{B}_0[a_1,b_0](z)+\ep\tilde{B}_1[a^\ep_2,b_1^\ep,\lambda_2^\ep](z),
\label{eqbep3}\\&\lambda_{2}^\ep= \frac{\int_{-1}^1(B_0[a_1,b_0](\bar{z})+\ep B_1[a_2^\ep,b_1^\ep,\lambda_2^\ep](\bar{z})) b_0(\bar{z}) \varphi'_{\kappa}(\bar{z}) \mbox{d}\bar{z}}{\int_{-1}^1 b_0(\bar{z}) b_0(\bar{z}) \varphi'_\kappa(\bar{z}) \mbox{d}\bar{z}}\label{eqlambda3},
\end{align}
where
\begin{align*}
\tilde{A}_0[a_1,\, b_0](z)&=A_0[a_1,b_0](z)-\frac{e^{-2m}}{2m}\int_{-1}^1\varphi'_\kappa(\bar{z})b_0(\bar{z})B_0[a_1,b_0](\bar{z})e^{-m\ep(z+\bar{z})}\mbox{d}\bar{z}\\
&+\frac{e^{-2m}}{2m}\frac{\int_{-1}^1 B_0[a_1,b_0](\bar{z})b_0(\bar{z})\varphi'_{\kappa}(\bar{z})\mbox{d}\bar{z}}{\int_{-1}^1 b_0(\bar{z}) b_0(\bar{z}) \varphi'_\kappa(\bar{z}) \mbox{d}\bar{z}}\left(\int_{-1}^1b_0(\bar{z}) b_0(\bar{z}) \varphi'_\kappa(\bar{z})e^{-m\ep(z+\bar{z})} \mbox{d}\bar{z}\right), \\
\tilde{A}_1[a_2^\ep,\,b_1^\ep, \lambda_2^\ep; a_1, b_0](z)&=A_1[a_2^\ep,\,b_1^\ep,\lambda_2^\ep; a_1, b_0](z)-\frac{e^{-2m}}{2m}
\int_{-1}^1\varphi'_z(\bar{z})b_0(\bar{z})B_1[a_2^\ep,\,b_1^\ep,\lambda_2^\ep](\bar{z})e^{-\ep m(z+\bar{z})}\mbox{d}\bar{z}\\
&+\frac{e^{-2m}}{2m}\frac{\int_{-1}^1B_1[a_2^\ep,b_1^\ep,\lambda_2^\ep](\bar{z})  b_0(\bar{z}) \varphi'_{\kappa}(\bar{z})\mbox{d}\bar{z} }
{\int_{-1}^1 b_0(\bar{z}) b_0(\bar{z}) \varphi'_\kappa(\bar{z}) \mbox{d}\bar{z}}\left(\int_{-1}^1b_0(\bar{z}) b_0(\bar{z}) \varphi'_\kappa(\bar{z})e^{-m\ep(z+\bar{z})} \mbox{d}\bar{z}\right),
\end{align*}
and
\begin{align*}
\tilde{B}_0[a_1,b_0](z)&=b_0(z)\left( B_0[a_1,b_0](z)-b_0(z)\frac{\int_{-1}^1 B_0[a_1,b_0](\bar{z})b_0(\bar{z}) \varphi'_{\kappa}(\bar{z}) \mbox{d}\bar{z}}{\int_{-1}^1 b_0(\bar{z}) b_0(\bar{z})\varphi'_{\kappa}(\bar{z}) \mbox{d}\bar{z}}\right)\\
\tilde{B}_1[a_2^\ep,b_1^\ep,\lambda_2^\ep](z)&=b_0(z)\left( B_1[a_2^\ep,b_1^\ep,\lambda_2^\ep](z)-b_0(z) \frac{\int_{-1}^1 B_1[a_2^\ep,b_1^\ep,\lambda_2^\ep](\bar{z}) b_0(\bar{z}) \varphi'_{\kappa}(\bar{z}) \mbox{d}\bar{z}}{\int_{-1}^1 b_0(\bar{z}) b_0(\bar{z})\varphi'_{\kappa}(\bar{z}) \mbox{d}\bar{z}}\right),
\end{align*}
will give us a solution of \eqref{eqaep2}, \eqref{eqbep2} and \eqref{eqlambda2}.\\

Finally, in order to solve \eqref{eqaep3}, \eqref{eqbep3} and \eqref{eqlambda3} we can use a contraction argument on parameter $\ep$. Indeed, let us call $F_A[a_2^\ep,b_1^\ep,\lambda_2^\ep;a_1,b_0](z)$, $F_B[a_2^\ep,b_1^\ep,\lambda_2^\ep;a_1,b_0](z)$, $F_\l[a_2^\ep,b_1^\ep,\lambda_2^\ep;a_1,b_0](z)$ to the right hand side of \eqref{eqaep3}, \eqref{eqbep3} and \eqref{eqlambda3} respectively. We define the constants values
$$C_A:= \|\tilde{A}_0[a_1,b_0]\|_{L^2([-1,1])}, \quad C_{B}:=\|\tilde{B}_0[a_1,b_0]\|_{L^2([-1,1])},$$
and \begin{align*}
C_{\lambda}:=\left|\frac{\int_{-1}^1 B_{0}[a_1,b_0](\bar{z}) b_0(\bar{z})\varphi'_{\kappa}(\bar{z})\mbox{d}\bar{z} }{\int_{-1}^1  b_0(\bar{z}) b_0(\bar{z})\varphi'_{\kappa}(\bar{z}) \mbox{d}\bar{z}}\right| \leq C(M).
\end{align*}

Note that the system \eqref{eqaep3}, \eqref{eqbep3} and \eqref{eqlambda3} is quadratic because of the products $\lambda_{2}^\ep a_2^\ep$ and $\lambda_2^\ep b_1^\ep$. Thus, we first check that, if  $\|2 a_2^\ep \|_{L^2([-1,1])}< 2 C_A$, $\|b_1^\ep\|_{L^2([-1,1])}< 2C_{B}$ and $|\lambda_{2}^\ep|<2C_{\lambda}$ for $\ep$ small enough (with respect to a constant that just depends on $M$), then
\begin{align*}
\|F_A[a_2^\ep,b_1^\ep,\lambda_2^\ep;a_1,b_0]\|_{L^2([-1,1])}<2 C_A,\\
\|F_B[a_2^\ep,b_1^\ep,\lambda_2^\ep;a_1,b_0]\|_{L^2([-1,1])}<2 C_B,\\
|F_{\lambda}[a_2^\ep,b_1^\ep,\lambda_2^\ep;a_1,b_0]|<2C_{\lambda}.
\end{align*}
In addition, for any pair $u=(2a_2^\ep, b_1^\ep, \lambda_2^\ep)$ and $\tilde{u}=(2\tilde{a}_2^\ep, \tilde{b}_1^\ep, \tilde{\lambda}_2^\ep)$ satisfying $\|2a_2^\ep, 2\tilde{a}^\ep\|_{L^2([-1,1])}< 2 C_A$, $\|b_1^\ep, \tilde{b}_1^\ep \|_{L^2([-1,1])}< 2C_{B}$ and $|\lambda_{2}^\ep, \tilde{\lambda}^\ep_2|<2C_{\lambda}$ we have that
\begin{align*}
\|F_A[u;a_1,b_0]-F_A[\tilde{u};a_1,b_0]\|_{L^2([-1,1])}\leq \ep C(M)\|u-\tilde{u}\|_{L^2([-1,1])\times L^2([-1,1])\times \R},\\
\|F_B[u;a_1,b_0]-F_B[\tilde{u};a_1,b_0]\|_{L^2([-1,1])}\leq \ep C(M)\|u-\tilde{u}\|_{L^2([-1,1])\times L^2([-1,1])\times \R},\\
|F_\lambda[u;a_1,b_0]-F_\lambda[\tilde{u};a_1,b_0]|\leq \ep C(M)\|u-\tilde{u}\|_{L^2([-1,1])\times L^2([-1,1])\times \R}.
\end{align*}
Taking $\e_0:=C(M)^{-1},$ we have that right hand side of \eqref{eqaep2}, \eqref{eqbep2}, \eqref{eqlambda2} is a contraction mapping for all $0<\ep<\e_0(M).$ This implies that there exist $(a_2^\ep, b_1^\ep, \lambda_{2}^\ep)\in L^2([-1,1])\times L^2([-1,1])\times \R$, with $\|2a_2^\ep\|_{L^2([-1,1])}<2C_A$, $\|b_1^\ep \|_{L^2([-1,1])}<{2C_B}$ and $|\lambda_{2}^\ep|<2C_{\lambda}$ solving \eqref{eqaep3}, \eqref{eqbep3} and \eqref{eqlambda3}, for all $0<\ep<\ep_0(M)$.

\begin{remark}
Let us emphasize that $C_A$, $C_B$ and $C_\lambda$ are uniformly bounded by some $C(M).$
\end{remark}

Therefore, we have shown the existence of a solution $(\lambda,a,b)\in \R\times L^2([-1,1])\times L^2([-1,1])$ solving \eqref{l+cero} and \eqref{l-cero}, with

\begin{align}
a(z)=& \,  a_1 \ep +a_2^\ep(z) \ep^2,\label{asolucion}\\
b(z)=& \, b_0(z)+b_1^\ep(z) \ep,\label{bsolucion}\\
\lambda  = & \, 1+\lambda_1\ep +\lambda_2^\ep \ep^2\label{lambdasolucion},
\end{align}
and
\begin{align}
\|a_2^\ep\|_{L^2([-1,1])}\leq C(M),\quad \|b_1^\ep\|_{L^2([-1,1])}\leq C(M),\quad |\lambda_2^\ep|\leq C(M)\label{cotas}.
\end{align}
Next we shall show that this solution is smooth.

\subsubsection{Regularity}\label{regularity}
In order to determine the precise regularity of the constructed solution $(a,b)$ we proceed as follows. In first place, we recall that the system \eqref{l+cero}, \eqref{l-cero} can be written  as
\begin{align}
(\l+1-\e +\e z -\e^2 \Phi_{\kappa}(z))a(z)&=\frac{\e}{2m}\int_{-1}^{1}\varphi_{\kappa}'(\bar{z}) \left( a(\bar{z}) e^{-m\e|z-\bar{z}|}-b(\bar{z}) e^{-m|2+\e(z+\bar{z})|}\right)\mbox{d}\bar{z}, \label{eqboostra}\\
(\l-1+\e-\e z +\e^2 \Phi_{\kappa}(z))b(z)&=
\frac{\e}{2m}\int_{-1}^{1}\varphi_{\kappa}'(\bar{z}) \left( a(\bar{z}) e^{-m|2+\e(z+\bar{z})|}-b(\bar{z}) e^{-m\e|z-\bar{z}|}\right)\mbox{d}\bar{z}. \label{eqboostrb}
\end{align}
Notice that neither $\l+1-\e +\e z -\e^2 \Phi_{\kappa}(z)$ nor $\l-1+\e-\e z +\e^2 \Phi_{\kappa}(z)$ have zeros on the unit interval. More specifically, for the computed eigenvalue $\l=1+\l_1\e+O(\e^2)$ we have that
\begin{align*}
\l+1-\e +\e z -\e^2 \Phi_{\kappa}(z)&=2+O(\e)\geq 1,\\
\l-1+\e-\e z +\e^2 \Phi_{\kappa}(z)&=(1+\l_1-z)\e+O(\e^2)\geq \e(\l_1+O(\e)).
\end{align*}
As we have proved before that $\l_1$ is strictly positive, see Lemma \ref{lambda1}, taking $\e$ small enough these two quantities are not zero and we are allow to divide both sides of \eqref{eqboostra}, \eqref{eqboostrb} by these factors:
\begin{align*}
a(z)&=\frac{\e}{2m(\l+1-\e +\e z -\e^2 \Phi_{\kappa}(z))}\int_{-1}^{1}\varphi_{\kappa}'(\bar{z}) \left( a(\bar{z}) e^{-m\e|z-\bar{z}|}-b(\bar{z}) e^{-m(2+\e(z+\bar{z}))}\right)\mbox{d}\bar{z},\\
b(z)&=\frac{\e}{2m(\l-1+\e-\e z +\e^2 \Phi_{\kappa}(z))}\int_{-1}^{1}\varphi_{\kappa}'(\bar{z}) \left( a(\bar{z}) e^{-m(2+\e(z+\bar{z}))}-b(\bar{z}) e^{-m\e|z-\bar{z}|}\right)\mbox{d}\bar{z}.
\end{align*}
Since these terms are the factors that appear multiplying the solution $(a,b)$ in the left hand side of \eqref{eqboostra}, \eqref{eqboostrb} we have that their inverses are $C^\infty-$functions on the unit interval $[-1,1]$. In addition, the integral terms
\begin{align*}
e^{-m\ep z}\int_{-1}^1\varphi_{\kappa}'(\bar{z})a(\bar{z})e^{-m \ep \bar{z}}\mbox{d}\bar{z}\quad \text{and}\quad e^{-m \ep z}\int_{-1}^1\varphi_{\kappa}'(\bar{z})b(\bar{z})e^{-m \ep \bar{z}}\mbox{d}\bar{z},
\end{align*}
are also $C^\infty-$functions. Finally, the remaining integral term $f(z)=\int_{-1}^1\varphi_{\kappa}'(\bar{z})a(\bar{z})e^{-m\ep |z-\bar{z}|}\mbox{d}\bar{z}$ is in $H^2([-1,1])$ for $a\in L^2([-1,1])$. To see this one just has to take two derivatives on $f(z)$ to get that
\begin{align*}
f''(z)=-2m\ep \varphi_\kappa'(z)a(z)+(m\ep)^2 f(z).
\end{align*}
In fact, we have that $f(z)\in H^{k+2}([-1,1])$ if $a\in H^k([-1,1])$. The same hods for
$$\int_{-1}^1 \varphi_{\kappa}(\bar{z})b(\bar{z})e^{-m\ep |z-\bar{z}|}\mbox{d}\bar{z}.$$
All this yields that if $(a,b)\in H^k([-1,1])\times H^k([-1,1])$ then $(a,b)\in H^{k+2}([-1,1])\times H^{k+2}([-1,1])$ for all $k\geq 0.$ Then we have proved that $(a,b)\in C^\infty([-1,1])$.

\subsubsection{Proof of uniqueness}
Until now we have shown that there exist $\lambda$ and  $(w_n^+,w_n^-)$ given by
\begin{align*}
w^+_n(z)=\left\{\begin{array}{cc} 0 & n\neq m,\\ a(z) & n=m, \end{array}\right. \quad \text{and} \quad  w^-_n(z)=\left\{\begin{array}{cc} 0 & n\neq m,\\ b(z) & n=m, \end{array}\right.
\end{align*}
with $(\lambda,a,b)\in \R\times C^\infty([-1,1])\times C^\infty([-1,1])$  solving \eqref{l+}, \eqref{l-} for all $n\geq 1$ and satisfying \eqref{asolucion}, \eqref{bsolucion}, \eqref{lambdasolucion} and \eqref{cotas}.

To finish the proof of Theorem \ref{eyufijadom} we need to check that the kernel of $\cL[\lambda]$, fixed that $\lambda$, is one dimensional. Therefore we have to prove that the solution  of system \eqref{l+}, \eqref{l-}, given by \eqref{w+solucion}, \eqref{w-solucion}, \eqref{asolucion}, \eqref{bsolucion}, for $\lambda$ given by \eqref{lambdasolucion}, is unique modulo multiplication by a constant.
In order to prove our goal, we will distinguish between two cases:
\begin{itemize}
	\item Case $n\neq m:$ Let us postpone this case to next section \ref{s:inverseoperator} (take $W^\pm_n=0$ on that section).
	
	\item Case $n=m:$  Let $(u,v)\in L^2([-1,1])\times L^2([-1,1])$ be a solution of \eqref{l+cero}, \eqref{l-cero}, with $\lambda$ given by \eqref{lambdasolucion}. Note that $(u,v)$ depends on $\ep$ although we do not make explicit this dependence.  Then, let us see that for some constant $C$ we have
\[
\begin{cases}
u(z)=C a(z),\\
v(z)=C b(z).
\end{cases}	
\]
\end{itemize}
The system \eqref{l+cero}, \eqref{l-cero} is linear in $(u,v)$ and we can assume without loss of generality that
$$\|u\|_{L^2([-1,1])}^2+\|v\|_{L^2([-1,1])}^2=1.$$
If it is not the case we only need to normalized the solution and enter that value into the final constant $C$.  Now, using \eqref{l+cero}, we get that $u(z)=\e u_1(z)$ with $\|u_1\|_{L^2([-1,1])}=O(1)$ in terms of $\e$. This information yields from \eqref{l-cero} that
\begin{align}\label{aqui}
(1+\lambda_1-z)v(z)+\frac{1}{2m}\int_{-1}^1\varphi_{\kappa}'(\bar{z}) v(\bar{z})\mbox{d}\bar{z}=F(z),
\end{align}
where $\|F\|_{L^2([-1,1])}=O(\e)$. %Multiplying \eqref{aqui} by $\varphi_{\kappa}'b_0$ and integrating between $-1$ and $1$ we obtain that $( F, \varphi'_{\kappa}b_0 )_{L^2}=0.$ Here we have used \eqref{eqb}.
Just dividing \eqref{aqui} by $(1+\lambda_1-z)$ we have
\begin{align}\label{general}
v(z)=\frac{C}{1+\lambda_1-z}+\frac{F(z)}{1+\lambda_1-z}.
\end{align}
%Plugging this expression  in \eqref{aqui} we find that \eqref{general} is a solution of \eqref{aqui} for all value of $\tilde{C}$.
In addition, recalling the precise form of $b_0(z)$ we obtain that
\begin{align}\label{vep}
v(z)=C b_0(z)+\e v_1(z),
\end{align}
where $\|v_1\|_{L^2([-1,1])}=O(1)$ in terms of $\e$ and $C$ is a constant.

Looking again \eqref{l+cero} we have that
\begin{align*}
2u_1(z)=C\frac{e^{-2m}}{2m}\int_{-1}^1 (-\varphi_{\kappa}'(\bar{z}))b_0(\bar{z})\mbox{d}\bar{z}+G(z),
\end{align*}
where $\|G\|_{L^2([-1,1])}=O(\e)$. Recalling the precise form of $a_1$, see \eqref{eqa}, forces to
\begin{align}\label{cu1}
u_1(z)=C a_1+\e u_2(z),
\end{align}
with $\|u_2\|_{L^2([-1,1])}=O(1)$ in terms of $\e$ and the same constant $C$ than in \eqref{vep}.

Finally, we get a coupled system for $(u_2,v_1)$ which is exactly the same as \eqref{eqaep2}, \eqref{eqbep2} for $( a_2^\ep,b_1^\ep)$ up to the  multiplicative constant $C$. Indeed, we have, from \eqref{vep} together with \eqref{b0formula} and \eqref{cu1} together with \eqref{eqa}, that
\begin{align*}
&2u_2(z)+\frac{e^{-2m}}{2m}\int_{-1}^1\varphi_\kappa(\bar{z})v_1(\bar{z})e^{-m\ep (z+\bar{z})}\mbox{d}\bar{z}=CA_0[a_1,b_0](z)+\ep A_1[u_2;\l_2^{\ep}, Ca_1, Cb_0](z),\\
&v_1(z)=b_0(z) (C B_0[a_1,b_0](z)+\ep B_1[u_2,v_1;\lambda_2^\ep](z)-C\lambda_2^\ep b_0(z)).
\end{align*}

 This system is linear in $(u_2,v_1)$ (now $\lambda$ is fixed). One can check that $C(a_2^\ep,b_1^\e)$ is a solution and by a fixed point argument must be unique. Then $(u_2, v_1)=C(a_2^\ep, b_1^\ep)$ (with the same constant $C$ as in \eqref{cu1} and \eqref{vep}) and consequently we have proved our goal.

%\newpage
\subsection{Codimension of the image of the linear operator}
Let $\l_{\e,\k,m}$ given by  Theorem \ref{eyufijadom}. In order to determine the codimension in  $Y$ of the linear operator $
\cL[\l_{\e,\k,m}] \equiv D_\ff F[\l_{\e,\k,m},0]$ on $X$, from expression \eqref{e:Linear} we see that we have to study the equation:
\begin{equation}\label{e:fullcodimension}
\cL[\l_{\ep,\kappa,m}]h=H,
\end{equation}
with $\hh\in X$ and $\HH\in Y.$  As we did before, we will use the expansions
\[
\hh(\bx)=\sum_{n=1}^\infty \hh_n(y)\cos(nx), \qquad \HH(\bx)=\sum_{n=1}^\infty \HH_n(y)\sin(nx),
\]
and, similarly to what we made in section \ref{s:simplifications}, we define the auxiliary functions
\begin{align*}\HH_n^{\pm}(y)&:=\HH_n(\pm y) \qquad \text{for} \quad  y\in[1-\e,1+\e],\\
h_n^\pm(y)& :=h_n(\pm y) \qquad\,\, \text{for} \quad  y\in[1-\e,1+\e],
\end{align*}
and
\begin{align*}
W_n^{\pm}(z)&:=\HH_n^{\pm}(1+\e z) \qquad \text{for} \quad  z\in[-1,1],\\
w_n^{\pm}(z)&:=h_n^{\pm}(1+\e z) \qquad\, \, \text{for} \quad  z\in[-1,1].
\end{align*}

Then, for any $\ep>0$, the functional equation \eqref{e:fullcodimension} is equivalent to solve the system
\begin{multline}\label{w+W}
-n(\lambda_m+1-\ep+\ep z-\ep^2\Phi_{\kappa}(z))w_n^+(z)\\
+\frac{\ep}{2}\int_{-1}^1\varphi'_{\kappa}(\bar{z})w_n^+(\bar{z})e^{- \ep n |z-\bar{z}|}\mbox{d}\bar{z}
-\frac{\ep e^{-2n}}{2}\int_{-1}^1\varphi'_{\kappa}(\bar{z})w_n^-(\bar{z})e^{-\ep n(z+\bar{z})}\mbox{d}\bar{z}= W^+_n(z),
\end{multline}
\begin{multline}\label{w-W}
-n(\lambda_m-1+\ep-\ep z+\ep^2 \Phi_{\kappa}(z))w_n^-(z)\\
+\frac{\ep e^{-2n}}{2}\int_{1}^1\varphi_{\kappa}'(\bar{z})w_n^+(\bar{z})e^{-\ep n(z+\bar{z})}\mbox{d}\bar{z}-\frac{\ep}{2}\int_{-1}^1\varphi'_{\kappa}(\bar{z})w_n^-(\bar{z})e^{-\ep n|z-\bar{z}|}\mbox{d}\bar{z}= W^-_n(z),
\end{multline}
%\begin{align}
%&-n(\lambda_m+1-\ep+\ep z-\ep^2\Phi_{\kappa}(z))w_n^+(z)\\
%&+\frac{\ep}{2}\int_{-1}^1\varphi'_{\kappa}(z')w_n^+(z')e^{-\ep n|z-z'|}dz'
%-\frac{\ep}{2}e^{-2n}\int_{-1}^1\varphi'_{\kappa}(z')w_n^-(z')e^{-\ep n(z+z')}dz'\label{w+W}\\
%=& W^+_n,\nonumber\\
%&-n(\lambda_m-1+\ep-\ep z+\ep^2 \Phi_{\kappa})w_n^-\\
%&+\frac{\ep}{2}e^{-2m}\int_{1}^1\varphi_{\kappa}'(z')w_n^+(z')e^{-\ep n(z+z')}-\frac{\ep}{2}\int_{-1}^1\varphi'_{\kappa}(z')w_n^-(z')e^{-\ep n|z-z'|}dz'\label{w-W}\\
%=& W^-_n.\nonumber,
%\end{align}
for all $n\geq 1.$ Here $\lambda_m\equiv\lambda_{\ep,\kappa,m}$. In what follow we will omit the dependence on both $\ep$ and $\kappa$ of $\lambda_{\ep,\kappa,m}$ and keep that one on $m$.
Let us define the function
\begin{align*}
\Lambda^+_m(z):=&\lambda_m+1+\ep(-1+z)-\ep^2\Phi_{\kappa}(z),\\
\Lambda^-_m(z):=&\lambda_m-1+\ep(+1-z)+\ep^2\Phi_{\kappa}(z),
\end{align*}
and the operators
\begin{align}
\label{t+}&T_n^+[u,v](z):=&\Lambda_m^+(z)u(z)-\frac{\ep}{2n}\int_{-1}^1\varphi_{\kappa}'(\bar{z})u(\bar{z})e^{-\ep n|z-\bar{z}|}\mbox{d}\bar{z}+\frac{\ep e^{-2n}}{2n}\int_{-1}^{1}\varphi_{\kappa}'(\bar{z})v(\bar{z})e^{-n\ep(z+\bar{z})}\mbox{d}\bar{z},
\\&\label{t-}T_n^-[u,v](z):=&\Lambda_m^-(z)v(z)+\frac{\ep}{2n}\int_{-1}^1\varphi_{\kappa}'(\bar{z})v(\bar{z})e^{-\ep n|z-\bar{z}|}\mbox{d}\bar{z}-\frac{\ep e^{-2n}}{2n}\int_{-1}^{1}\varphi_{\kappa}'(\bar{z})u(\bar{z})e^{-n\ep(z+\bar{z})}\mbox{d}\bar{z}.
\end{align}
Note that $T_n^+$ and $T_n^-$ are nothing but $\tfrac{1}{n}L^+_n[\lambda_m]$ and $\tfrac{1}{n}L^-_n[\lambda_m]$ respectively, with $L^+_n$ and $L^-_n$ given by \eqref{l+} and \eqref{l-}. Therefore, the system \eqref{w+W}, \eqref{w-W} is equivalent to
\begin{align}
\label{t+0}T_n^+[w^+_n,w^-_n](z)=&-\frac{1}{n}W_n^+(z),\\
\label{t-0}T_n^-[w^+_n,w^-_n](z)=&-\frac{1}{n}W_n^-(z),
\end{align}
for all $n\geq 1.$

\begin{lemma} \label{adjointness}Let $f,g,u,v\in L^2([-1,1])$ then
\begin{align*}
-(T_n^+[u,v],\varphi'_{\kappa} f)_{L^2}+(T^-_n[u,v],\varphi_{\kappa}'g)_{L^2}=-(\varphi_{\kappa}'u, T^+_n[f,g])_{L^2}+(\varphi'_{\kappa}v,T^-_n[f,g])_{L^2}.
\end{align*}
\end{lemma}
\begin{proof}
Direct computation.
\end{proof}
Lemma \ref{adjointness} implies the following necessary condition  for  \eqref{t+0}, \eqref{t-0} in order to be solvable:
\begin{lemma}\label{necesario}
Let $(a,b)\in L^2([-1,1])$ given by \eqref{asolucion} and \eqref{bsolucion}. That is, the pair $(a,b)$ solves
\begin{align*}
T^+_m[a,b]=0, \qquad T^-_m[a,b]=0.
\end{align*}
Then
\begin{align*}
-(W^+_m,\varphi_{\kappa}' a)_{L^2}+(W^-_m, \varphi_{\kappa}' b)_{L^2}=0
\end{align*}
\end{lemma}
\begin{proof}
Direct computation.
\end{proof}

Next  we state and prove the main theorem of this section:

\begin{theorem}\label{codimension} Let $M>1$. There exist $\kappa_0=\kappa_0(M)$ and $\ep_0=\ep_0(M)$ such that for all $0\leq \kappa<\kappa_0$, $0<\ep<\ep_0$ and $1\leq m < M$, the following statement hold:\\
\noindent
Let $(\lambda_m,a,b)\in \R \times L^2([-1,1])\times L^2([-1,1])$ given by \eqref{asolucion}, \eqref{bsolucion} and \eqref{lambdasolucion} and $\{W_n^\pm\}_{n=1}^\infty$  such that
\begin{align*}
\sum_{j=0}^3\sum_{k=0}^{3-j}\sum_{n=1}^\infty n^{2k}\|\pa^j_z W_{n}^\pm \|^2_{L^2([-1,1])}<\infty.
\end{align*}
and
\begin{align*}
-(W^+_m,\varphi_{\kappa}' a)_{L^2}+(W^-_m, \varphi_{\kappa}' b)_{L^2}=0.
\end{align*}
Then,  there exist $\{w^\pm_n\}_{n=1}^\infty$, with
\begin{align*}
\sum_{j=0}^3\sum_{k=0}^{4-j}\sum_{n=1}^\infty n^{2k}\|\pa^j_z w_{n}^\pm \|^2_{L^2([-1,1])}<\infty,
\end{align*}
satisfying \eqref{t+0} and \eqref{t-0} for all $n\geq 1.$
\end{theorem}
\begin{proof} The proof will be split into two cases. The case $n=m$ and the case $n\neq m$.
\subsubsection{The case $n= m$.}
To deal with this case we introduce the space $L^2([-1,1])$ with weight $-\varphi_{\kappa}'.$ That is, $L^2_{\varphi'_{\kappa}}([-1,1])$ with inner-product and  norm given by
$$(u,v)_{L^2_{\varphi'_{\kappa}}}:=\int_{-1}^{1}(-\varphi'_{\kappa}(\bar{z})) u(\bar{z}) v(\bar{z}) \mbox{d}\bar{z},$$
and
$$\|v\|^2_{L^2_{\varphi'_{\kappa}}}:=(v,v)_{L^2_{\varphi'_{\kappa}}}=\int_{-1}^{1}(-\varphi'_{\kappa}(\bar{z}))|v(\bar{z})|^2 \mbox{d}\bar{z}.$$
In addition, we define the bilinear forms
\begin{align*}
B[(u_1,v_1),(u_2,v_2)]&:= (T^+_m[u_1,v_1],u_2)_{\ldosv}+(T^-_m[u_1,v_1],v_2)_{\ldosv},\\
\mathcal{B}[v_1,v_2]&:= \int_{-1}^1(1+\lambda_1-\bar{z})v_1(\bar{z})v_2(\bar{z})(-\varphi'_{\kappa}(\bar{z}))\mbox{d}\bar{z}\\
&\quad -\frac{1}{2m}\left(\int_{-1}^1(-\varphi'_{\kappa}(\bar{z}))v_1(\bar{z})\mbox{d}\bar{z}\right)\left(\int_{-1}^{1}(-\varphi'_{\kappa}(\bar{z}))v_2(\bar{z})\mbox{d}\bar{z}\right).
\end{align*}
We start with the properties of the bilinear form $\mathcal{B}$. Notice that by definition \eqref{b0formula} of $b_0$ we have
\[
\mathcal{B}[b_0,b_0]=0.
\]
We also define the orthogonal complement of $b_0$ in $\ldosv([-1,1]):$
\[
b_0^{\perp}:=\lbrace v\in \ldosv([-1,1]) \mid (b_0,v)_{\ldosv}=0 \rbrace.
\]
\begin{lemma}\label{goticocoercivo} Let $v_1, v_2\in L^2_{\varphi'_{\kappa}}([-1,1])$. Then, the following estimate holds
\begin{align*}
|\mathcal{B}[v_1,v_2]|&\leq C \|v_1\|_{\ldosv}\|v_2\|_{\ldosv},
\end{align*}
where the constant $C$ is independent of $\ep, \k, m.$ Moreover, if $v\in b_0^\perp$ we have
\begin{align*}
\mathcal{B}[v,v]\geq c(M)||v||_{\ldosv}^2.
\end{align*}
\end{lemma}
\begin{proof}
The first conclusion is trivial. For the other one we need to proceed as follow. In first place, we fix an arbitrary $v\in b_0^{\perp}$, i.e. $v\in \ldosv([-1,1])$ such that $$\int_{-1}^1\varphi'_{\kappa}(z)v(z)b_0(z)\dz=0.$$
Notice that by definition
\begin{align}\label{Bvv}
\mathcal{B}[v,v]=\int_{-1}^1(-\varphi'_{\kappa}(z))\frac{|v(z)|^2}{b_0(z)}dz-\frac{1}{2m}\left(\int_{-1}^1 (-\varphi_{\kappa}'(z))v(z) \dz\right)^2.
\end{align}
In addition, for all $\mu\in \R$, we have that
\begin{align*}
\int_{-1}^1(-\varphi_\kappa'(z))v(z)\dz&=\int_{-1}^1(-\varphi_{\kappa}'(z))(1-\mu b_0(z))v(z)\dz
=\int_{-1}^1\frac{(-\varphi_{\kappa}'(z))}{\sqrt{b_0(z)}}\sqrt{b_0(z)}\left(1-\mu b_0(z)\right)v(z)\dz\\
&\leq \left(\int_{-1}^1 (-\varphi_{\kappa}'(z)) \frac{|v(z)|^2}{b_0(z)}\dz\right)^\frac{1}{2}\left(\int_{-1}^1(-\varphi_{\kappa}'(z))b_0(z)(1-\mu b_0(z))^2 \dz\right)^\frac{1}{2}.
\end{align*}
Thus
\begin{align*}
\frac{1}{2m}\left(\int_{-1}^1(-\varphi_\kappa'(z))v(z)\dz\right)^2\leq \frac{1}{2m}\left(\int_{-1}^1(-\varphi_{\kappa}'(z))b_0(z)(1-\mu b_0(z))^2 \dz\right) \left(\int_{-1}^1 (-\varphi_{\kappa}'(z)) \frac{|v(z)|^2}{b_0(z)} \dz\right)
\end{align*}
and consequently
\begin{align*}
\mathcal{B}[v,v]\geq \left(1-\frac{1}{2m}\int_{-1}^1(-\varphi'_{\kappa}(z)) b_0(z)(1-\mu b_0(z))^2 \dz\right)\int_{-1}^1(-\varphi'_{\kappa}(z))(1+\lambda_1-z)|v(z)|^2 \dz.
\end{align*}
Expanding the square  we have
\begin{align*}
\int_{-1}^1(-\varphi'_{\kappa}(z))b_0(z)(1-\mu b_0(z))^2 \dz =2m-2\mu \int_{-1}^1(-\varphi_{\kappa}'(z))|b_0(z)|^2 \dz+\mu^2\int_{-1}^1(-\varphi'_{\kappa}(z))(b_0(z))^3 \dz,
\end{align*}
and therefore
\begin{multline*}
1-\frac{1}{2m}\int_{-1}^1(-\varphi'_{\kappa}(z))b_0(z)(1-\mu b_0(z))^2 \dz\\
=\frac{\mu}{2m}\left(2\int_{-1}^1(-\varphi_{\kappa}'(z))|b_0(z)|^2 \dz-\mu\int_{-1}^1(-\varphi'_{\kappa}(z))(b_0(z))^3 \dz\right).
\end{multline*}
Choosing $$\mu=\frac{\int_{-1}^1 (-\varphi'_{\kappa}(z))(b_0(z))^2 \dz}{\int_{-1}^1(-\varphi'_{\kappa}(z))(b_0(z))^3 \dz},$$
we obtain the lower bound
\begin{align*}
1-\frac{1}{2m}\int_{-1}^1 (-\varphi'_{\kappa}(z)) b_0(z)(1-\mu b_0(z))^2 \dz=\frac{\mu}{2m}\int_{-1}^1(-\varphi'_{\kappa}(z))|b_0(z)|^2 \dz\geq c(M).
\end{align*}
Finally, as $\l_1\geq \l_1^{\ast}(M)$ we get
\begin{align*}
\mathcal{B}[v,v]\geq c(M)\int_{-1}^1 (-\varphi_{\kappa}'(z)) (1+\lambda_1-z)|v(z)|^2 \dz \geq c(M)\lambda_1 \|v\|_{\ldosv}^2\geq c(M)\|v\|_{\ldosv}^2.
\end{align*}
\end{proof}
Next we shall study the bilinear for $B[(u_1,v_1),(u_2,v_2)]$.
\begin{lemma} Let $u,v\in L^2_{\varphi'_{\kappa}}([-1,1])$. Then, the following expression holds
\begin{align}\label{breduccion}
B[(u,v),(u,v)]=& 2\int_{-1}^1(-\varphi'_{\kappa}(z))|u(z)|^2 \dz+\ep \mathcal{B}[v,v] +\ep B_1[u,u]+\ep^2 B_2[v,v],
\end{align}
with $B_1,B_2$ given by \eqref{B1uu}, \eqref{B2vv} and satisfying
\begin{align*}
|B_1[u,u]|\leq C \|u\|_{\ldosv}^2, \qquad |B_2[v,v]|\leq C \|v\|^2_{\ldosv},
\end{align*}
where the constant $C$ is independent of $\ep, \k, m.$
\end{lemma}
\begin{proof}
From $\eqref{t+}$ and \eqref{t-} we find that
\begin{align*}
B[(u_1,v_1), (u_2,v_2)]=&\int_{-1}^1(-\varphi_{\kappa}'(z))\Lambda^+_m(z)u_1(z)u_2(z) \dz\\
&+\frac{\ep}{2m}\int_{-1}^1\int_{-1}^1\varphi_{\kappa}'(z)\varphi_{\kappa}'(\bar{z})u_1(z)u_2(\bar{z})e^{-\ep n|z-\bar{z}|}\dz \mbox{d}\bar{z}\\
&-\frac{\ep e^{-2m}}{2m}\int_{-1}^1\int_{-1}^1\varphi_{\kappa}'(z)\varphi_{\kappa}'(\bar{z})v_1(z)u_2(\bar{z})e^{-\ep m(z+\bar{z})}\dz \mbox{d}\bar{z}\\
&+\int_{-1}^1(-\varphi_{\kappa}'(z))\Lambda^-_m(z)v_1(z)v_2(z)\dz\\
&-\frac{\ep}{2m}\int_{-1}^1\int_{-1}^1 \varphi_{\kappa}'(z)\varphi_{\kappa}'(\bar{z})v_1(z)v_2(\bar{z})e^{-\ep m|z-\bar{z}|}\dz \mbox{d}\bar{z}\\
&+\frac{\ep e^{-2m}}{2m}\int_{-1}^1\int_{-1}^1 \varphi_{\kappa}'(z)\varphi_{\kappa}'(\bar{z})u_1(z)v_2(\bar{z})e^{-\ep m(z+\bar{z})}\dz \mbox{d}\bar{z}.
\end{align*}
%\begin{align*}
%B[(u_1,v_1), (u_2,v_2)]=&\int_{-1}^1(-\varphi_{\kappa}'(z))\Lambda^+_m(z)u_1(z)u_2(z)dz\\
%&+\frac{\ep}{2n}\int_{-1}^1\int_{-1}^1\varphi_{\kappa}'(z)\varphi_{\kappa}'(z')u_1(z)u_2(z')e^{-\ep n|z-z'|}dzdz'\\
%&-\frac{\ep e^{-2m}}{2n}\int_{-1}^1\int_{-1}^1\varphi_{\kappa}'(z)\varphi_{\kappa}'(z')u_1(z)v_2(z')e^{-\ep n(z+z')}dzdz'\\
%&+\int_{-1}^1(-\varphi_{\kappa}'(z))\Lambda^-_m(z)v_1(z)v_2(z)dz\\
%&-\frac{\ep}{2n}\int_{-1}^1\int_{-1}^1\varphi_{\kappa}'(z)\varphi_{\kappa}'(z')v_1(z)v_2(z')e^{-\ep n|z-z'|}dzdz'\\
%&+\frac{\ep e^{-2m}}{2n}\int_{-1}^1\int_{-1}^1\varphi_{\kappa}'(z)\varphi_{\kappa}'(z')u_1(z)v_2(z')e^{-\ep n(z+z')}dzdz'
%\end{align*}
%\textcolor{blue}{SE CANCELAN LOS TERMINOS CRUZADOS.}
From the above expression we can check that
\begin{align}\label{Bu1v1u2v2}
|B[(u_1,v_1),(u_2,v_2)]|\leq C \left(\|u_1\|_{\ldosv}+\|v_1\|_{\ldosv}\right)\left(\|u_2\|_{\ldosv}+\|v_2\|_{\ldosv}\right),
\end{align}
where the constant $C$ is independent of $\ep, \k, m.$ In addition, if $(u_1,v_1)=(u,v)=(u_2,v_2)$ the cross terms cancels and we obtain
\begin{align*}
B[(u,v),(u,v)]=&\int_{-1}^1(-\varphi_{\kappa}'(z))\Lambda^+_m(z)|u(z)|^2 \dz\\
&+\frac{\ep}{2m}\int_{-1}^1\int_{-1}^1\varphi_{\kappa}'(z)\varphi_{\kappa}'(\bar{z})u(z)u(\bar{z})e^{-\ep m|z-\bar{z}|}\dz \mbox{d}\bar{z}\\
%&-\frac{\ep e^{-2m}}{2n}\int_{-1}^1\int_{-1}^1\varphi_{\kappa}'(z)\varphi_{\kappa}'(z')u(z)v(z')e^{-\ep n(z+z')}dzdz'\\
&+\int_{-1}^1(-\varphi_{\kappa}'(z))\Lambda^-_m(z)|v(z)|^2 \dz\\
&-\frac{\ep}{2m}\int_{-1}^1\int_{-1}^1\varphi_{\kappa}'(z)\varphi_{\kappa}'(\bar{z})v(z)v(\bar{z})e^{-\ep m|z-\bar{z}|}\dz \mbox{d}\bar{z}\\
&:=I+II.
%&+\frac{\ep e^{-2m}}{2m}\int_{-1}^1\int_{-1}^1\varphi_{\kappa}'(z)\varphi_{\kappa}'(\bar{z})u(z)v(\bar{z})e^{-\ep m(z+\bar{z})}\dz \mbox{d}\bar{z},
\end{align*}
where
\begin{align*}
I&=\int_{-1}^1(-\varphi_{\kappa}'(z))\Lambda^+_m(z)|u(z)|^2 \dz+\frac{\ep}{2m}\int_{-1}^1\int_{-1}^1\varphi_{\kappa}'(z)\varphi_{\kappa}'(\bar{z})u(z)u(\bar{z})e^{-\ep m|z-\bar{z}|}\dz \mbox{d}\bar{z},\\
II&=\int_{-1}^1(-\varphi_{\kappa}'(z))\Lambda^-_m(z)|v(z)|^2 \dz-\frac{\ep}{2m}\int_{-1}^1\int_{-1}^1\varphi_{\kappa}'(z)\varphi_{\kappa}'(\bar{z})v(z)v(\bar{z})e^{-\ep m|z-\bar{z}|}\dz \mbox{d}\bar{z}.
\end{align*}
We recall that
\begin{align}
\Lambda^+_m(z)&=2+(-1+\lambda_1+z)\ep+(\lambda_2^\ep-\Phi_{\kappa}(z))\ep^2,\nonumber\\
\Lambda^-_m(z)&=(1+\lambda_1-z)\ep +(\lambda_2^\ep+\Phi_{\kappa}(z))\ep^2. \label{Lm^-good}
\end{align}
Thus
\begin{align*}
I&=2 \int_{-1}^1(-\varphi_{\kappa}'(z))|u(z)|^2 \dz+ \ep \left(\int_{-1}^1(-\varphi_{\kappa}'(z))\left[(-1+\lambda_1+z)+\ep(\lambda_2^\ep-\Phi_{\kappa}(z))\right]|u(z)|^2 \dz \right)\\
&\quad +\frac{\ep}{2m}\int_{-1}^1\int_{-1}^1\varphi_{\kappa}'(z)\varphi_{\kappa}'(\bar{z})u(z)u(\bar{z})e^{-\ep m|z-\bar{z}|}\dz \mbox{d}\bar{z},
\end{align*}
and
\begin{align*}
II&=\ep\int_{-1}^1(-\varphi_{\kappa}'(z))\frac{|v(z)|^2}{b_0(z)} \dz+\ep^2\int_{-1}^1\varphi'_{\kappa}(z)(\lambda_2^\ep+\Phi_{\kappa}(z))|v(z)|^2\dz\pm\frac{\ep}{2m}\left(\int_{-1}^1\varphi'_{\kappa}(z)v(z)\dz\right)^2\\
&=\ep\mathcal{B}[v,v]+\ep^2\int_{-1}^1\varphi'_{\kappa}(z)(\lambda_2^\ep+\Phi_{\kappa}(z))|v(z)|^2\dz\\
&\quad -\frac{\ep^2}{2m}\int_{-1}^1\int_{-1}^1\varphi'_{\kappa}(z)\varphi'_{\kappa}(\bar{z})v(z)v(\bar{z})\frac{e^{-\ep m|z-\bar{z}|}-1}{\ep}\dz \mbox{d}\bar{z},
\end{align*}
where in the last step we have used the expression \eqref{Bvv}.
%\begin{align*}
%&\int_{-1}^1(-\varphi'_{\kappa}(z))\Lambda^-_m(z)|v(z)|^2 \dz-\frac{\ep}{2m}\int_{-1}^1\int_{-1}^1\varphi'_{\kappa}(z)\varphi'_{\kappa}(\bar{z})v(z)v(\bar{z})e^{-\ep m|z-\bar{z}|}\dz \mbox{d}\bar{z}\\
%&=\ep\mathcal{B}[v,v]+\ep^2\int_{-1}^1\varphi'_{\kappa}(z)(\lambda_2^\ep+\Phi_{\kappa}(z))|v(z)|^2\dz\\
%&\quad -\frac{\ep^2}{2m}\int_{-1}^1\int_{-1}^1\varphi'_{\kappa}(z)\varphi'_{\kappa}(\bar{z})v(z)v(\bar{z})\frac{e^{-\ep m|z-\bar{z}|}-1}{\ep}\dz \mbox{d}\bar{z}.
%\end{align*}

Putting all together, we have
\begin{align*}
B[(u,v),(u,v)]=& 2\int_{-1}^1(-\varphi'_{\kappa}(z))|u(z)|^2 \dz+\ep \mathcal{B}[v,v] +\ep B_1[u,u]+\ep^2 B_2[v,v],
\end{align*}
where
\begin{align}
B_1[u,u]:=&\int_{-1}^1\varphi'_{\kappa}(z)\left[(-1+\lambda_1+z)+\ep(\lambda^\ep_2-\Phi_{\kappa}(z))\right]|u(z)|^2 \dz\label{B1uu}\\
&+\frac{1}{2m}\int_{-1}^1\int_{-1}^1\varphi'_{\kappa}(z)\varphi'_{\kappa}(\bar{z})u(z)u(\bar{z})e^{-m\ep|z-\bar{z}|}\dz \mbox{d}\bar{z},\nonumber\\
B_2[v,v]:=&\int_{-1}^1\varphi'_{\kappa}(z)(\lambda_2^\ep+\Phi_{\kappa}(z))|v(z)|^2 \dz\label{B2vv}\\
&-\frac{\ep^2}{2m}\int_{-1}^1\int_{-1}^1\varphi'_{\kappa}(z)\varphi'_{\kappa}(\bar{z})v(z)v(\bar{z})\frac{e^{-\ep m|z-\bar{z}|}-1}{\ep} \dz \mbox{d}\bar{z}.\nonumber
\end{align}
Finally, one can see that
\begin{align*}
|B_1[u,u]|\leq C \|u\|_{\ldosv}^2, && |B_2[v,v]|\leq C \|v\|^2_{\ldosv},
\end{align*}
for some constant $C$, which is independent of $\ep, \k, m.$
\end{proof}

Combining all the previous results, we are in a good position to prove the following lemma. Before that, we recall the solution $(a,b)$ defined in previous section \ref{s:existence} and given by \eqref{asolucion}, \eqref{bsolucion}.
We also define the orthogonal complement of $(-a,b)$ in $\ldosv([-1,1])\times\ldosv([-1,1])$:
\[
(-a,b)^\perp:=\left\lbrace (u,v)\in \ldosv([-1,1])\times\ldosv([-1,1]) \mid -(u,a)_{\ldosv}+(v,b)_{\ldosv}=0 \right\rbrace.
\]
\begin{lemma}\label{coercivo} Let $M>1.$ There exist $\kappa_0=\kappa_0(M)$, $\ep_0=\ep_0(M)$ such that for all $0<\ep<\ep_0$, $0\leq \kappa<\kappa_0$ and $1\leq m<M$, the following estimates hold
\begin{align}
|B[(u_1,v_1),(u_2,v_2)]|&\leq C \left(\|u_1\|_{\ldosv}+\|v_1\|_{\ldosv}\right)\left(\|u_2\|_{\ldosv}+\|v_2\|_{\ldosv}\right),\nonumber\\
B[(u,v),(u,v)]&\geq c(M)\left(\|u\|_{\ldosv}^2+\ep \|v\|_{\ldosv}^2\right),\label{lowerboundBuv}
\end{align}
for all $(u_1,v_1), (u_2,v_2)\in \ldosv([-1,1])\times \ldosv([-1,1])$ and for all $(u,v)\in (-a,b)^\perp.$
\end{lemma}
\begin{proof} Notice that the first estimate had been proved in \eqref{Bu1v1u2v2}. To prove \eqref{lowerboundBuv} we start using the fact that $(u,v)\in(-a,b)^{\perp},$ which implies that $$(v,b)_{\ldosv}=(u,a)_{\ldosv}.$$
Recalling expressions \eqref{asolucion} and \eqref{bsolucion} we see that
\begin{align*}
(v,b_0)_{\ldosv}=\ep(u,a_1)_{\ldosv}+\ep^2(u,a_{2}^\ep)_{\ldosv}-\ep(v,b_1^\ep)_{\ldosv}.
\end{align*}
Thus we can assume that $v=v_0+\ep v_1$, where $v_0\in b_0^\perp$ and $v_1\in \text{span}\{b_0\}$. So, we have
\[
(v_1,b_0)_{\ldosv}=(u,a_1)_{\ldosv}+\ep(u,a_{2}^\ep)_{\ldosv}-(v,b_1^\ep)_{\ldosv}.
\]
By the explicit expressions of $b_0$ and $a_1$ together with the upper bounds \eqref{cotas}, we obtain
\begin{equation}\label{auxv1}
\|v_1\|_{\ldosv}\leq C(M)(\|u\|_{\ldosv}+\|v\|_{\ldosv}).
\end{equation}
By the definition of $v$ as the sum of two orthogonal functions we have
\begin{equation}\label{auxv0}
\|v\|_{\ldosv}^2=\|v_0\|_{\ldosv}^2+\ep^2 \|v_1\|_{\ldosv}^2.
\end{equation}
In addition, combining \eqref{auxv1}, \eqref{auxv0} we get
\begin{align}\label{r}
\|v_0\|_{\ldosv}^2&=\|v\|_{\ldosv}^2-\ep^2\|v_1\|_{\ldosv}^2\geq \|v\|^2_{\ldosv}-C(M)\ep^2(\|u\|^2_{\ldosv}+\|v\|_{\ldosv}^2)
\\&\geq c(M)\|v\|^2_{\ldosv}-C(M)\ep^2\|u\|^2_{\ldosv}\nonumber,
\end{align}
for $\ep$ small enough.

We then can estimate
\begin{align}\label{bcoercivocota}
\mathcal{B}[v,v]=&\mathcal{B}[v_0,v_0]+\ep\mathcal{B}[v_0,v_1]+\ep \mathcal{B}[v_1,v_0]+\ep^2\mathcal{B}[v_1,v_1]
\\ \geq & c(M)\|v_0\|^2_{\ldosv}-C\ep(\|u\|_{\ldosv}^2+\|v\|_{\ldosv}^2)\nonumber
\\ \geq & c(M)\|v\|^2_{\ldosv}-C(M)\ep \|u\|_{\ldosv}^2,\nonumber
\end{align}
where we have used Lemma \ref{goticocoercivo}, inequality \eqref{r} and taken $\ep$ small enough.

Finally, from \eqref{bcoercivocota} and \eqref{breduccion} we achieve the conclusion of the lemma. Recall that \eqref{breduccion} tell us that
\begin{align*}
B[(u,v),(u,v)]&= 2\|u\|_{\ldosv}^2+\ep \mathcal{B}[v,v] -C\ep \|u\|_{\ldosv}^2 -C\ep^2 \|v\|^2_{\ldosv}\\
&\geq 2\|u\|_{\ldosv}^2 +c(M)\ep\|v\|^2_{\ldosv}-C(M)\ep^2 \|u\|_{\ldosv}^2-C\ep \|u\|_{\ldosv}^2 -C\ep^2 \|v\|^2_{\ldosv}\\
&\geq c(M)\left(\|u\|_{\ldosv}^2+\ep \|v\|_{\ldosv}^2\right),
\end{align*}
where in the last step we are taking $\ep$ small enough.
\end{proof}

Then we have all the ingredients to prove the main result of this section. We just need to consider the functional equation
\begin{align}\label{funcional}
B[(u_1,v_1),(u,v)]=-\frac{1}{m}(W^+_m,u)_{\ldosv}-\frac{1}{m}(W^-_m,v)_{\ldosv},\quad \text{for all $(u,v)\in (-a,b)^\perp$.}
\end{align}
Combining Lemma \ref{coercivo} and Lax-Milgram theorem, there exists $(u_1,v_1)\in (-a,b)^\perp$ satisfying \eqref{funcional}. This implies that there exist $\gamma\in \R$ such that
\begin{align*}
(T^+_m[u_1,v_1], T^-_m[u_1,v_1])=-\frac{1}{m}(W^+_m,W^-_m)+\gamma (-a,b).
\end{align*}
But then, taking the inner product on $\ldosv([-1,1])\times\ldosv([-1,1])$ against $(-a,b)$, we get
\[
-(T^+_m[u_1,v_1],a)_{\ldosv}+(T^-_m[u_1,v_1],b)_{\ldosv}=-\frac{1}{m}(-(W^+_m,a)_{\ldosv}+(W^-_m,b)_{\ldosv})+\gamma(\|a\|_{\ldosv}^2+\|b\|_{\ldosv}^2).
\]
Since $T_m^{+}[a,b]=0$ and $T_m^{-}[a,b]=0$, we have that
\begin{align*}
-(T^+_m[u_1,v_1],a)_{\ldosv}+(T^-_m[u_1,v_1],b)_{\ldosv}&=0, \qquad \text{by Lemma } \ref{adjointness},\\
-\frac{1}{m}(-(W^+_m,a)_{\ldosv}+(W^-_m,b)_{\ldosv})&=0, \qquad \text{by Lemma } \ref{necesario},
\end{align*}
%\begin{align*}
%\underbrace{-(T^+_m[u_1,v_1],a)_{\ldosv}+(T^-_m[u_1,v_1],b)_{\ldosv}}_{=0}=\underbrace{-\frac{1}{m}(-(W^+_m,a)_{\ldosv}+(W^-_m,b)_{\ldosv})}_{=0}+\gamma(||a||_{\ldosv}^2+||b||_{\ldosv}^2).
%\end{align*}
which implies that $\gamma=0$. Therefore, there exist $(u_1,v_1)\in \ldosv([-1,1])\times\ldosv([-1,1])$ satisfying
\begin{align}\label{TW}
(T^+_m[u_1,v_1], T^-_m[u_1,v_1])=-\frac{1}{m}(W^+_m,W^-_m).
\end{align}
Next we shall improve the regularity of $(u_1,v_1)$. By \eqref{lowerboundBuv} and \eqref{funcional}, we have that
\begin{align*}
\|u_1\|_{\ldosv}^2+\ep \|v_1\|^2_{\ldosv}&\leq C(M)B[(u_1,v_1),(u_1,v_1)]=-\frac{C(M)}{m}\left((W^+_m,u_1)_{\ldosv}+(W^-_m,v_1)_{\ldosv}\right)\\
&\leq C(M)\left(\|W^+_m \|_{\ldosv}\|u_1\|_{\ldosv}+\|W^-_m\|_{\ldosv}\|v_1\|_{\ldosv}\right)\\
&\leq C(M)\left(\|W^+_m\|^2_{\ldosv}+\frac{1}{\ep}\|W^-_m\|_{\ldosv}^2\right)+\frac{1}{2}\left(\|u_1\|_{\ldosv}^2+\ep\|v_1\|_{\ldosv}^2\right),
\end{align*}
where in the last step we have used the generalized Young inequality. Since $\varphi_{\k}'\in L^{\infty}([-1,1]),$ we get the right-hand side of the following expression in the unweighted $L^2([-1,1])-$norm.
\begin{align}\label{therefore}
\|u_1\|_{\ldosv}^2+\ep\|v_1\|_{\ldosv}^2\leq C(M)\left(\|W^+_m\|_{L^2}^2+\frac{1}{\ep}\|W^-_m\|_{L^2}^2\right).
\end{align}
Let us look to  \eqref{t-0}. We can write this equation in the following way
\begin{align}\label{auxv}
v_1(z)=&-\frac{\ep}{2m\Lambda_m^-(z)}\int_{-1}^1\varphi_{\kappa}'(\bar{z})v_1(\bar{z})e^{-\ep m|z-\bar{z}|}\mbox{d}\bar{z}\\
&+\frac{\ep}{2m\Lambda_m^-(z)}e^{-2m}\int_{-1}^1\varphi_{\kappa}'(\bar{z})u_1(\bar{z})e^{-\ep m (z+\bar{z})}\mbox{d}\bar{z} -\frac{1}{m\Lambda^-_m(z)}W^-_m(z).\nonumber
\end{align}
Recalling \eqref{Lm^-good} we have
\begin{equation}\label{bound1/Lambda}
\| 1/\Lambda_m^-\|_{L^\infty([-1,1])}\leq \frac{C(M)}{\e}, \qquad \text{and} \qquad  \|\ep/\Lambda_m^-\|_{L^\infty([-1,1])}\leq C(M).
\end{equation}
Then, taking the $L^2([-1,1])$ norm of \eqref{auxv} yields
\begin{align*}
\|v_1\|_{L^2}\leq C(M)\left(\|v_1\|_{\ldosv}+\|u_1\|_{\ldosv}+\frac{1}{\ep}\|W^-_m\|_{L^2}\right),
\end{align*}
and using inequality \eqref{therefore}  give us
\begin{align*}
\|v_1\|_{L^2}\leq &\frac{C(M)}{\sqrt{\ep}}\left(\|W^+_m \|_{L^2}+\frac{1}{\sqrt{\ep}}\|W^-_m \|_{L^2}\right)+C(M)\left(\|W^+_m \|_{L^2}
+\frac{1}{\sqrt{\ep}}\|W^-_m \|_{L^2}\right)+\frac{C(M)}{\ep}\|W^-_m \|_{L^2}\\
\leq & \frac{C(M)}{\ep}\left(\|W^+_m\|_{L^2}+\|W^-_m\|_{L^2}\right).
\end{align*}
Once we have proved that $v_1\in L^2([-1,1])$ we only have to take derivatives on expression \eqref{auxv}, but before that, we observe that
\begin{equation}\label{bound1/Lambdageneral}
\| 1/\Lambda_m^-\|_{\dot{W}^{k,\infty}([-1,1])}\leq \frac{C(M)}{\e}, \qquad \text{and} \qquad  \|\ep/\Lambda_m^-\|_{\dot{W}^{k,\infty}([-1,1])}\leq C(M), \quad \text{for } k=1,2.
\end{equation}
For $k=3$ we need to distinguish between degenerate ($\k=0$) and non-degenerate ($\k>0$) case:
\begin{equation}\label{boundk=3}
\| 1/\Lambda_m^-\|_{\dot{W}^{3,\infty}([-1,1])}\leq \begin{cases}
\frac{C(M)}{\e\k} \quad &\k>0,\\
\frac{C}{\e} \quad &\k=0,
\end{cases}
 \quad \text{and} \quad  \|\ep/\Lambda_m^-\|_{\dot{W}^{3,\infty}([-1,1])}\leq \begin{cases}
\frac{C(M)}{\k} \quad &\k>0,\\
C \quad &\k=0.
\end{cases}
\end{equation}
\begin{proof}[Proof of \eqref{bound1/Lambdageneral}]
Note that
\begin{align*}
\p_z^{1)}\left(\frac{1}{\Lambda_m^-(z)}\right)&=-\p_z  \Lambda_m^-(z)\left(\frac{1}{ \Lambda_m^-(z)}\right)^2,\\
\p_z^{2)}\left(\frac{1}{\Lambda_m^-(z)}\right)&=-\p_z^{2)} \Lambda_m^-(z)\left(\frac{1}{\Lambda_m^-(z)}\right)^2-2 \p_z^{1)} \Lambda_m^-(z) \left(\frac{1}{\Lambda_m^-(z)}\right) \p_z^{1)}\left(\frac{1}{\Lambda_m^-(z)}\right),
\end{align*}
with
\begin{align*}
\Lambda_m^-(z)&=(1+\l_1-z)\ep +(\l_2^{\ep}+\Phi_{\k}(z))\ep^2,\\
\p_z \Lambda_m^-(z) &=-\ep + \Phi_{\k}'(z)\ep^2,\\
\p_z^2 \Lambda_m^-(z) &= \Phi_{\k}''(z)\ep^2,
%\p_z^3 \Lambda_m^-(z) &= \Phi_{\k}'''(z)\ep^2,
\end{align*}
give us \eqref{bound1/Lambdageneral} for $k=1$ by \eqref{bound1/Lambda}. Similarly, we get \eqref{bound1/Lambdageneral} for $k=2$ by \eqref{bound1/Lambda} and \eqref{bound1/Lambdageneral} for $k=1.$
\end{proof}

\begin{proof}[Proof of \eqref{boundk=3}] For $k=3$ we have the expression
\begin{align}\label{3derivadas1/L}
\p_z^{3)}\left(\frac{1}{\Lambda_m^-(z)}\right)=&-\p_z^{3)} \Lambda_m^-(z)\left(\frac{1}{\Lambda_m^-(z)}\right)^2-4\p_z^{2)} \Lambda_m^-(z)\left(\frac{1}{\Lambda_m^-(z)}\right)\p_z^{1)}\left(\frac{1}{\Lambda_m^-(z)}\right)\\
&-2 \p_z^{1)} \Lambda_m^-(z)  \left(\p_z^{1)}\left(\frac{1}{\Lambda_m^-(z)}\right)\right)^2 \nonumber \\
&-2 \p_z^{1)} \Lambda_m^-(z) \left(\frac{1}{\Lambda_m^-(z)}\right) \p_z^{2)}\left(\frac{1}{\Lambda_m^-(z)}\right).\nonumber
\end{align}
The difference between degenerate and non-degenerate case is due to the factor:
\[
\p_z^3 \Lambda_m^-(z) = \Phi_{\k}'''(z)\ep^2=\varphi_{\k}''(z)\ep^2.
\]
Note that for the degenerate case, we have that  $\varphi=\frac{1-z}{2}$ is a linear function and the above term vanishes. In contrast, for the non-degenerate case ($\k>0$) we obtain the bound
\[
\|\p_z^3 \Lambda_m^-\|_{L^{\infty}([-1,1])}= \|\varphi_{\k}''\|_{L^{\infty}([-1,1])} \ep^2 \leq  \frac{C}{\k}\ep^2.
\]
Finally, by taking the $L^{\infty}([-1,1])-$norm on both sides of \eqref{3derivadas1/L} and combining previous estimates \eqref{bound1/Lambda} and \eqref{bound1/Lambdageneral} we arrive to the required result.
\end{proof}
\begin{remark}
Notice that the bound obtained in \eqref{boundk=3} is completely natural due to the fact that
$$\varphi_{\k}'(z)\xrightarrow{\k\to 0} -\frac{1}{2}\chi_{[-1,1]}(z), \quad \text{in }  H^{\frac{3}{2}-}(\R),$$
but not in higher-order regularity spaces. Then, for $\k>0$ the Sobolev embedding just give us
$$\|\varphi_{\k}''\|_{L^{\infty}([-1,1])}\lesssim \|\varphi_{\k}'\|_{H^{\frac{3}{2}+}([-1,1])}\leq \frac{C}{\k}.$$
\end{remark}

Therefore, we have proved
\begin{align*}
\|v_1\|_{H^k}\leq c_{\k}\frac{C(M)}{\ep}\left(\|W^+_m\|_{H^k}+\|W^-_m\|_{H^k}\right),
\end{align*}
for $k=0,1,2$ and $3$ with
\begin{equation}\label{degenerateconstant}
c_{\k}:=\begin{cases}
1 \qquad &\text{if } \k=0,\\
\k^{-1} \qquad &\text{if } \k>0.
\end{cases}
\end{equation}

Proceeding similarly for $u_1$, we have the expression:
\begin{align}\label{auxu}
u_1(z)=&\frac{\ep}{2m\Lambda_m^+(z)}\int_{-1}^1\varphi_{\kappa}'(\bar{z})u_1(\bar{z})e^{-\ep m|z-\bar{z}|}\mbox{d}\bar{z}\\
&-\frac{\ep}{2m\Lambda_m^+(z)}e^{-2m}\int_{-1}^1\varphi_{\kappa}'(\bar{z})v_1(\bar{z})e^{-\ep m (z+\bar{z})}\mbox{d}\bar{z} -\frac{1}{m\Lambda^+_m(z)}W^+_m(z).\nonumber
\end{align}
Now, since
\begin{align*}
\Lambda_m^+(z)&=2+(-1+\l_1+z)\ep +(\l_2^{\ep}-\Phi_{\k}(z))\ep^2,\\
\p_z \Lambda_m^+(z) &=\ep - \Phi_{\k}'(z)\ep^2,\\
\p_z^2 \Lambda_m^+(z) &= -\Phi_{\k}''(z)\ep^2,\\
\p_z^3 \Lambda_m^+(z) &= -\Phi_{\k}'''(z)\ep^2,
\end{align*}
we observe that
\begin{equation*}
\| 1/\Lambda_m^+\|_{\dot{W}^{k,\infty}([-1,1])}\leq \e^k, \qquad \text{and} \qquad  \|\ep/\Lambda_m^+\|_{\dot{W}^{k,\infty}([-1,1])}\leq \e^{k+1}, \quad \text{for } k=0,1,2.
\end{equation*}
%\begin{equation*}
%\| 1/\Lambda_m^+\|_{\dot{W}^{k,\infty}([-1,1])}\leq \frac{\e^k}{1+\tfrac{\e}{C(M)}}, \qquad \text{and} \qquad  \|\ep/\Lambda_m^+\|_{\dot{W}^{k,\infty}([-1,1])}\leq \frac{\e^{k+1}}{1+\tfrac{\e}{C(M)}}, \quad \text{for } k=0,1,2.
%\end{equation*}
and
\begin{equation*}
\| 1/\Lambda_m^-\|_{\dot{W}^{3,\infty}([-1,1])}\leq c_{\k}\e^2
 \qquad \text{and} \qquad  \|\ep/\Lambda_m^-\|_{\dot{W}^{3,\infty}([-1,1])}\leq c_{\k} \e^3.
\end{equation*}
Therefore, we have proved that
\begin{align*}
\|u_1\|_{H^k}\leq c_{\k}  \left(\|W^+_m\|_{H^k}+\|W^-_m\|_{H^k}\right),
\end{align*}
for $k=0,1,2,3$ and $c_{\k}$ given by \eqref{degenerateconstant}.

\subsubsection{The case $n\neq m$.}\label{s:inverseoperator}

In this case we have to solve
\begin{align}\label{uno}
&\Lambda^+_m(z)u(z)-\frac{\ep}{2n}\int_{-1}^1\varphi_{\kappa}'(\bar{z})u(\bar{z})e^{-\ep n|z-\bar{z}|}\mbox{d}\bar{z}+\frac{\ep e^{-2n}}{2n}\int_{-1}^1\varphi'_{\kappa}(\bar{z})v(\bar{z})e^{-\ep n(z+\bar{z})}\mbox{d}\bar{z}=-\frac{1}{n}W^+_n(z),\\
&\frac{\Lambda^-_m(z)}{\ep}v(z)+\frac{1}{2n}\int_{-1}^1\varphi_{\kappa}'(\bar{z})v(\bar{z})e^{-\ep n|z-\bar{z}|}\mbox{d}\bar{z}-\frac{e^{-2n}}{2n}\int_{-1}^1\varphi'_{\kappa}(\bar{z})u(\bar{z})e^{-\ep n(z+\bar{z})}\mbox{d}\bar{z}=-\frac{1}{n\ep}W^-_n(z).\label{dos}
\end{align}
Since $\Lambda^+_m(z)=O(1)$ we can divide \eqref{uno} by $\Lambda^+_m(z)$ to get
\begin{align}\label{iu}
u(z)=-\frac{1}{n\Lambda^+_m(z)}W^+_n(z)+\ep U[u,v](z),
\end{align}
where
\begin{align*}
U[u,v](z):=\frac{1}{2n\Lambda^+_m(z)}\int_{-1}^1\varphi_{\kappa}'(\bar{z})u(\bar{z})e^{-\ep n|z-\bar{z}|}\mbox{d}\bar{z}-\frac{e^{-2n}}{2n\Lambda^+_m(z)}\int_{-1}^1\varphi'_{\kappa}(\bar{z})v(\bar{z})e^{-\ep n(z+\bar{z})}\mbox{d}\bar{z}.
\end{align*}
And then we can write \eqref{dos}, using \eqref{iu} as
\begin{align}\label{tres}
(1+\lambda_1-z)v(z)+\frac{1}{2n}\int_{-1}^1\varphi_\kappa'(\bar{z})v(\bar{z})\mbox{d}\bar{z}=&-\frac{1}{n\ep}W^-_n(z)\\ +
&\frac{e^{-2n}}{2n}\int_{-1}^1\varphi_{\kappa}'(\bar{z})\left(-\frac{1}{n\Lambda^+_m(\bar{z})}W^+_n(\bar{z})\right)e^{-\ep n(z+\bar{z})}\mbox{d}\bar{z} \nonumber \\
+&\ep V[u,v](z), \nonumber
\end{align}
where
\begin{align*}
V[u,v](z):=&- (\lambda_2^\ep+\Phi_{\kappa}(z))v(z)+\frac{1}{2n}e^{-2n}\int_{-1}^1\varphi_{\kappa}'(\bar{z})U[u,v](\bar{z})e^{-\ep n(z+\bar{z})}\mbox{d}\bar{z}\\
&-\frac{1}{2n}\int_{-1}^1\varphi'_{\kappa}(\bar{z})v(\bar{z})\frac{e^{-n\ep|z-\bar{z}|}-1}{\ep}\mbox{d}\bar{z}.
\end{align*}

In order to solve \eqref{tres}, we will use the following lemma.
\begin{lemma}\label{inversion}Let $F\in L^2([-1,1])$. Then, the unique solution to
\begin{align}\label{laecuacion}
(1+\lambda_1-z)f(z)+\frac{1}{2n}\int_{-1}^1
\varphi_\kappa'(\bar{z})f(\bar{z})\mbox{d}\bar{z}=F(z),
\end{align}
is given by
\begin{align*}
f(z)=\frac{1}{1+\lambda_1-z}\left(F(z)+\frac{1}{2(m-n)}\int_{-1}^1\varphi_\kappa'(\bar{z})\frac{F(\bar{z})}{1+\lambda_1-\bar{z}}\mbox{d}\bar{z}\right).
\end{align*}
\end{lemma}
\begin{proof}If $f(z)$ is a solution of \eqref{laecuacion} then
\begin{align}\label{vsolucion}
f(z)=\frac{C}{1+\lambda_1-z}+\frac{F(z)}{1+\lambda_1-z},
\end{align}
for some constant $C$. Plugging this expression of $f$ into \eqref{laecuacion} yields
\begin{align*}
C+\frac{C}{2n}\int_{-1}^1\frac{\varphi_{\kappa}'(\bar{z})}{1+\lambda_1-\bar{z}}\mbox{d}\bar{z}+\frac{1}{2n}\int_{-1}^1 \varphi_{\kappa}'(\bar{z}) \frac{F(\bar{z})}{1+\lambda_1-\bar{z}}\mbox{d}\bar{z}
=0,
\end{align*}
Then, using \eqref{lambda1eq} we get
\begin{align*}
C\left(1-\frac{2m}{2n}\right)=-\frac{1}{2n}\int_{-1}^1 \varphi_{\kappa}'(\bar{z}) \frac{F(\bar{z})}{1+\lambda_1-\bar{z}}\mbox{d}\bar{z},
\end{align*}
or
\begin{align}\label{lac}
C=\frac{1}{2(m-n)}\int_{-1}^1 \varphi_{\kappa}'(\bar{z}) \frac{F(\bar{z})}{1+\lambda_1-\bar{z}}\mbox{d}\bar{z}.
\end{align}
Conversely, it is easy to check that $f$ in \eqref{vsolucion}, with $C$ given by \eqref{lac}, solves the equation \eqref{laecuacion}.
\end{proof}
To finish, let us introduce
\begin{align*}
I[F](z):=\frac{1}{1+\lambda_1-z}\left(F(z)+\frac{1}{2(m-n)}\int_{-1}^1\varphi_\kappa'(\bar{z})\frac{F(\bar{z})}{1+\lambda_1-\bar{z}}\mbox{d}\bar{z}\right),
\end{align*}
and
\begin{align*}
W(z):=-\frac{1}{n\ep}W^-_n(z)+\frac{e^{-2n}}{2n}\int_{-1}^1\varphi_{\kappa}'(\bar{z})\left(-\frac{1}{n\Lambda^+_m(\bar{z})}W^+_n(\bar{z})\right)e^{-\ep n(z+\bar{z})}\mbox{d}\bar{z}.
\end{align*}
By using Lemma \ref{inversion}, equation \eqref{tres} can be written as
\begin{align}\label{iv}
v(z)=I[W](z)+\ep I[V[u,v]](z).
\end{align}

The coupled system given by \eqref{iu} and \eqref{iv} is a linear contraction on $L^2([-1,1])$ for $\ep$ small enough. Therefore, there exists a unique $(u,v)\in L^2([-1,1])\times L^2([-1,1])$ solving \eqref{iu}, \eqref{iv} with
\begin{align*}
\|u\|_{L^2} + \|v\|_{L^2}\leq \frac{C(M)}{\ep n}\left(\|W^+_n\|_{L^2}+ \|W^-_n\|_{L^2}\right)
\end{align*}

Therefore, this $(u,v)$ is the unique solution of \eqref{uno} and \eqref{dos}. Taking derivatives on \eqref{iu} and \eqref{iv} we find that actually
\begin{align*}
\|u\|_{H^k} + \|v\|_{H^k}\leq c_{\k}\frac{C(M)}{\ep n}\left(\|W^+_n\|_{H^k}+ \|W^-_n\|_{H^k}\right)
\end{align*}
for $k=0,1,2,3$ and $c_{\k}$ given by \eqref{degenerateconstant}.

Then we have achieved all the conclusions of Theorem \ref{codimension}.
\end{proof}

\subsection{The transversality property}
To check the complete assumptions of Crandall-Rabinowitz's theorem it remains to prove the transversality assumption. In order to do this is enough to show that
\[
D_{\l,\ff}^2 F[\l_{\ep,\kappa,m},0]\hh_{\ep,\kappa,m} \not\in\cR(\cL_{\ep,\kappa}[\l_{\ep,\kappa,m}]),  \qquad \text{where} \qquad
\cN( \cL_{\ep,\kappa}[\l_{\ep,\kappa,m}]) = \text{span }\{\hh_{\ep,\kappa,m}\}.
\]
This will be done in a straightforward way without any  difficulty. Recall that for any $m\in \N$ we have proved before that there exist an eigenvalue $\l_{\ep,\kappa,m}\in\R$ given by Theorem \ref{eyufijadom} such that $\cN( \cL_{\ep,\kappa}[\lambda_{\ep,\kappa,m}]) = \text{span }\{\hh_{\ep,\kappa,m}\},$ where $\hh_{\ep,\kappa,m}(\bx)=\hh(y)\cos( m x)$ and with $\hh$ given by
\begin{equation}\label{def:tranversality}
h(y)= \left\{
\begin{split}
a(\tfrac{+y-1}{\e}) \qquad &\text{if} \quad  y\in[+1-\e,+1+\e],\\
b(\tfrac{-y-1}{\e}) \qquad &\text{if} \quad y\in [-1-\e,-1+\e].
\end{split}\right.
\end{equation}
 To finish, let us proceed by reduction to absurd. Since  $D_{\l,\ff}^2 F[\lambda_{\ep,\kappa,m},0]\hh_{\ep,\kappa,m}=-\hh(y)m \sin(m x),$ using Lemma \ref{necesario}  we have that if  $D_{\l,\ff}^2 F[\l_{\ep,\kappa,m},0]\hh_{\ep,\kappa,m} \in\cR(\cL_{\kappa,\ep}[\l_{\ep,\kappa,m}]),$ then $$-\|a\|^2_{L^2_{\varphi_{\kappa}}} + \|b\|^2_{L^2_{\varphi_{\kappa}}}=0.$$ This give us a contradiction since $\|a\|_{L^2_{\varphi'_\kappa}}=O(\ep)$ and $\|b\|_{L^2_{\varphi'_\kappa}}=O(1)$ and concludes the last required condition of Crandall-Rabinowitz  theorem \ref{th:CR}.

\section{Main theorem}

After verifying  all the conditions for the application of the Crandall-Rabinowitz  theorem \ref{th:CR} and the discussion in Section \ref{s:formulation} we obtain the following theorem:

\begin{theorem}\label{thm5} Fixed $1< M<\infty$. There exist $\ep_0(M)$, $\kappa_0(M)$ such that, for every $0<\ep<\ep_0$, $0<\kappa<\kappa_0$ and $m\in \N$, $m<M$, there exist a branch of solutions, $\ff_{\ep,\kappa,m}^\sigma\in H^{4,3}(D_{\ep})$ parameterize by $\sigma$, of equation  \eqref{e:goal}, with $|\sigma|<\sigma_0$, for some small number $\sigma_0>0$, $\varpi_{\ep,\kappa}$ as in section \ref{profile} and $\lambda=\lambda^\sigma_{\ep,\kappa,m}.$ These solutions satisfy:
\begin{enumerate}
\item $\ff^\sigma_{\ep,\kappa,m}(x,y)$ is $\frac{2\pi}{m}-$periodic on $x$.
\item The branch $$\ff^\sigma_{\ep,\kappa,m}=\sigma h_{\ep,\kappa,m}+o(\sigma) \quad\text{ in $H^{4,3}(D_\ep)$},$$
and the speed
$$\lambda^\sigma=\lambda_{\ep,\,\kappa,m}+o(\sigma),$$ where $(\lambda_{\ep,\kappa,\sigma},\, h_{\ep,\kappa,m})$ are given in Theorem \ref{eyufijadom} and Remark \ref{remark}.
\item  $\ff^\sigma_{\ep,\kappa,m}(x,y)$ depends on $x$ in a nontrivial way.
\end{enumerate}
In addition, the vorticity $\omega^\sigma_{\kappa,\sigma,m}\in H^{4,3}(\T\times \R)$,  given implicitly by
\begin{align*}
&\omega^\sigma_{\ep,\kappa,m}(x_1,x_2)=\varpi_{\e,\k}(y),
\end{align*}
for $(x_1,x_2)=(x,y+f^\sigma_{\ep,\kappa,m}(x,y))$ with $x\in\T$ and $y\in [-1-\ep,-1+\ep]\cup [1-\ep,1+\ep]$,
\begin{align*}
\omega^\sigma_{\ep,\kappa,m}(x_1,x_2)=\ep,
\end{align*}
for $x_1\in\T$ and $-1+\ep+f^\sigma_{\ep,\kappa,m}(x_1,-1+\ep)<x_2<1-\ep+f^\sigma_{\ep,\kappa,m}(x_1,1-\ep)$, and
\begin{align*}
\omega^\sigma_{\ep,\kappa,m}(x_1,x_2)=0,
\end{align*}
for $x_1\in \T$ and either $x_2>1+\ep+f^\sigma_{\ep,\kappa,m}(x_1,1+\ep)$ or $x_2<-1-\ep +f^{\sigma}_{\ep,\kappa,m}(x_1,-1-\ep)$,\\
yields a traveling way solution for 2D Euler in the sense that $$\omega^\sigma_{\ep,\kappa,m}(x_1+\lambda^\sigma_{\ep,\kappa,m}t,x_2)$$
satisfies  the system \eqref{e:E2Dperturbado}. Importantly, $\omega^\sigma_{\ep,\kappa,m}(x_1,x_2)$ depends non trivially on $x_1$.
\end{theorem}

Then, in order to prove Theorem \ref{thmbasic} it remains to prove that $H^{\frac{3}{2}-}(\T\times \R)$-norm of $\omega^\sigma_{\ep,\kappa,m}$ can be made as small as we want and that $\omega^\sigma_{\ep,\kappa,\sigma}\in C^\infty(\T\times\R)$. We do this in Theorems \ref{thm6} and \ref{thm7}.

\subsection{Distance of the traveling wave to the Couette flow}\label{s:distance}
The solution $\omega^\sigma_{\ep,\kappa,m}$ obtained in Theorem \ref{thm5} satisfies the following statement:
\begin{theorem}\label{thm6}
Fixed $M>1$, $0<\kappa<\kappa_0$ and $0<\gamma< \frac{3}{2}$, for all $\epsilon>0$ and $1\leq m<M$, there exist $\ep>0$ and $\sigma>0$ such that
\begin{align*}
\|\omega^\sigma_{\ep,\kappa,m}\|_{H^\gamma(\T\times \R)}<\epsilon.
\end{align*}
\end{theorem}

\begin{proof}
Let us emphasis that we can make $\|\ff^\sigma_{\ep,\kappa,m}\|_{H^{4,3}}$ arbitrarily small fixed $\ep$, $\kappa$ and $m$, taking $\sigma$ small. We have all the ingredients to obtain a quantitative estimate of the distance between the Couette flow and the constructed traveling wave. We could compute the $H^{\frac{3}{2}-}$ norm of $w^\sigma_{\ep,\kappa,m}$ from the expression
\begin{align*}
\Lambda^{\gamma}w^\sigma_{\ep,\kappa,m}(\bx)=k_{\gamma}\int_{\R^2}\frac{w^\sigma_{\ep,\kappa,m}(\bx)-w^\sigma_{\ep,\kappa,m}(\by)}{|\bx-\by|^{2+\gamma}}\mbox{d}\by.
\end{align*}
We could show that this quantity is as small as we want by making $\sigma$ and $\ep$ small with independence of $\kappa$. Remember that $\varpi_{\ep,0}$ is in $H^{\frac{3}{2}-}$. However, to avoid tedious computations, we will take a shortcut using interpolation of Sobolev norms:

\[
\|g\|_{\dot{H}^s}\lesssim  \|g\|_{L^2}^{1-s} \|g\|_{\dot{H}^1}^{s}, \qquad (0<s<1),
\]
which (taking $0<\g\ll 1$) give us
\begin{equation}\label{boundgoal}
\|w+1\|_{\dot{H}^{\frac{3-\g}{2}}}\equiv \|\omega\|_{\dot{H}^{\frac{3-\g}{2}}}\lesssim \| \omega\|_{\dot{H}^1}^{\frac{1+\g}{2}} \|\omega\|_{\dot{H}^2}^{\frac{1-\g}{2}}.
\end{equation}
\begin{remark}
This way of proceeding will make us lose the independence on $\kappa$. But since we are actually interested on the case $\kappa>0$ it will be good enough.
\end{remark}

To alleviate the notation let us skip the subscripts $(\ep,\kappa,m)$ and the superscript $\sigma$ on $w^\sigma_{\ep,\kappa,m}$, $\omega^\sigma_{\ep,\kappa,m}$, and $\ff^\sigma_{\ep,\kappa,m}$ in the rest of the section. We will keep $\varpi_{\ep,\kappa}$ and $\varphi_{\kappa}$ as we did before.

In order to compute the right-hand side of \eqref{boundgoal}, we have that
\begin{equation}\label{1derivative}
\n\omega(x,y+\ff(\bx))=\frac{\varpi_{\e,\k}'(y)}{1+\ff_y(\bx)}(-\ff_x(\bx),1),\quad \text{on $\supp(\nabla \omega)$},
\end{equation}
thus, making the appropriate change of variables, we obtain
\begin{align*}
\|\omega\|_{\dot{H}^2}^2=&\int_{\T\times\R}\left|\n^2 \omega(\bx)\right|^2 \mbox{d}\bx=\int_{\supp(\n^2\omega)}\left|\p^2 \omega(\bx)\right|^2 \mbox{d}\bx\\
=&\int_{D_\ep}|\n^2\omega (x,y+\ff(\bx))|^2(1+\ff_y(\bx))\mbox{d}\bx,
\end{align*}
and computing second order derivatives we get
\begin{align*}
\pa_{x}^2\omega(x,y+\ff(\bx))=&A_1(\bx)\varpi_{\e,\k}'(y)+A_2(\bx)\varpi_{\e,\k}''(y),\\
\pa_{y}^2\omega(x,y+\ff(\bx))=&A_3(\bx)\varpi_{\e,\k}'(y)+A_4(\bx)\varpi_{\e,\k}''(y),\\
\pa_{xy}^2\omega(x,y+\ff(\bx))=&A_5(\bx)\varpi_{\e,\k}'(y)+A_6(\bx)\varpi_{\e,\k}''(y),\\
\end{align*}
where $A_{i}$ for $i=1,...,6$ are functions which depend on $\pa_x \ff, \pa_y \ff, \pa^2_{x}\ff, \pa^2_{y}\ff, \pa^2_{xy}\ff$.

The $H^{4,3}-$norm of $\ff$ could depend on $\e$ badly. However we always could choose $\s$ small enough in such a way that the $H^{4,3}$-norm of $\ff$ is small. Noticing that $H^{4,3} \subset C^2$ (see \cite[Lemma 4.1]{CCG2}) we have that $A_i$ for $i=1,...,6$ are bounded functions.
Then,  we have that
\begin{align*}
\|\omega\|_{\dot{H}^1} &\approx \|\varpi_{\e,\k}'\|_{L^2(I_\ep)},\\
\|\omega\|_{\dot{H}^2} &\approx \|\varpi_{\e,\k}'\|_{L^2(I_\ep)}+\|\varpi_{\e,\k}''\|_{L^2(I_\ep)}.
\end{align*}
Let us note that by  parity of the profile $\varpi_{\e,\k}$ on $I_\e=[-1-\e,-1+\ep]\cup[+1-\ep,+1+\ep]$, we can just reduced our problem to study the following norms:
\[
\|\varpi_{\e,\k}'\|_{L^2([1-\e,1+\e])}, \|\varpi_{\e,\k}''\|_{L^2([1-\e,1+\e])}.
\]
Recalling the definitions of $\varpi_{\ep,\kappa}$ and  $\varphi_{\kappa}$ in section \ref{profile}. We have that

$$\varpi_{\ep,\kappa}'(y+1)=-\frac{\psi'_{\kappa}\left(\frac{y}{\ep}\right)}{\int_{-1}^1\psi_{\ep,\kappa}'(\bar{z})\mbox{d}\bar{z}}, \qquad \text{for } |y| \leq \e.$$
At this point it does not matter to consider the normalization factor  $\int_{-1}^1\psi'_{\kappa}(\bar{z})\mbox{d}\bar{z}$. Then
\begin{equation}\label{auxH1}
\|\varpi_{\e,\k}'\|_{L^2([1-\e,1+\e])}=\int_{1-\ep}^{1+\ep}|\varpi_{\ep,\kappa}'(y)|^2 \dy=\int_{-\ep}^\ep \left|\psi_{\kappa}'\left(\frac{y}{\ep}\right)\right|^2  \dy\leq 2\ep \|\psi_{\kappa}'\|_{L^{\infty}([-1,1])}^2,
\end{equation}
and
\begin{align}\label{varpiH2}
\|\varpi_{\e,\k}''\|_{L^2([1-\e,1+\e])}=\int_{1-\ep}^{1+\ep}|\varpi_{\ep,\kappa}''(y)|^2 \dy=\frac{1}{\ep^2}\int_{-\ep}^\ep \left|\psi_{\kappa}''\left(\frac{y}{\ep}\right)\right|^2  \dy.
\end{align}
Since
\begin{align*}
\psi'_{\kappa}(z)=\int_{-1+\kappa}^{1+\kappa}\Theta_{\kappa}(z-\bar{z})\mbox{d}\bar{z},
\end{align*}
thus
\begin{align*}
\psi''_{\kappa}(z)=\Theta_\kappa (z-(1-\kappa))-\Theta_{\kappa}(z-(-1+\kappa)),
\end{align*}
and
\begin{align*}
\psi''_{\kappa}\left(\frac{y}{\ep}\right)=\frac{1}{\kappa}\left(\Theta\left(\frac{y-\ep(1-\kappa)}{\ep\kappa}\right)
-\Theta\left(\frac{y-\ep(-1+\kappa)}{\ep\kappa}\right)\right).
\end{align*}
Putting the above expression into \eqref{varpiH2} we get
\begin{align*}
\|\varpi_{\e,\k}''\|_{L^2([1-\e,1+\e])} \leq \frac{1}{\ep^2 \k^2}\left(\int_{-\ep}^\ep  \left|\Theta\left(\frac{y-\ep(1-\kappa)}{\ep\kappa}\right)\right|^2\dy+\int_{-\ep}^\ep  \left|\Theta\left(\frac{y-\ep(-1+\kappa)}{\ep\kappa}\right)\right|^2\dy\right).
\end{align*}
Now, we make the changes of variables $\tilde{y}=y-\ep(1-\kappa)$ in the first integral and $\tilde{y}=y-\ep(-1+\kappa)$ in the second one. Note that the limits of integration will be $(-\ep(2-\kappa),\ep\kappa)$ and $(-\ep\kappa, \ep(2-\kappa))$, respectively. But $\ep(2-\kappa)> \ep \kappa$ for   $\kappa <1$, and the support of $\Theta(\tilde{y} / \ep \kappa )$ is inside of $(-\ep\kappa,\ep\kappa)$.
Thus both integrals run from $-\ep\kappa$ to $\ep\kappa$:
\begin{equation}\label{auxH2}
\|\varpi_{\e,\k}''\|_{L^2([1-\e,1+\e])} \leq \frac{2}{\ep^2 \k^2}\int_{-\ep\k}^{\ep\k}  \left|\Theta\left(\frac{y}{\ep\kappa}\right)\right|^2 \dy \leq \frac{4}{\ep \k}\|\Theta\|_{L^{\infty}([-1,1])}^2.
\end{equation}

Therefore, combining \eqref{auxH1} and \eqref{auxH2}, there exists $C>0$ such that
%\begin{align*}
%||\Lambda^\frac{1-\gamma}{2}\nabla \omega||_{L^2}\leq ||\nabla \omega||_{L^2}^\frac{1+\gamma}{2}||\nabla^2\omega||_{L^2}^\frac{1-\gamma}{2}
%\leq C \ep^\frac{1+\gamma}{4}\ep^\frac{-1+\gamma}{4}\kappa^\frac{-1+\gamma}{4}\leq C \ep^\frac{\gamma}{2}\kappa^\frac{-1+\gamma}{4}.
%\end{align*}
\[
\|w+1\|_{\dot{H}^{\frac{3-\g}{2}}}\equiv \|\omega\|_{\dot{H}^{\frac{3-\g}{2}}}\lesssim \| \omega\|_{\dot{H}^1}^{\frac{1+\g}{2}} \|\omega\|_{\dot{H}^2}^{\frac{1-\g}{2}}\leq C \ep^\frac{1+\gamma}{4}\ep^\frac{-1+\gamma}{4}\kappa^\frac{-1+\gamma}{4}\leq C \ep^\frac{\gamma}{2}\kappa^\frac{-1+\gamma}{4}.
\]
In addition
\begin{align*}
\|w+1\|_{L^2}&\equiv \|\omega\|_{L^2}\lesssim \e.
\end{align*}
Consequently, for any  $\epsilon>0$ and for any $0<s<3/2$, taking $\e$ and $\sigma$ small enough, we find a traveling wave such that its vorticity satisfies
$\|w^\sigma_{\ep,\kappa, m}+1\|_{H^s(\T\times\R)}<\epsilon.$
\end{proof}

\subsection{Full regularity of the solution}\label{s:fullregularity}
This section is devote to proving the following result.

\begin{theorem}\label{thm7}
The solution $\omega^\sigma_{\ep,\kappa,m}$ in Theorem \ref{thm5} is actually $C^\infty(\T\times \R)$.
\end{theorem}

\begin{proof}
 In order to prove this theorem  we will use equation \eqref{e:goal}, i.e.,
\begin{align*}
&(\lambda^\sigma_{\ep,\kappa,m} +y+\ff^\sigma_{\ep,\kappa,m}(\bx)-\mathfrak{u}_1[\ff^\sigma_{\ep,\kappa,m}](\bx))\pa_x\ff^\sigma_{\ep,\kappa,m}(\bx)=\mathfrak{u}_2[\ff^\sigma_{\ep,\kappa,m}](\bx),	\qquad \bx\in D_\e=\T\times I_\ep,
\end{align*}
where
\begin{align*}
\mathfrak{u}_1[\ff^\sigma_{\ep,\kappa,m}](\bx)=\frac{1}{4\pi}\int_{\T\times I_\ep}\varpi'_{\ep,\kappa}(\bar{y})
\log\left(\cosh(y-\bar{y}+\ff^\sigma_{\ep,\kappa,m}(\bx)-\ff^\sigma_{\ep,\kappa,m}(\bar{\bx}))-\cos(x-\bar{x})\right)\mbox{d}\bar{\bx},
\end{align*}
and
\begin{align*}
\mathfrak{u}_2[\ff^\sigma_{\ep,\kappa,m}](\bx)=&\frac{-1}{4\pi}\int_{\T\times I_\ep}\varpi'_{\ep,\kappa}(\bar{y})\log\left(\cosh(y-\bar{y}+\ff^\sigma_{\ep,\kappa,m}(\bx)-\ff^\sigma_{\ep,\kappa,m}(\bar{\bx}))-\cos(x-\bar{x})\right)\pa_x \ff^\sigma_{\ep,\kappa,m}(\bar{\bx})\mbox{d}\bar{\bx}.
\end{align*}
Let us remove the superscript $\sigma$ and the subscripts $\ep$, $\kappa$ and $m$ from $\ff^\sigma_{\ep,\kappa,m}$ to alleviate the notation.

First of all we notice that, since $\ff\in X(D_{\ep})$ by construction, in particular it is mean zero in $x$. So, we can recover $\ff$ from $\pa_x\ff$ through the expression
\begin{align*}
\ff(\bx)=\text{Int}[\pa_x \ff](\bx):= \int_{0}^x \pa_x\ff(\bar{x},y)\mbox{d}\bar{x}-\frac{1}{2\pi}\int_{-\pi}^\pi \left(\int_{0}^{\bar{x}}\pa_x\ff(\tilde{x},y)\mbox{d}\tilde{x}\right) \mbox{d}\bar{x}, \qquad x>0.
\end{align*}
Therefore, if $\pa_x\ff \in H^{k}(D_\ep)$ then it is clear that $\ff\in H^{k+1,k}(D_\ep)$. Next, we will show that if $\pa_x\ff\in H^{k}(D_\ep)$, then in fact $\ff\in H^{k+1}(D_\ep)$ for $k\geq 3.$ Using the above expressions we have that, if $\pa_x \ff\in H^k(D_\e)$ then $\mathfrak{u}_i[\ff]\in H^{k+1}(D_\e)$ for $i=1,2$.  In addition, we can split $\mathfrak{u}_1[\ff]$ in the following way
\begin{align*}
\mathfrak{u}_1[\ff](\bx)=\Omega_{\ep,\kappa}(y)+\left(\mathfrak{u}_1[\ff](\bx)-\Omega_{\ep,\kappa}(y)\right),
\end{align*}
where (see \eqref{e:simplification0})
\begin{align*}
\Omega_{\ep,\kappa}(y)=\frac{1}{4\pi}\int_{\T\times I_\ep}\varpi'_{\ep,\kappa}(\bar{y})
\log\left(\cosh(y-\bar{y})-\cos(x-\bar{x})\right)\mbox{d}\bar{\bx}=\int_{0}^y\varpi_{\ep,\kappa}(\bar{y})\mbox{d}\bar{y}.
\end{align*}

Let us fix $k=3$, i.e. $\pa_x \ff\in H^3(D_\e)$ and prove that in fact $\ff\in H^4(D_\e)$. Then we have
\begin{align*}
&\|\mathfrak{u}_1[\ff]-\Omega_{\ep,\kappa}\|_{L^\infty}\leq \sigma\, C_{\ep,\kappa,M},\\
&\|\mathfrak{u}_2[\ff]\|_{H^4}\leq \sigma \,C_{\ep,\kappa,M},
\end{align*}
and
\begin{align*}
\lambda^{\sigma}_{\ep,\kappa,m}+y-\Omega_{\ep,\kappa}(y)\geq\lambda^{\sigma}_{\ep,\kappa,m}\geq  c_{\ep,\kappa,M}>0.
\end{align*}
Then, taking  $\sigma$ small enough we obtain
\begin{align}\label{lambdasigma}
\lambda^\sigma_{\ep,\kappa,m} +y+\ff(\bx)-\mathfrak{u}_1[\ff](\bx)>c_{\ep,\kappa,M},
\end{align}
and we can write
\begin{align}\label{bonita}
\pa_x \ff (\bx)=\frac{1}{\lambda^\sigma_{\ep,\kappa,m} +y+\ff(\bx)-\mathfrak{u}_1[\ff](\bx)}\mathfrak{u}_2[\ff](\bx).
\end{align}
%Expression \eqref{bonita} is enough to get regularity on $x$.
To get extra regularity on the vertical variable let us take three derivatives on $y$ in \eqref{bonita} to obtain
\begin{align*}
\pa_x\pa_y^3\ff(\bx) +\frac{\mathfrak{u}_2[f](\bx)}{\left(\lambda^\sigma_{\ep,\kappa,m}+y+\ff(\bx)-\mathfrak{u}_1[\ff](\bx)\right)^2}\pa^3_y \ff(\bx)= \text{terms at least in $H^{0,1}(D_\e)$}.
\end{align*}

For $\sigma$ small enough  such that \eqref{lambdasigma} holds, let us define, for $g\in \{h\in L^2(D_\ep): \int_{-\pi}^{\pi}h(x,y)\dx=0\}$ the operator
\begin{align*}
T[\ff]g:= g +\frac{\mathfrak{u}_2[\ff]}{\left(\lambda^\sigma_{\ep,\kappa,m}+y+\ff-\mathfrak{u}_1[\ff]\right)^2} \text{Int}[g].
\end{align*}
For any $F\in L^2(D_\e)$ there exists $T^{-1}[\ff]$, such that, if $g\in \{h\in L^2(D_\ep) : \int_{-\pi}^{\pi}h(x,y)\dx=0\}$ satisfies
\begin{align*}
T[\ff]g=F,
\end{align*}
then
\begin{align*}
g=T^{-1}[\ff]F.
\end{align*}
In addition $$T^{-1}[\ff] : H^{0,1}(D_\ep)\to H^{0,1}(D_\ep).$$
Therefore,  $\pa_x\pa^3_y \ff$ is in $H^{0,1}(D_\ep)$. Finally, we can iterate this process to show that $\ff$ is in $C^\infty(D_\ep)$.

\end{proof}

%\newpage
\section{Appendix}\label{s:appendix}
In order to facilitate the presentation of the manuscript, we collect in this section all the technical lemmas used previously.
We start recalling the definition of the kernel given by
\begin{equation}\label{e:kernel}
K[g](\bx,\bar{\bx}):=\log\left[\cosh(\bar{y}+\D_{\bar{\bx}}[g](\bx))-\cos(\bar{x})\right],
\end{equation}
where
\[
\D_{\bar{\bx}}[g](\bx):=g(\bx)-g(\bx-\bar{\bx}).
\]
In order to take derivatives into the kernel \eqref{e:kernel} we have introduced the following two  functions
\begin{align*}
\Psi_1[g](\bx,\bar{\bx})&:=\frac{\sinh(\bar{y}+\D_{\bar{\bx}}[g](\bx))}{\cosh(\bar{y}+\D_{\bar{\bx}}[g](\bx))-\cos(\bar{x})},\\
\Psi_2[g](\bx,\bar{\bx})&:=\frac{\cosh(\bar{y}+\D_{\bar{\bx}}[g](\bx))}{\cosh(\bar{y}+\D_{\bar{\bx}}[g](\bx))-\cos(\bar{x})},
\end{align*}
with derivatives given by the expressions
\begin{align}
\p_{\bx}\Psi_1[g](\bx,\bar{\bx})&=[\Psi_2[g]-\Psi_1^2[g]](\bx,\bar{\bx}) \D_{\bar{\bx}}[\p_{\bx}g](\bx), \label{DPsi1}\\
\p_{\bx}\Psi_2[g](\bx,\bar{\bx})&=[\Psi_1[g]\left(1-\Psi_2[g]\right)](\bx,\bar{\bx})\D_{\bar{\bx}}[\p_{\bx}g](\bx). \label{DPsi2}
\end{align}
Combining \eqref{e:kernel} with \eqref{DPsi1} and  \eqref{DPsi2} it is just a matter of algebra to obtain
\begin{align*}
\p_{\bx}K[g](\bx,\bar{\bx})&=\Psi_1[g](\bx,\bar{\bx})\D_{\bar{\bx}}[\p_{\bx}g](\bx),\\
\p_{\bx}^2 K[g](\bx,\bar{\bx})&=[\Psi_2[g]-\Psi_1^2[g]](\bx,\bar{\bx})\left(\D_{\bar{\bx}}[\p_{\bx}g](\bx)\right)^2+\Psi_1[g](\bx,\bar{\bx})\D_{\bar{\bx}}[\p_{\bx}^2 g](\bx),\\
\p_{\bx}^3 K[g](\bx,\bar{\bx})&=\Psi_1[g](\bx,\bar{\bx})[ 1-3\Psi_2[g]+ 2  \Psi_1^2[g]](\bx,\bar{\bx})  \left(\D_{\bar{\bx}}[\p_{\bx}g](\bx)\right)^3\\
&\quad +3[\Psi_2[g]- \Psi_1^2[g]](\bx,\bar{\bx})\D_{\bar{\bx}}[\p_{\bx}g](\bx)\D_{\bar{\bx}}[\p_{\bx}^2 g](\bx)+ \Psi_1[g](\bx,\bar{\bx}) \D_{\bar{\bx}}[\p_{\bx}^3 g](\bx).
\end{align*}
Now, before starting with the proof of the lemmas, we remember the Taylor expansion of the trigonometric functions involved in the definition of $\Psi_{i}(\bx,\bar{\bx})$ for $i=1,2.$
\[
\sinh(z):=\sum_{n=0}^{+\infty}\frac{z^{2n+1}}{(2n+1)!},\qquad \cosh(z):=\sum_{n=0}^{+\infty}\frac{z^{2n}}{(2n)!},\qquad \text{and} \qquad \cos(z):=\sum_{n=0}^{+\infty}\frac{(-1)^n z^{2n}}{(2n)!}.
\]
Using the above expressions, it is simple to check that there exists some constant $C(\|g\|_{H^{3}(D_{\ep})})$ such that for any $\bar{\bx}\in\mathbb{T}\times \R$ with $|\bar{\bx}|\ll 1$, we get
\begin{align}
|\cosh(\bar{y}+\D_{\bar{\bx}}[g](\bx))|&\leq C(\| g\|_{H^{3}(D_{\ep})}), \nonumber\\
|\sinh(\bar{y}+\D_{\bar{\bx}}[g](\bx))|&\leq C(\|g\|_{H^{3}(D_{\ep})}) |\bar{\bx}|, \nonumber\\
|\cosh(\bar{y}+\D_{\bar{\bx}}[g](\bx))-\cos(\bar{x})|&\geq \left(\tfrac{1}{2}-\|g\|_{H^{3}(D_{\ep})}\right) |\bar{\bx}|^2-C(\|g\|_{H^{3}(D_{\ep})}) |\bar{\bx}|^4,\label{denominator}
\end{align}
and as an immediate consequence we obtain the following result.
\begin{corollary}\label{boundsPSI} Let $g\in \mathbb{B}_{\d}(H^3(D_{\ep}))$ with $0< \d(\e)\ll 1$. For $\bar{\bx}\in\T\times\R$ such that $|\bar{\bx}|\ll  1$ we have
$$|\Psi_i[g](\bx,\bar{\bx})|\leq \frac{C(\|g \|_{H^{3}(D_{\ep})})}{ 1-C(\|g \|_{H^{3}(D_{\ep})}) }\frac{1}{|\bar{\bx}|^i} \qquad \text{for} \quad i=1,2.$$
\end{corollary}
Now, we focus our attention in the higher order derivatives $\p_{\bx}^2 K[g](\bx,\bar{\bx})$ and $\p_{\bx}^3 K[g](\bx,\bar{\bx})$. It will be convenient to introduce some notation in order to handle these terms. To be more specific, we will use the following expressions:
\begin{align}
\p_{\bx}^2 K[g](\bx,\bar{\bx})&=\II_1[g](\bx,\bar{\bx})+\II_2[g](\bx,\bar{\bx}), \label{p2K}\\
\p_{\bx}^3 K[g](\bx,\bar{\bx})&=\JJ_1[g](\bx,\bar{\bx})+\JJ_2[g](\bx,\bar{\bx})+\JJ_3[g](\bx,\bar{\bx}) \label{p3K},
\end{align}
where
\begin{equation*}
  \begin{split}
\II_1[g](\bx,\bar{\bx})&:=\tilde{\II}_1[g](\bx,\bar{\bx})\left(\D_{\bar{\bx}}[\p_{\bx}g](\bx)\right)^2,\\
\II_2[g](\bx,\bar{\bx})&:=\tilde{\II}_2[g](\bx,\bar{\bx})\D_{\bar{\bx}}[\p_{\bx}^2 g](\bx),\\
\JJ_1[g](\bx,\bar{\bx})&:=\tilde{\JJ}_1[g](\bx,\bar{\bx})\left(\D_{\bar{\bx}}[\p_{\bx}g](\bx)\right)^3,\\
\JJ_2[g](\bx,\bar{\bx})&:=\tilde{\JJ}_2[g](\bx,\bar{\bx})\D_{\bar{\bx}}[\p_{\bx}g](\bx)\D_{\bar{\bx}}[\p_{\bx}^2 g](\bx),\\
\JJ_3[g](\bx,\bar{\bx})&:=\tilde{\JJ}_3[g](\bx,\bar{\bx})\D_{\bar{\bx}}[\p_{\bx}^3 g](\bx),
  \end{split}
\quad \quad
  \begin{split}
\tilde{\II}_1[g]&:=\Psi_2[g]-\Psi_1^2[g],\\
\tilde{\II}_2[g]&:=\Psi_1[g],\\
\tilde{\JJ}_1[g]&:=\Psi_1[g]\left(1- 3\Psi_2[g]+2  \Psi_1^2[g]\right),\\
\tilde{\JJ}_2[g]&:=3\Psi_2[g]-\Psi_1^2[g],\\
\tilde{\JJ}_3[g]&:=\Psi_1[g].
  \end{split}
\end{equation*}
As an immediate consequence of Corollary \ref{boundsPSI} we get  the following bounds for the above expressions.
\begin{corollary}\label{boundstilde} Let $g\in \mathbb{B}_{\d}(H^3(D_{\ep}))$ with $0< \d(\e)\ll 1$. For $\bar{\bx}\in\T\times\R$ such that $|\bar{\bx}|\ll 1$ we have
\begin{align*}
|\tilde{\II}_2[g](\bx,\bar{\bx})|, |\tilde{\JJ}_3[g](\bx,\bar{\bx})|\leq C(\|g\|_{H^{3}(D_{\ep})}) |\bar{\bx}|^{-1},\\
|\tilde{\II}_1[g](\bx,\bar{\bx})|, |\tilde{\JJ}_2[g](\bx,\bar{\bx})|\leq C(\|g\|_{H^{3}(D_{\ep})}) |\bar{\bx}|^{-2},\\
 |\tilde{\JJ}_1[g](\bx,\bar{\bx})|\leq C(\|g \|_{H^{3}(D_{\ep})}) |\bar{\bx}|^{-3}.
\end{align*}
\end{corollary}

Before starting with the proof of the main lemmas of the Appendix, we will collect a couple of auxiliary results. In first place, we note that for any $g\in H^{4,3}(D_{\ep}),$ an inequality that we will repeatedly apply in our arguments is the following:
\begin{equation}\label{bound_f'-f''}
|\D_{\bar{\bx}}[\p_{\bx}^ig](\bx)|\leq |\bar{\bx}|^{2-i} \|g\|_{H^{4,3}(D_{\ep})}, \qquad i=1,2.
\end{equation}
Note that the above reduces to the Sobolev embedding $ H^{3,2}(D_\ep)\subset C^{1}(D_{\ep})$ and $H^{2,1}(D_\ep)\subset C(D_{\ep})$.
Secondly, we obtain an uniform bound for the difference of the auxiliary functions $\Psi_i[\ff']-\Psi_i[\ff'']$ and their derivatives in terms of the $H^{4,3}(D_{\ep})$-norm of the difference $\ff'-\ff''.$

\begin{lemma}\label{l:Psi_1[f']-Psi_1[f'']}
Let  $\ff', \ff'' \in\mathbb{B}_{\d}(H^{4,3}(D_{\ep}))$ with $0<\d(\e)\ll 1$ small enough. The following bounds hold
\begin{enumerate}[i)]
	\item $$ |(\Psi_i[\ff']-\Psi_i[\ff''])(\bx,\bar{\bx})|^2 |\bar{\bx}|^{2i}\leq  C(\e) \|\ff'-\ff''\|_{H^{3}(D_{\ep})}^2, \qquad i=1,2.$$
	
	\item $$ |(\p_{\bx}\Psi_i[\ff']-\p_{\bx}\Psi_i[\ff''])(\bx,\bar{\bx})|^2 |\bar{\bx}|^{2i}\leq  C(\e) \|\ff'-\ff''\|_{H^{4,3}(D_{\ep})}^2, \qquad i=1,2.$$

\item $$ |(\p_{\bx}^2\Psi_i[\ff']-\p_{\bx}^2\Psi_i[\ff''])(\bx,\bar{\bx})|^2 |\bar{\bx}|^{2(i+1)}\leq  C(\e) \|\ff'-\ff''\|_{H^{4,3}(D_{\ep})}^2, \qquad i=1,2.$$

\end{enumerate}
\end{lemma}

\begin{proof}[Proof of \emph{i)}]
Using some trigonometric identities for hyperbolic functions $\sinh(\cdot), \cosh(\cdot)$ we get
\begin{align*}
\left(\Psi_1[\ff']-\Psi_1[\ff'']\right)(\bx,\bar{\bx})&=\frac{\sinh(\D_{\bar{\bx}}[\ff'-\ff''](\bx))-\cos(\bar{x})\cosh(\bar{y})\left[\sinh(\D_{\bar{\bx}}[\ff'](\bx))-\sinh(\D_{\bar{\bx}}[\ff''](\bx))\right]}{\left[\cosh(\bar{y}+\D_{\bar{\bx}}[\ff'](\bx))-\cos(\bar{x})\right]\left[\cosh(\bar{y}+\D_{\bar{\bx}}[\ff''](\bx))-\cos(\bar{x})\right]}\nonumber\\
&\quad -\frac{\cos(\bar{x})\sinh(\bar{y})\left[\cosh(\D_{\bar{\bx}}[\ff'](\bx))-\cosh(\D_{\bar{\bx}}[\ff''](\bx))\right]}{\left[\cosh(\bar{y}+\D_{\bar{\bx}}[\ff'](\bx))-\cos(\bar{x})\right]\left[\cosh(\bar{y}+\D_{\bar{\bx}}[\ff''](\bx))-\cos(\bar{x})\right]},
\end{align*}
and
\begin{align*}
\left(\Psi_2[\ff']-\Psi_2[\ff'']\right)(\bx,\bar{\bx})&=\frac{\cos(\bar{x})\left[\cosh(\bar{y}+\D_{\bar{\bx}}[\ff''](\bx))-\cosh(\bar{y}+\D_{\bar{\bx}}[\ff'](\bx))\right]}{\left[\cosh(\bar{y}+\D_{\bar{\bx}}[\ff'](\bx))-\cos(\bar{x})\right]\left[\cosh(\bar{y}+\D_{\bar{\bx}}[\ff''](\bx))-\cos(\bar{x})\right]}.
\end{align*}
Notice that both denominators vanish if and only if  $\bar{\bx}=0$. Now, since $\ff',\ff''\in H^{4,3}(D_{\ep})\subset C^2(D_{\ep})$ are continuous functions on a bounded domain, the outer region is easily bounded by the required term. Then, we just focus our attention on the inner region $|\bar{\bx}|\ll 1$.

Applying Taylor expansion of each trigonometric function involved we obtain that each numerator of the previous expression can be bound as follows
\begin{multline*}
\left|\sinh(\D_{\bar{\bx}}[\ff'-\ff''](\bx))-\cos(\bar{x})\cosh(\bar{y})\left[\sinh(\D_{\bar{\bx}}[\ff'](\bx))-\sinh(\D_{\bar{\bx}}[\ff''](\bx))\right]\right|\\
\leq C(\e) |\bar{\bx}|^3  \frac{1} {1-\max\{\| \ff' - \ff''\|_{H^{3}},\| \ff'\|_{H^{3}},\| \ff''\|_{H^{3}}\}\text{diam}(D_{\ep})} \|\ff'-\ff''\|_{H^{3}},
\end{multline*}
\begin{multline*}
\left|\cos(\bar{x})\sinh(\bar{y})\left[\cosh(\D_{\bar{\bx}}[\ff'](\bx))-\cosh(\D_{\bar{\bx}}[\ff''](\bx))\right]\right|\\
\leq C(\e) |\bar{\bx}|^3  \frac{\max\{\| \ff'\|_{H^{3}},\| \ff''\|_{H^{3}}\}} {1-\max\{\| \ff'\|_{H^{3}},\| \ff''\|_{H^{3}}\}\text{diam}(D_{\ep})} \|\ff'-\ff''\|_{H^{3}},
\end{multline*}
and
\begin{multline*}
\left|\cos(\bar{x})\left[\cosh(\bar{y}+\D_{\bar{\bx}}[\ff'](\bx))-\cosh(\bar{y}+\D_{\bar{\bx}}[\ff''](\bx))\right]\right|\\
\leq C(\e) |\bar{\bx}|^2  \frac{\max\{\| \ff'\|_{H^{3}},\| \ff''\|_{H^{3}}\}} {1-\max\{\| \ff'\|_{H^{3}},\| \ff''\|_{H^{3}}\}\text{diam}(D_{\ep})} \|\ff'-\ff''\|_{H^{3}}.
\end{multline*}
Combining all with the usual lower bound for the denominator (see \eqref{denominator}), we obtain
\[
|(\Psi_i[\ff']-\Psi_i[\ff''])(\bx,\bar{\bx})|^2 |\bar{\bx}|^{2j}
\leq   \left(\frac{C(\e,\|\ff'\|_{H^{3}(D_{\ep})},\|\ff''\|_{H^{3}(D_{\ep})})}{1-\max\{\| \ff'\|_{H^{3}},\| \ff''\|_{H^{3}}\}\text{diam}(D_{\ep})}\right)\|\ff'-\ff''\|_{H^3(D_{\ep})}^2.
\]
Finally, as $\ff', \ff'' \in\mathbb{B}_{\d}(H^{4,3}(D_{\ep})),$ taking $0<\d<\text{diam}(D_{\ep})^{-1}$  we have proved our goal.
\end{proof}

\begin{proof}[Proof of \emph{ii)}]
Due to the relations
\begin{align*}
\p_{\bx}\Psi_1[g](\bx,\bar{\bx})&=(\Psi_2[g]-\Psi_1^2[g])(\bx,\bar{\bx})\D_{\bar{\bx}}[\p_{\bx}g](\bx),\\
\p_{\bx}\Psi_2[g](\bx,\bar{\bx})&=(\Psi_1[g]-\Psi_1[g]\Psi_2[g])(\bx,\bar{\bx})\D_{\bar{\bx}}[\p_{\bx}g](\bx),
\end{align*}
it is clear that adding and subtracting some appropriate term we obtain the expressions
\begin{align}
(\p_{\bx}\Psi_1[\ff']-\p_{\bx}\Psi_1[\ff''])(\bx,\bar{\bx})&=(\Psi_2[\ff']-(\Psi_1[\ff'])^2)(\bx,\bar{\bx})\D_{\bar{\bx}}[\p_{\bx}(\ff'-\ff'')](\bx) \label{diferenceDX1_PSI1}\\
&\quad +(\Psi_2[\ff']-\Psi_2[\ff''])(\bx,\bar{\bx})\D_{\bar{\bx}}[\p_{\bx}\ff''](\bx)\nonumber\\
&\quad - (\Psi_1[\ff']-\Psi_1[\ff''])(\bx,\bar{\bx})(\Psi_1[\ff']+\Psi_1[\ff''])(\bx,\bar{\bx}) \D_{\bar{\bx}}[\p_{\bx}\ff'](\bx)\nonumber,
\end{align}
and
\begin{align}
(\p_{\bx}\Psi_2[\ff']-\p_{\bx}\Psi_2[\ff''])(\bx,\bar{\bx})&=(\Psi_1[\ff']-\Psi_1[\ff']\Psi_2[\ff'])(\bx,\bar{\bx})\D_{\bar{\bx}}[\p_{\bx}(\ff'-\ff'')](\bx)\label{diferenceDX1_PSI2}\\
&\quad + (\Psi_1[\ff']-\Psi_1[\ff''] )(\bx,\bar{\bx})\D_{\bar{\bx}}[\p_{\bx}\ff''](\bx)\nonumber\\
&\quad - \Psi_1[\ff'](\bx,\bar{\bx})(\Psi_2[\ff']-\Psi_2[\ff''] )(\bx,\bar{\bx})\D_{\bar{\bx}}[\p_{\bx}\ff''](\bx)\nonumber\\
&\quad - \Psi_2[\ff''](\bx,\bar{\bx})(\Psi_1[\ff']-\Psi_1[\ff''] )(\bx,\bar{\bx})\D_{\bar{\bx}}[\p_{\bx}\ff''](\bx).\nonumber
\end{align}
As before, we just focus on the inner region $|\bar{\bx}|\ll 1,$ which is the most singular part. Now, remembering \eqref{bound_f'-f''} and applying repeatedly Corollary \ref{boundsPSI}, we get for $1\leq i \leq 2 $ and $|\bar{\bx}|\ll 1$ that
\begin{multline*}
|(\p_{\bx}\Psi_i[\ff']-\p_{\bx}\Psi_i[\ff''])(\bx,\bar{\bx})||\bar{\bx}|^{i}\\
\lesssim \|\ff'-\ff''\|_{H^{4,3}(D_{\ep})}+ |\bar{\bx}||(\Psi_1[\ff']-\Psi_1[\ff''])(\bx,\bar{\bx})| +|\bar{\bx}|^2|(\Psi_2[\ff']-\Psi_2[\ff''])(\bx,\bar{\bx})|.
\end{multline*}
Finally, applying i) in each of the last two terms of the above expressions we have proved ii).
\end{proof}
\begin{proof}[Proof of iii)] Taking one derivative on \eqref{diferenceDX1_PSI1}, \eqref{diferenceDX1_PSI2} and adding and subtracting appropriate terms, we obtain  the expressions
\begin{align*}
(\p_{\bx}^2\Psi_1[\ff']-\p_{\bx}^2\Psi_1[\ff'])(\bx,\bar{\bx})&=(\Psi_2[\ff']-\Psi_1^2[\ff''])(\bx,\bar{\bx})\D_{\bar{\bx}}[\p_{\bx}^2(\ff'-\ff'')](\bx)\\
&\quad + (\Psi_1[\ff']-3\Psi_1[\ff'] \Psi_2[\ff']+2\Psi_1^3[\ff'])(\bx,\bar{\bx}) \D_{\bar{\bx}}[\p_x\ff'](\bx)\D_{\bar{\bx}}[\p_{\bx}(\ff'-\ff'')](\bx)\\
&\quad + (\Psi_2[\ff']-\Psi_2[\ff''])(\bx,\bar{\bx})\D_{\bar{\bx}}[\p_{\bx}^2\ff''](\bx)\\
&\quad + (\p_\bx\Psi_2[\ff']-\p_{\bx}\Psi_2[\ff''])(\bx,\bar{\bx})\D_{\bar{\bx}}[\p_{\bx}\ff''](\bx)\\
&\quad - (\Psi_1[\ff']-\Psi_1[\ff''])(\Psi_1[\ff']+\Psi_1[\ff''])(\bx,\bar{\bx}) \D_{\bar{\bx}}[\p_{\bx}^2\ff'](\bx),\\
&\quad - (\p_{\bx}\Psi_1[\ff']-\p_{\bx}\Psi_1[\ff''])(\Psi_1[\ff']+\Psi_1[\ff''])(\bx,\bar{\bx}) \D_{\bar{\bx}}[\p_{\bx}\ff'](\bx),\\
&\quad - (\Psi_1[\ff']-\Psi_1[\ff''])(\Psi_2[\ff']-\Psi_1^2[\ff'])(\bx,\bar{\bx})\D_{\bar{\bx}}[\p_{\bx}\ff'](\bx) \D_{\bar{\bx}}[\p_{\bx}\ff'](\bx),\\
&\quad - (\Psi_1[\ff']-\Psi_1[\ff''])(\Psi_2[\ff'']-\Psi_1^2[\ff''])(\bx,\bar{\bx})\D_{\bar{\bx}}[\p_{\bx}\ff''](\bx) \D_{\bar{\bx}}[\p_{\bx}\ff'](\bx),\\
\end{align*}
and
\begin{align*}
(\p_{\bx}^2\Psi_2[\ff']-\p_{\bx}^2\Psi_2[\ff''])(\bx,\bar{\bx})&=(\Psi_1[\ff']-\Psi_1[\ff']\Psi_2[\ff'])(\bx,\bar{\bx})\D_{\bar{\bx}}[\p_{\bx}^2(\ff'-\ff'')](\bx)\\
&\quad +(\Psi_2[\ff']-2\Psi_1^2[\ff'])(1-\Psi_2[\ff'])(\bx,\bar{\bx})\D_{\bar{\bx}}[\p_{\bx}\ff'](\bx)\D_{\bar{\bx}}[\p_{\bx}(\ff'-\ff'')](\bx)\\
%&\quad +(\Psi_2[\ff']-2\Psi_1^2[\ff']+2\Psi_1^2[\ff']\Psi_2[\ff']-\Psi_2^2[\ff'])(\bx,\bar{\bx})\D_{\bar{\bx}}[\p_{\bx}\ff'](\bx)\D_{\bar{\bx}}[\p_{\bx}(\ff'-\ff'')](\bx)\\
&\quad + (\Psi_1[\ff']-\Psi_1[\ff''] )(\bx,\bar{\bx})\D_{\bar{\bx}}[\p_{\bx}^2\ff''](\bx)\\
&\quad + (\p_{\bx}\Psi_1[\ff']-\p_{\bx}\Psi_1[\ff''] )(\bx,\bar{\bx})\D_{\bar{\bx}}[\p_{\bx}\ff''](\bx)\\
&\quad - \Psi_1[\ff'](\Psi_2[\ff']-\Psi_2[\ff''])(\bx,\bar{\bx})\D_{\bar{\bx}}[\p_{\bx}^2\ff''](\bx)\\
&\quad - \Psi_1[\ff'](\p_{\bx}\Psi_2[\ff']-\p_{\bx}\Psi_2[\ff''] )(\bx,\bar{\bx})\D_{\bar{\bx}}[\p_{\bx}\ff''](\bx)\\
&\quad - (\Psi_2[\ff']-\Psi_1^2[\ff'])(\Psi_2[\ff']-\Psi_2[\ff''])(\bx,\bar{\bx}) \D_{\bar{\bx}}[\p_{\bx}\ff'](\bx) \D_{\bar{\bx}}[\p_{\bx}\ff''](\bx)\\
&\quad - \Psi_2[\ff''](\bx,\bar{\bx})\left(\Psi_1[\ff']-\Psi_1[\ff''] \right)(\bx,\bar{\bx})\D_{\bar{\bx}}[\p_{\bx}^2\ff''](\bx)\\
&\quad - \Psi_2[\ff''](\bx,\bar{\bx})(\p_{\bx}\Psi_1[\ff']-\p_{\bx}\Psi_1[\ff''] )(\bx,\bar{\bx})\D_{\bar{\bx}}[\p_{\bx}\ff''](\bx)\\
&\quad - (\Psi_1[\ff'']-\Psi_1[\ff'']\Psi_2[\ff''])\left(\Psi_1[\ff']-\Psi_1[\ff''] \right)(\bx,\bar{\bx})\left(\D_{\bar{\bx}}[\p_{\bx}\ff''](\bx)\right)^2.
\end{align*}
The computations are very long and tedious but share lot of similarities. For this reason we shall focus only on $(\p_{\bx}^2\Psi_2[\ff']-\p_{\bx}^2\Psi_2[\ff''])(\bx,\bar{\bx})$ to illustrate how the estimates work. The ideas is to write the above long expressions in groups of terms as
\begin{itemize}
	\item $(1-\Psi_2[\ff'])(\Psi_1[\ff'] \D_{\bar{\bx}}[\p_{\bx}^2(\ff'-\ff'')](\bx)
+(\Psi_2[\ff']-2\Psi_1^2[\ff'])\D_{\bar{\bx}}[\p_{\bx}\ff'](\bx)\D_{\bar{\bx}}[\p_{\bx}(\ff'-\ff'')](\bx)),$

	\item $(\Psi_1[\ff']-\Psi_1[\ff''] )(\D_{\bar{\bx}}[\p_{\bx}^2\ff''](\bx)- (\Psi_1[\ff'']-\Psi_1[\ff'']\Psi_2[\ff''])\left(\D_{\bar{\bx}}[\p_{\bx}\ff''](\bx)\right)^2 - \Psi_2[\ff'']\D_{\bar{\bx}}[\p_{\bx}^2\ff''](\bx)),$
	
	\item $(\Psi_2[\ff']-\Psi_2[\ff''])\left( -\Psi_1[\ff']\D_{\bar{\bx}}[\p_{\bx}^2\ff''](\bx)- (\Psi_2[\ff']-\Psi_1^2[\ff'])\D_{\bar{\bx}}[\p_{\bx}\ff'](\bx) \D_{\bar{\bx}}[\p_{\bx}\ff''](\bx)\right),$
	
	\item $(\p_{\bx}\Psi_1[\ff']-\p_{\bx}\Psi_1[\ff''] )(1-\Psi_2[\ff''])\D_{\bar{\bx}}[\p_{\bx}\ff''](\bx),$
	
	\item $(\p_{\bx}\Psi_2[\ff']-\p_{\bx}\Psi_2[\ff''] ) (- \Psi_1[\ff'] \D_{\bar{\bx}}[\p_{\bx}\ff''](\bx)).$
\end{itemize}
Remembering \eqref{bound_f'-f''} and applying repeatedly Corollary \ref{boundsPSI}, we get for $1\leq i \leq 2 $ and $|\bar{\bx}|\ll 1$ that
\begin{align*}
|(\p_{\bx}^2\Psi_i[\ff']-\p_{\bx}^2\Psi_i[\ff''])(\bx,\bar{\bx})||\bar{\bx}|^{i}
&\lesssim |\bar{\bx}|^{-1}\|\ff'-\ff''\|_{H^{4,3}(D_{\ep})}+ |(\Psi_1[\ff']-\Psi_1[\ff''])(\bx,\bar{\bx})|\\
&\quad +|\bar{\bx}||(\Psi_2[\ff']-\Psi_2[\ff''])(\bx,\bar{\bx})|+|\bar{\bx}||(\p_{\bx}\Psi_1[\ff']-\p_{\bx}\Psi_1[\ff''])(\bx,\bar{\bx})|\\
&\quad +|\bar{\bx}|^2|(\p_{\bx}\Psi_2[\ff']-\p_{\bx}\Psi_2[\ff''])(\bx,\bar{\bx})|.
\end{align*}
Finally, applying i) and ii) in each of the terms of the above expressions we have proved iii).
\end{proof}

At this point, we have all the ingredients to  prove the main lemmas of the Appendix. The first one allows us to take derivatives into the kernel $\p_{\bx}^i K[g]$ as follows:

\begin{lemma}\label{l:DK[f]} Let  $g \in\mathbb{B}_{\d}(H^{4,3}(D_{\ep}))$ with $0<\d(\e)\ll 1$ small enough. The following bounds hold
\begin{enumerate}[i)]
	\item $$\sup_{\bx\in D_{\ep}}\left( \sup_{\bar{\bx}\in D_\ep(y)} |K[g](\bx,\bar{\bx})|^2\right)\leq C(\e,\| g\|_{H^{4,3}(D_{\ep})}),$$
	
	\item $$\sup_{\bx\in D_{\ep}}\left(\int_{D_\ep(y)}|\p_{\bx} K[g](\bx,\bar{\bx})|^2\mbox{d}\bar{\bx} \right) \leq C(\e,\|g\|_{H^{4,3}(D_{\ep})}),$$
	
	\item $$\int_{D_{\ep}}\left(\int_{D_\ep(y)}\left|\p_{\bx}^2 K[g](\bx,\bar{\bx}) \right|^2 |\bar{\bx}|^{2\g}\mbox{d}\bar{\bx} \right)\mbox{d}\bx\leq  C(\e,\| g\|_{H^{4,3}(D_{\ep})}),  $$
	
	\item $$\int_{D_{\ep}}\left(\int_{D_\ep(y)}\left|\p_{\bx}^3 K[g](\bx,\bar{\bx}) \right|^2 |\bar{\bx}|^{2}\mbox{d}\bar{\bx} \right)\mbox{d}\bx\leq  C(\e,\| g\|_{H^{4,3}(D_{\ep})}).$$
\end{enumerate}
\end{lemma}
\begin{proof}
As we can see the computations share lot of similarities. For this reason we shall focus only on one significant term iii) to illustrate how the estimates work. Remembering \eqref{p2K} we get
\[
\int_{D_{\ep}}\left(\int_{D_\ep(y)}\left|\p_{\bx}^2 K[g](\bx,\bar{\bx}) \right|^2 |\bar{\bx}|^{2\g}\mbox{d}\bar{\bx} \right)\mbox{d}\bx \leq \sum_{i=1}^{2} \int_{D_{\ep}}\left(\int_{D_\ep(y)}\left|\II_i[g](\bx,\bar{\bx}) \right|^2 |\bar{\bx}|^{2\g}\mbox{d}\bar{\bx} \right)\mbox{d}\bx
\]
Proceeding as usual, we split the integral into inner and outer regions. As $g\in H^{4,3}(D_{\ep})\subset C^{2}(D_{\ep})$ is a continuous function on a bounded domain the outer integral is trivially bounded by some universal constant. To handle the remaining inner integral $\{(\bx,\bar{\bx})\in D_{\ep}\times D_{\ep}(y):|\bar{\bx}|\ll 1\}$ of each term, we use Corollary \ref{boundstilde} as follows:
\begin{enumerate}
	\item[$\II_1)$] As $g\in H^{4,3}(D_{\ep})$ we have that $\p_{\bx}g \in H^{3,2}(D_{\ep})\subset C^1(D_{\ep})$ and consequently we get
\begin{align*}
\int_{D_{\ep}}\left(\int_{|\bar{\bx}|\ll 1}\left|\II_1[g](\bx,\bar{\bx}) \right|^2 |\bar{\bx}|^{2\g}\mbox{d}\bar{\bx} \right)\mbox{d}\bx &=  \int_{D_{\ep}}\left(\int_{|\bar{\bx}|\ll 1}\left|\tilde{\II}_1[g](\bx,\bar{\bx})\right|^2 \left|\D_{\bar{\bx}}[\p_{\bx}g](\bx)\right|^4 |\bar{\bx}|^{2\g}\mbox{d}\bar{\bx} \right)\mbox{d}\bx \\
&\leq C(\ep,\|g \|_{H^{4,3}(D_{\ep})})\left( \int_{|\bar{\bx}|\ll 1} |\bar{\bx}|^{2\g}\mbox{d}\bar{\bx}\right),
\end{align*}

	\item[$\II_2)$] As $g\in H^{4,3}(D_{\ep})$ we have that $\p_{\bx}^2g \in H^{2,1}(D_{\ep})\subset C(D_{\ep})$ and consequently we get
\begin{align*}
\int_{D_{\ep}}\left(\int_{|\bar{\bx}|\ll 1}\left|\II_2[g](\bx,\bar{\bx}) \right|^2 |\bar{\bx}|^{2\g}\mbox{d}\bar{\bx} \right)\mbox{d}\bx &=  \int_{D_{\ep}}\left(\int_{|\bar{\bx}|\ll 1}\left|\tilde{\II}_2[g](\bx,\bar{\bx})\right|^2 \left|\D_{\bar{\bx}}[\p_{\bx}^2 g](\bx)\right|^2 |\bar{\bx}|^{2\g}\mbox{d}\bar{\bx} \right)\mbox{d}\bx \\
&\leq C(\d,\|g \|_{H^{4,3}(D_{\ep})})\left( \int_{|\bar{\bx}|\ll 1} \frac{1}{|\bar{\bx}|^{2(1-\g)}}\mbox{d}\bar{\bx}\right).
\end{align*}
\end{enumerate}
As $0<\g<1,$ the last term is integrable and we have proved our goal. The rest of the terms i), ii) and iv) work in a similar way and we omit the details.
\end{proof}

Now, we continue with an analog result for the difference $\p_{\bx}^i K[\ff']-\p_{\bx}^iK[\ff'']$.

\begin{lemma}\label{l:DK[f',f'']} Let  $\ff', \ff'' \in\mathbb{B}_{\d}(H^{4,3}(D_{\ep}))$ with $0<\d(\e)\ll 1$ small enough. The following bounds hold
\begin{enumerate}[i)]
	\item $$\sup_{\bx\in D_{\ep}}\left( \sup_{\bar{\bx}\in D_\ep(y)} |\cK[\ff',\ff''](\bx,\bar{\bx})|^2\right)\leq C(\e) \|\ff'-\ff''\|_{H^{4,3}(D_{\ep})}^2.$$
	
	\item $$\sup_{\bx\in D_{\ep}}\left(\int_{D_\ep(y)}|\p_{\bx}\cK[\ff',\ff''](\bx,\bar{\bx})|^2\mbox{d}\bar{\bx} \right) \leq C(\e) \|\ff'-\ff''\|_{H^{4,3}(D_{\ep})}^2.$$
	
	\item $$\int_{D_{\ep}}\left(\int_{D_\ep(y)}\left|\p_{\bx}^2 \cK[\ff',\ff''](\bx,\bar{\bx}) \right|^2 |\bar{\bx}|^{2\g}\mbox{d}\bar{\bx} \right)\mbox{d}\bx\leq C(\e) \|\ff'-\ff''\|_{H^{4,3}(D_{\ep})}^2. $$
	
	\item $$\int_{D_{\ep}}\left(\int_{D_\ep(y)}\left|\p_{\bx}^3 \cK[\ff',\ff''](\bx,\bar{\bx}) \right|^2 |\bar{\bx}|^{2}\mbox{d}\bar{\bx} \right)\mbox{d}\bx\leq C(\e) \|\ff'-\ff''\|_{H^{4,3}(D_{\ep})}^2.$$
\end{enumerate}
\end{lemma}
\begin{proof}[Proof of \emph{i)}]
We start by remembering the kernel definition $\cK[\ff',\ff'']=K[\ff']-K[\ff'']$ given by
\[
\cK[\ff',\ff''](\bx,\bar{\bx})=\log\left[\frac{\cosh\left(\bar{y}+\D_{\bar{\bx}}[\ff'](\bx)\right)-\cos(\bar{x})}{\cosh\left(\bar{y}+\D_{\bar{\bx}}[\ff''](\bx)\right)-\cos(\bar{x})}\right].
\]
Adding and subtracting some appropriate term, one finds
\[
\frac{\cosh\left(\bar{y}+\D_{\bar{\bx}}[\ff'](\bx)\right)-\cos(\bar{x})}{\cosh\left(\bar{y}+\D_{\bar{\bx}}[\ff''](\bx)\right)-\cos(\bar{x})}=1+\frac{\cosh\left(\bar{y}+\D_{\bar{\bx}}[\ff'](\bx)\right)-\cosh\left(\bar{y}+\D_{\bar{\bx}}[\ff''](\bx)\right)}{\cosh\left(\bar{y}+\D_{\bar{\bx}}[\ff''](\bx)\right)-\cos(\bar{x})},
\]
and using the standard logarithmic inequality $1-x^{-1}\leq \log x\leq x-1$ for all $x>0$, we get the lower and upper bounds
\begin{align*}
\cK[\ff',\ff''](\bx,\bar{\bx})&\leq \frac{\cosh\left(\bar{y}+\D_{\bar{\bx}}[\ff'](\bx)\right)-\cosh\left(\bar{y}+\D_{\bar{\bx}}[\ff''](\bx)\right)}{\cosh\left(\bar{y}+\D_{\bar{\bx}}[\ff''](\bx)\right)-\cos(\bar{x})},\\
\cK[\ff',\ff''](\bx,\bar{\bx})&\geq \frac{\cosh\left(\bar{y}+\D_{\bar{\bx}}[\ff'](\bx)\right)-\cosh\left(\bar{y}+\D_{\bar{\bx}}[\ff''](\bx)\right)}{\cosh\left(\bar{y}+\D_{\bar{\bx}}[\ff'](\bx)\right)-\cos(\bar{x})}.
\end{align*}
As $\ff', \ff'' \in\mathbb{B}_{\d}(H^{4,3}(D_{\ep}))$ are arbitrary functions we can assume without loss of generality that
\[
|\cK[\ff',\ff''](\bx,\bar{\bx})|\leq \left|\frac{\cosh\left(\bar{y}+\D_{\bar{\bx}}[\ff'](\bx)\right)-\cosh\left(\bar{y}+\D_{\bar{\bx}}[\ff''](\bx)\right)}{\cosh\left(\bar{y}+\D_{\bar{\bx}}[\ff''](\bx)\right)-\cos(\bar{x})}\right|,
\]
and using the trigonometric identity $\cosh(a+b)=\cosh(a)\cosh(b)+\sinh(a)\sinh(b)$ we finally get
\begin{align*}
\frac{\cosh\left(\bar{y}+\D_{\bar{\bx}}[\ff'](\bx)\right)-\cosh\left(\bar{y}+\D_{\bar{\bx}}[\ff''](\bx)\right)}{\cosh\left(\bar{y}+\D_{\bar{\bx}}[\ff''](\bx)\right)-\cos(\bar{x})}&=\cosh(\bar{y})\frac{\cosh\left(\D_{\bar{\bx}}[\ff'](\bx)\right)-\cosh\left(\D_{\bar{\bx}}[\ff''](\bx)\right)}{\cosh\left(\bar{y}+\D_{\bar{\bx}}[\ff''](\bx)\right)-\cos(\bar{x})}\\
&\quad +\sinh(\bar{y})\frac{\sinh\left(\D_{\bar{\bx}}[\ff'](\bx)\right)-\sinh\left(\D_{\bar{\bx}}[\ff''](\bx)\right)}{\cosh\left(\bar{y}+\D_{\bar{\bx}}[\ff''](\bx)\right)-\cos(\bar{x})}.
\end{align*}
Now, we will use the Taylor expansion of the trigonometric functions $\cosh(\cdot)$ and $\sinh(\cdot)$ together with the algebraic identity  $a^n-b^n=(a-b)\sum_{k=0}^{n-1}a^k b^{n-1-k}$ to obtain
\begin{align*}
\cosh\left(\D_{\bar{\bx}}[\ff'](\bx)\right)-\cosh\left(\D_{\bar{\bx}}[\ff''](\bx)\right)&=\D_{\bar{\bx}}[\ff'-\ff''](\bx)\sum_{n=1}^{\infty}\sum_{k=0}^{2n-1} \frac{\left(\D_{\bar{\bx}}[\ff'](\bx)\right)^{k}\left(\D_{\bar{\bx}}[\ff''](\bx)\right)^{2n-1-k}}{(2n)!},\\
\sinh\left(\D_{\bar{\bx}}[\ff'](\bx)\right)-\sinh\left(\D_{\bar{\bx}}[\ff''](\bx)\right)&=\D_{\bar{\bx}}[\ff'-\ff''](\bx)\sum_{n=0}^{\infty}\sum_{k=0}^{2n} \frac{\left(\D_{\bar{\bx}}[\ff'](\bx)\right)^{k}\left(\D_{\bar{\bx}}[\ff''](\bx)\right)^{2n-k}}{(2n+1)!}.
\end{align*}
Computing the infinite sum of the geometric series  $\sum_{n=0}^{\infty}r^{2n}$ with
$$r:=\max\{\| \ff'\|_{H^{4,3}(D_{\ep})},\| \ff''\|_{H^{4,3}(D_{\ep})}\}\text{diam}(D_{\ep}),$$
we finally get (under  condition $\d<\text{diam}(D_{\ep})^{-1}$) that
\begin{align*}
\left|\cosh\left(\D_{\bar{\bx}}[\ff'](\bx)\right)-\cosh\left(\D_{\bar{\bx}}[\ff''](\bx)\right)\right|&\leq C(\e) |\bar{\bx}|^2  \frac{ \max\{\| \ff'\|_{H^{4,3}(D_{\ep})},\| \ff''\|_{H^{4,3}(D_{\ep})}\} \|\ff'-\ff''\|_{H^{4,3}(D_{\ep})}} {1-\max\{\| \ff'\|_{H^{4,3}(D_{\ep})},\| \ff''\|_{H^{4,3}(D_{\ep})}\}\text{diam}(D_{\ep})} ,\\
\left|\sinh\left(\D_{\bar{\bx}}[\ff'](\bx)\right)-\sinh\left(\D_{\bar{\bx}}[\ff''](\bx)\right)\right|&\leq C(\e) |\bar{\bx}|  \frac{ \|\ff'-\ff''\|_{H^{4,3}(D_{\ep})}} {1-\max\{\| \ff'\|_{H^{4,3}(D_{\ep})},\| \ff''\|_{H^{4,3}(D_{\ep})}\}\text{diam}(D_{\ep})},
\end{align*}
where we have applied the Sobolev embedding $L^{\infty}(D_{\ep})\hookrightarrow H^2(D_{\ep}).$

Now, it is not difficult to see that the following uniform bound holds
\[
\sup_{\bx\in D_{\ep}}\left( \sup_{\bar{\bx}\in D_\ep(y)} \left|\frac{\cosh(\bar{y})|\bar{\bx}|^2+\sinh(\bar{y})|\bar{\bx}|}{\cosh(\bar{y}+\D_{\bar{\bx}}[\ff''](\bx))-\cos(\bar{x})}\right|^2\right)\leq C(\e)
\]
thanks to the fact (see \eqref{denominator}) that for $|\bar{\bx}|\ll 1$ we have
\[
|\cosh(\bar{y}+\D_{\bar{\bx}}[\ff''](\bx))-\cos(\bar{x})|\geq \left(\tfrac{1}{2}-\|\ff''\|_{H^{3}(D_{\ep})}\right) |\bar{\bx}|^2-C(\|\ff''\|_{H^{3}(D_{\ep})}) |\bar{\bx}|^4.
\]
Then, taking $0<\d< \text{diam}(D_{\ep})^{-1} $ small enough and combining all we have proved our goal.
\end{proof}

\begin{proof}[Proof of \emph{ii)}] As $\cK[\ff',\ff''](\bx,\bar{\bx})=( K[\ff']-K[\ff''])(\bx,\bar{\bx})$ and $\p_{\bx}K[g](\bx,\bar{\bx})=\Psi_1[g](\bx,\bar{\bx})\D_{\bar{\bx}}[\p_{\bx}g](\bx)$, adding and subtracting some appropriate term we obtain
\begin{equation}\label{e:p1K[f',f'']}
\p_{\bx}\cK[\ff',\ff''](\bx,\bar{\bx})=\Psi_1[\ff'](\bx,\bar{\bx})\D_{\bar{\bx}}[\p_{\bx}(\ff'-\ff'')](\bx)+(\Psi_1[\ff']-\Psi_1[\ff''])(\bx,\bar{\bx}) \D_{\bar{\bx}}[\p_{\bx}\ff''](\bx).
\end{equation}
To bound the first term we use $\ff',\ff''\in\mathbb{B}_{\d}(H^{4,3}(D_{\ep}))$, which implies  $\p_{\bx}(\ff'-\ff'')\in H^{3,2}(D_\ep)\subset C^{1}(D_{\ep})$. As usual, splitting the integral and using  Corollary \ref{boundsPSI} in the inner region we obtain directly that
\[
\int_{D_\ep(y)}|\Psi_1[\ff'](\bx,\bar{\bx})\D_{\bar{\bx}}[\p_{\bx}(\ff'-\ff'')](\bx)|^2\mbox{d}\bar{\bx}\\
\leq  C(\e,\|\ff'\|_{H^{4,3}(D_{\ep})})\|\ff'-\ff''\|_{H^{4,3}(D_{\ep})}^2.
\]
For the second term of \eqref{e:p1K[f',f'']}, using that $\p_{\bx}\ff''\in H^{3,2}(D_{\ep})\subset C^{1}(D_{\ep})$ and applying Lemma \ref{l:Psi_1[f']-Psi_1[f'']} we get
\begin{align*}
\int_{D_\ep(y)}|(\Psi_1[\ff']-\Psi_1[\ff''])(\bx,\bar{\bx}) \D_{\bar{\bx}}[\p_{\bx}\ff''](\bx)|^2\mbox{d}\bar{\bx} &\leq \|\ff''\|_{H^{4,3}(D_{\ep})} \int_{D_\ep(y)}|(\Psi_1[\ff']-\Psi_1[\ff''])(\bx,\bar{\bx})|^2 |\bar{\bx}|^2\mbox{d}\bar{\bx}\\
&\leq C(\e,\|\ff'\|_{H^{4,3}(D_{\ep})},\|\ff''\|_{H^{4,3}(D_{\ep})}) \|\ff'-\ff''\|_{H^{4,3}(D_{\ep})}^2.
\end{align*}
Finally, as the above holds for any $\bx\in D_{\ep}$, combining the above and taking the supremum over all the domain we have proved the desired inequality.
\end{proof}

\begin{proof}[Proof of \emph{iii)}] Taking a derivative of \eqref{e:p1K[f',f'']} we obtain
\begin{align}\label{e:p2K[f',f'']}
\p_{\bx}^2\cK[\ff',\ff''](\bx,\bar{\bx})&=\p_{\bx}\Psi_1[\ff'](\bx,\bar{\bx})\D_{\bar{\bx}}[\p_{\bx}(\ff'-\ff'')](\bx)+\Psi_1[\ff'](\bx,\bar{\bx})\D_{\bar{\bx}}[\p_{\bx}^2(\ff'-\ff'')](\bx) \nonumber\\
&\quad +(\p_{\bx}\Psi_1[\ff']-\p_{\bx}\Psi_1[\ff''])(\bx,\bar{\bx}) \D_{\bar{\bx}}[\p_{\bx}\ff''](\bx)+(\Psi_1[\ff']-\Psi_1[\ff''])(\bx,\bar{\bx}) \D_{\bar{\bx}}[\p_{\bx}^2\ff''](\bx).
\end{align}
To bound the first term of \eqref{e:p2K[f',f'']}, as $\ff',\ff''\in\mathbb{B}_{\d}(H^{4,3}(D_{\ep}))$ we get  $\p_{\bx}\ff', \p_{\bx}(\ff'-\ff'')\in H^{3,2}(D_\ep)\subset C^{1}(D_{\ep})$. Now,  splitting the integral and using \eqref{DPsi1} together with Corollary \ref{boundsPSI} in the inner region we obtain
\[
\int_{D_{\ep}}\left(\int_{|\bar{\bx}|\ll 1}\left|\p_{\bx}\Psi_1[\ff'](\bx,\bar{\bx})\D_{\bar{\bx}}[\p_{\bx}(\ff'-\ff'')](\bx) \right|^2 |\bar{\bx}|^{2\g}\mbox{d}\bar{\bx} \right)\mbox{d}\bx\leq  \frac{C(\|\ff' \|_{H^{4,3}(D_{\ep})})}{ 1-C(\|\ff' \|_{H^{4,3}(D_{\ep})}) } \|\ff'-\ff'' \|_{H^{4,3}(D_{\ep})}^2.
\]
To bound the second term of  \eqref{e:p2K[f',f'']} we use the same type of ideas. As $\ff',\ff''\in\mathbb{B}_{\d}(H^{4,3}(D_{\ep}))$ we get  $\p_{\bx}^2(\ff'-\ff'')\in H^{2,1}(D_\ep)\subset C(D_{\ep})$. Now,  splitting the integral and applying Corollary \ref{boundsPSI} in the inner region we obtain
\begin{multline*}
\int_{D_{\ep}}\left(\int_{|\bar{\bx}|\ll 1}\left|\Psi_1[\ff'](\bx,\bar{\bx})\D_{\bar{\bx}}[\p_{\bx}^2(\ff'-\ff'')](\bx) \right|^2 |\bar{\bx}|^{2\g}\mbox{d}\bar{\bx} \right)\mbox{d}\bx\\
\leq  \frac{C(\|\ff' \|_{H^{4,3}(D_{\ep})})}{ 1-C(\|\ff' \|_{H^{4,3}(D_{\ep})}) } \|\ff'-\ff'' \|_{H^{4,3}(D_{\ep})}^2\left( \int_{|\bar{\bx}|\ll 1}\frac{1}{|\bar{\bx}|^{2(1-\g)}}\mbox{d}\bar{\bx}\right),
\end{multline*}
where the last integral is bounded thanks to the fact that $0<\g<1.$ Finally, to bound the last two terms of \eqref{e:p2K[f',f'']} we use the fact that $\p_{\bx}\ff'' \in H^{3,2}(D_{\ep})\subset C^1(D_{\ep})$ and $\p_{\bx}^2\ff'' \in H^{2,1}(D_{\ep})\subset C(D_{\ep})$ to get
\begin{multline*}
\int_{D_{\ep}}\left(\int_{D_\ep(y)}\left|(\p_{\bx}\Psi_1[\ff']-\p_{\bx}\Psi_1[\ff''])(\bx,\bar{\bx}) \D_{\bar{\bx}}[\p_{\bx}\ff''](\bx) \right|^2 |\bar{\bx}|^{2\g}\mbox{d}\bar{\bx} \right)\mbox{d}\bx\\
\leq \|\ff'' \|_{H^{4,3}(D_{\ep})} \int_{D_{\ep}}\left(\int_{D_\ep(y)}\left|(\p_{\bx}\Psi_1[\ff']-\p_{\bx}\Psi_1[\ff''])(\bx,\bar{\bx})  \right|^2 |\bar{\bx}|^{2(1+\g)}\mbox{d}\bar{\bx} \right)\mbox{d}\bx\,
\end{multline*}
and
\begin{multline*}
\int_{D_{\ep}}\left(\int_{D_\ep(y)}\left|(\Psi_1[\ff']-\Psi_1[\ff''])(\bx,\bar{\bx}) \D_{\bar{\bx}}[\p_{\bx}^2\ff''](\bx) \right|^2 |\bar{\bx}|^{2\g}\mbox{d}\bar{\bx} \right)\mbox{d}\bx\\
\leq \|\ff'' \|_{H^{4,3}(D_{\ep})}\int_{D_{\ep}}\left(\int_{D_\ep(y)}\left|(\Psi_1[\ff']-\Psi_1[\ff''])(\bx,\bar{\bx})  \right|^2 |\bar{\bx}|^{2\g}\mbox{d}\bar{\bx} \right)\mbox{d}\bx.
\end{multline*}
Using auxiliary Lemma \ref{l:Psi_1[f']-Psi_1[f'']} on the last terms of the above expressions we have proved our goal.
\end{proof}

\begin{proof}[Proof of \emph{iv)}] As before,  taking a derivative of \eqref{e:p2K[f',f'']} we obtain
\begin{align}\label{e:p3K[f',f'']}
\p_{\bx}^3\cK[\ff',\ff''](\bx,\bar{\bx})&=\p_{\bx}^2\Psi_1[\ff'](\bx,\bar{\bx})\D_{\bar{\bx}}[\p_{\bx}(\ff'-\ff'')](\bx)+2\p_{\bx}\Psi_1[\ff'](\bx,\bar{\bx})\D_{\bar{\bx}}[\p_{\bx}^2(\ff'-\ff'')](\bx) \nonumber\\
&\quad +\Psi_1[\ff'](\bx,\bar{\bx})\D_{\bar{\bx}}[\p_{\bx}^3(\ff'-\ff'')](\bx)+2\left(\p_{\bx}\Psi_1[\ff']-\p_{\bx}\Psi_1[\ff'']\right)(\bx,\bar{\bx}) \D_{\bar{\bx}}[\p_{\bx}^2\ff''](\bx) \nonumber\\
&\quad +\left(\Psi_1[\ff']-\Psi_1[\ff'']\right)(\bx,\bar{\bx}) \D_{\bar{\bx}}[\p_{\bx}^3\ff''](\bx)+\left(\p_{\bx}^2\Psi_1[\ff']-\p_{\bx}^2\Psi_1[\ff'']\right)(\bx,\bar{\bx}) \D_{\bar{\bx}}[\p_{\bx}\ff''](\bx).
\end{align}
In first place, due to the relations \eqref{DPsi1} and \eqref{DPsi2}, we have
\begin{align*}
\p_{\bx}^2\Psi_1[\ff']=[\Psi_1[\ff'](1-\Psi_2[\ff'])-2\Psi_1[\ff'](\Psi_2[\ff']-\Psi_1^2[\ff'])]\left(\D_{\bar{\bx}}[\p_{\bx}\ff'](\bx)\right)^2 +(\Psi_2[\ff']-\Psi_1^2[\ff'])\D_{\bar{\bx}}[\p_{\bx}^2\ff'].
\end{align*}
Remembering \eqref{bound_f'-f''} and applying repeatedly Corollary \ref{boundsPSI}, we get that the first three terms of \eqref{e:p3K[f',f'']} can be handle in the following way:
\begin{align*}
\sum_{i=1}^{3}\int_{D_{\ep}}\left(\int_{D_\ep(y)}\left|\p_{\bx}^{3-i}\Psi_1[\ff'](\bx,\bar{\bx})\D_{\bar{\bx}}[\p_{\bx}^{i}(\ff'-\ff'')](\bx) \right|^2 |\bar{\bx}|^{2}\mbox{d}\bar{\bx} \right)\mbox{d}\bx \leq  C(\e) \|\ff'-\ff''\|_{H^{4,3}(D_{\ep})}^2.
\end{align*}
The remaining three terms of \eqref{e:p3K[f',f'']} follows by auxiliary Lemma  \ref{l:Psi_1[f']-Psi_1[f'']}. Notice that using \eqref{bound_f'-f''} we get
\begin{multline*}
\int_{D_{\ep}}\left(\int_{D_\ep(y)}\left|(\Psi_1[\ff']-\Psi_1[\ff''])(\bx,\bar{\bx}) \D_{\bar{\bx}}[\p_{\bx}^3\ff''](\bx) \right|^2 |\bar{\bx}|^{2}\mbox{d}\bar{\bx} \right)\mbox{d}\bx\\
\leq \|\ff''\|_{H^{4,3}(D_{\ep})} \int_{D_{\ep}}\left(\sup_{\bar{\bx}\in D_\ep(y)}\left|(\Psi_1[\ff']-\Psi_1[\ff''])(\bx,\bar{\bx})  \right|^2 |\bar{\bx}|^{2} \right)\mbox{d}\bx,
\end{multline*}
\begin{multline*}
\int_{D_{\ep}}\left(\int_{D_\ep(y)}\left|(\p_{\bx}\Psi_1[\ff']-\p_{\bx}\Psi_1[\ff''])(\bx,\bar{\bx}) \D_{\bar{\bx}}[\p_{\bx}^2\ff''](\bx) \right|^2 |\bar{\bx}|^{2}\mbox{d}\bar{\bx} \right)\mbox{d}\bx\\
\leq \|\ff''\|_{H^{4,3}(D_{\ep})} \int_{D_{\ep}}\left(\int_{D_\ep(y)}\left|(\p_{\bx}\Psi_1[\ff']-\p_{\bx}\Psi_1[\ff''])(\bx,\bar{\bx})  \right|^2 |\bar{\bx}|^{2}\mbox{d}\bar{\bx} \right)\mbox{d}\bx,
\end{multline*}
and

\begin{multline*}
\int_{D_{\ep}}\left(\int_{D_\ep(y)}\left|(\p_{\bx}^2\Psi_1[\ff']-\p_{\bx}^2\Psi_1[\ff''])(\bx,\bar{\bx}) \D_{\bar{\bx}}[\p_{\bx}\ff''](\bx) \right|^2 |\bar{\bx}|^{2}\mbox{d}\bar{\bx} \right)\mbox{d}\bx\\
\leq \|\ff''\|_{H^{4,3}(D_{\ep})} \int_{D_{\ep}}\left(\int_{D_\ep(y)}\left|(\p_{\bx}^2\Psi_1[\ff']-\p_{\bx}^2\Psi_1[\ff''])(\bx,\bar{\bx})  \right|^2 |\bar{\bx}|^{4}\mbox{d}\bar{\bx} \right)\mbox{d}\bx.
\end{multline*}
Using auxiliary Lemma \ref{l:Psi_1[f']-Psi_1[f'']} on the last terms of the above expressions we have proved our goal.
\end{proof}

\begin{lemma}\label{l:K/TAU-PSI}
Let  $\ff \in\mathbb{B}_{\d}(H^{4,3}(D_{\ep}))$ with $0<\d(\e)\ll 1$ small enough and $\hh\in H^{4,3}(D_{\ep})$ with $\|\hh\|_{H^{4,3}(D_{\ep})}=1.$ For $0<\tau \ll 1$, the following bound holds
\begin{enumerate}[i)]
	\item  $$\sup_{\bx\in D_{\ep}}\left( \sup_{\bar{\bx}\in D_{\ep}(y)}\left|\frac{\cK[\ff+\tau \hh,\ff](\bx,\bar{\bx})}{\tau}-\Psi_1[\ff](\bx,\bar{\bx})\D_{\bar{\bx}}[\hh](\bx)\right|^2 \right) \leq C(\ep)\tau^2.$$
	
	\item $$\sup_{\bx\in D_{\ep}}\left( \int_{D_{\ep}(y)}\left|\p_{\bx}\left\lbrace\frac{\cK[\ff+\tau \hh,\ff](\bx,\bar{\bx})}{\tau}-\Psi_1[\ff](\bx,\bar{\bx})\D_{\bar{\bx}}[\hh](\bx)\right\rbrace\right|^2\mbox{d}\bar{\bx} \right)\leq C(\ep)\tau^2.$$
	
	\item $$ \int_{D_{\ep}}\left(\int_{D_\ep(y)}\left|\p_{\bx}^2\left\lbrace \frac{\cK[\ff+\tau \hh,\ff](\bx,\bar{\bx})}{\tau}-\Psi_1[\ff](\bx,\bar{\bx})\D_{\bar{\bx}}[\hh](\bx)\right\rbrace\right|^2 |\bar{\bx}|^{2\g}\mbox{d}\bar{\bx} \right)\mbox{d}\bx\leq C(\ep)\tau^2.$$
	
	\item $$ \int_{D_{\ep}}\left(\int_{D_\ep(y)}\left|\p_{\bx}^3\left\lbrace \frac{\cK[\ff+\tau \hh,\ff](\bx,\bar{\bx})}{\tau}-\Psi_1[\ff](\bx,\bar{\bx})\D_{\bar{\bx}}[\hh](\bx)\right\rbrace\right|^2 |\bar{\bx}|^{2}\mbox{d}\bar{\bx} \right)\mbox{d}\bx\leq C(\ep)\tau^2.$$
\end{enumerate}
\end{lemma}

\begin{proof}[Proof of \emph{i)}]
Just applying the standard logarithmic inequality $1-x^{-1}\leq \log x\leq x-1$ for all $x>0$, we get the lower and upper bounds
\begin{align*}
\cK[\ff+\tau \hh,\ff](\bx,\bar{\bx})&\leq \frac{\cosh\left(\bar{y}+\D_{\bar{\bx}}[\ff+\tau \hh](\bx)\right)-\cosh\left(\bar{y}+\D_{\bar{\bx}}[\ff](\bx)\right)}{\cosh\left(\bar{y}+\D_{\bar{\bx}}[\ff](\bx)\right)-\cos(\bar{x})},\\
\cK[\ff+\tau \hh,\ff](\bx,\bar{\bx})&\geq \frac{\cosh\left(\bar{y}+\D_{\bar{\bx}}[\ff+\tau \hh](\bx)\right)-\cosh\left(\bar{y}+\D_{\bar{\bx}}[\ff](\bx)\right)}{\cosh\left(\bar{y}+\D_{\bar{\bx}}[\ff+\tau \hh](\bx)\right)-\cos(\bar{x})}.
\end{align*}
Therefore, we have
\begin{multline*}
\left|\frac{\cK[\ff+\tau \hh,\ff](\bx,\bar{\bx})}{\tau}-\Psi_1[\ff](\bx,\bar{\bx})\D_{\bar{\bx}}[\hh](\bx)\right|\\
\leq \max_{0\leq\xi\leq \tau}\left|\frac{1}{\tau}\left(\frac{\cosh\left(\bar{y}+\D_{\bar{\bx}}[\ff+\tau \hh](\bx)\right)-\cosh\left(\bar{y}+\D_{\bar{\bx}}[\ff](\bx)\right)}{\cosh\left(\bar{y}+\D_{\bar{\bx}}[\ff+\xi \hh](\bx)\right)-\cos(\bar{x})}\right)-\Psi_1[\ff](\bx,\bar{\bx})\D_{\bar{\bx}}[\hh](\bx)\right|.
\end{multline*}
Now, using the trigonometric identity $\cosh(a+b)=\cosh(a)\cosh(b)+\sinh(a)\sinh(b)$ on the  above expression, we get
\begin{multline*}
\frac{\cosh\left(\bar{y}+\D_{\bar{\bx}}[\ff+\tau \hh](\bx)\right)-\cosh\left(\bar{y}+\D_{\bar{\bx}}[\ff](\bx)\right)}{\cosh\left(\bar{y}+\D_{\bar{\bx}}[\ff+\xi \hh](\bx)\right)-\cos(\bar{x})}\\
=\cosh(\bar{y}+\D_{\bar{\bx}}[\ff](x))\frac{\cosh\left(\D_{\bar{\bx}}[\tau\hh](\bx)\right)-1}{\cosh\left(\bar{y}+\D_{\bar{\bx}}[\ff+\xi\hh](\bx)\right)-\cos(\bar{x})}\\
\quad +\sinh(\bar{y}+\D_{\bar{\bx}}[\ff](x))\frac{\sinh\left(\D_{\bar{\bx}}[\tau\hh](\bx)\right)}{\cosh\left(\bar{y}+\D_{\bar{\bx}}[\ff+\xi\hh](\bx)\right)-\cos(\bar{x})},
\end{multline*}
and consequently
\begin{multline*}
\left|\frac{\cK[\ff+\tau \hh,\ff](\bx,\bar{\bx})}{\tau}-\Psi_1[\ff](\bx,\bar{\bx})\D_{\bar{\bx}}[\hh](\bx)\right|\\
\leq \left|\frac{\cosh\left(\D_{\bar{\bx}}[\tau\hh](\bx)\right)-1}{\tau}\right| \max_{0\leq\xi\leq \tau}\left|\frac{\cosh(\bar{y}+\D_{\bar{\bx}}[\ff](x))}{\cosh\left(\bar{y}+\D_{\bar{\bx}}[\ff+\xi\hh](\bx)\right)-\cos(\bar{x})}\right|\\
+\left|\frac{\sinh\left(\D_{\bar{\bx}}[\tau\hh](\bx)\right)}{\tau}-\D_{\bar{\bx}}[\hh](\bx)\right| \max_{0\leq\xi\leq \tau}\left|\frac{\sinh(\bar{y}+\D_{\bar{\bx}}[\ff](x))}{\cosh\left(\bar{y}+\D_{\bar{\bx}}[\ff+\xi\hh](\bx)\right)-\cos(\bar{x})}\right|.
\end{multline*}
Applying the Taylor expansion of the trigonometric functions $\sinh(\cdot), \cosh(\cdot)$ together with the fact that $\|\hh\|_{H^3(D_{\ep})}=1$ we get the bounds
\begin{align*}
\left|\frac{\cosh\left(\D_{\bar{\bx}}[\tau\hh](\bx)\right)-1}{\tau}\right|\leq C(\e)\tau |\bar{\bx}|^2,\\
\left|\frac{\sinh\left(\D_{\bar{\bx}}[\tau\hh](\bx)\right)}{\tau}-\D_{\bar{\bx}}[\hh](\bx)\right| \leq C(\e)\tau^2 |\bar{\bx}|^3.
\end{align*}
In addition, we have
\begin{multline*}
\max\left\lbrace\left|\frac{\cosh(\bar{y}+\D_{\bar{\bx}}[\ff](x))}{\cosh\left(\bar{y}+\D_{\bar{\bx}}[\ff+\xi\hh](\bx)\right)-\cos(\bar{x})}\right|,\left|\frac{\sinh(\bar{y}+\D_{\bar{\bx}}[\ff](x))}{\cosh\left(\bar{y}+\D_{\bar{\bx}}[\ff+\xi\hh](\bx)\right)-\cos(\bar{x})}\right|\right\rbrace\\
\leq \frac{C(\|\ff\|_{H^{3}(D_{\ep})})}{\left(\tfrac{1}{2}-\|\ff+\xi\hh\|_{H^{3}(D_{\ep})}\right) |\bar{\bx}|^2-C(\|\ff+\xi\hh\|_{H^{3}(D_{\ep})}) |\bar{\bx}|^4},
\end{multline*}
which give us
\[
\left|\frac{\cK[\ff+\tau \hh,\ff](\bx,\bar{\bx})}{\tau}-\Psi_1[\ff](\bx,\bar{\bx})\D_{\bar{\bx}}[\hh](\bx)\right|\leq \max_{0\leq \xi\leq\tau}\frac{C(\e,\|\ff\|_{H^{3}(D_{\ep})}) \tau}{\left(\tfrac{1}{2}-\|\ff+\xi\hh\|_{H^{3}(D_{\ep})}\right) -C(\|\ff+\xi\hh\|_{H^{3}(D_{\ep})}) |\bar{\bx}|^2}.
\]
Coming back to the starting point, as $\ff\in \mathbb{B}_{\d}(X(D_{\ep}))$ and $\|\hh\|_{H^{4,3}(D_{\ep})}=1,$  taking $0<\d(\e),\tau\ll 1$ small enough we have proved our desired inequality.
\end{proof}

\begin{proof}[Proof of \emph{ii)}]
As $\cK(\ff',\ff'')(\bx,\bar{\bx})=\left( K[\ff']-K[\ff'']\right)(\bx,\bar{\bx})$ and $\p_{\bx}K[g](\bx,\bar{\bx})=\Psi_1[g](\bx,\bar{\bx})\D_{\bar{\bx}}[\p_{\bx}g](\bx)$, adding and subtracting some appropriate term we obtain
%\begin{equation}\label{e:p1K[f',f'']}
\begin{equation*}
\p_{\bx}\cK[\ff',\ff''](\bx,\bar{\bx})=\Psi_1[\ff'](\bx,\bar{\bx})\D_{\bar{\bx}}[\p_{\bx}(\ff'-\ff'')](\bx)+\left(\Psi_1[\ff']-\Psi_1[\ff'']\right)(\bx,\bar{\bx}) \D_{\bar{\bx}}[\p_{\bx}\ff''](\bx).
\end{equation*}
%\end{equation}
So, it is just a matter of algebra to get
\begin{equation}\label{A+B}
\p_{\bx}\left\lbrace\frac{\cK[\ff+\tau \hh,\ff](\bx,\bar{\bx})}{\tau}-\Psi_1[\ff](\bx,\bar{\bx})\D_{\bar{\bx}}[\hh](\bx)\right\rbrace= \mathsf{A}[\ff,\hh](\bx,\bar{\bx})+\mathsf{B}[\ff,\hh](\bx,\bar{\bx}),
\end{equation}
where both new functionals are given respectively by
\begin{align}\label{defA}
\mathsf{A}[\ff,\hh](\bx,\bar{\bx})&:=\left[\frac{(\Psi_1[\ff+\tau\hh]-\Psi_1[\ff])(\bx,\bar{\bx})}{\tau}-(\Psi_2[\ff]-\Psi_1^2[\ff])(\bx,\bar{\bx})\D_{\bar{\bx}}[\hh](\bx)\right]\D_{\bar{\bx}}[\p_{\bx}\ff](\bx),\\
\mathsf{B}[\ff,\hh](\bx,\bar{\bx})&:=(\Psi_1[\ff+\tau\hh]-\Psi_1[\ff])(\bx,\bar{\bx})\D_{\bar{\bx}}[\p_{\bx}\hh](\bx). \nonumber
\end{align}
Directly from auxiliary Lemma \ref{l:Psi_1[f']-Psi_1[f'']}, as $\hh\in H^{4,3}(D_{\ep})$ implies $\p_{\bx}\hh\in H^{3,2}(D_{\ep})\subset C^{1}(D_{\ep})$ and by hypothesis $\|\hh\|_{H^{4,3}(D_{\ep})}=1,$ we get
\[
\sup_{\bx\in D_{\ep}} \int_{D_{\ep}(y)}\left|\mathsf{B}[\ff,\hh]\right|^2\mbox{d}\bar{\bx}\leq   C(\e)\tau^2.
\]
To obtain something similar for $\mathsf{A}[\ff,\hh]$ we need to work a little bit more. In first place, thanks to the definition of $\Psi_1[\cdot](\bx,\bar{\bx})$, adding and subtracting some appropriate term we have
\begin{multline*}
\left(\Psi_1[\ff+\tau\hh]-\Psi_1[\ff]\right)(\bx,\bar{\bx})=\frac{\sinh(\bar{y}+\D_{\bar{\bx}}[\ff+\tau\hh](\bx))-\sinh(\bar{y}+\D_{\bar{\bx}}[\ff](\bx))}{\cosh(\bar{y}+\D_{\bar{\bx}}[\ff](\bx))-\cos(\bar{x})}\\
-\sinh(\bar{y}+\D_{\bar{\bx}}[\ff+\tau\hh](\bx))\frac{\cosh(\bar{y}+\D_{\bar{\bx}}[\ff+\tau\hh](\bx))-\cosh(\bar{y}+\D_{\bar{\bx}}[\ff](\bx))}{\left[\cosh(\bar{y}+\D_{\bar{\bx}}[\ff+\tau\hh](\bx))-\cos(\bar{x})\right]\left[\cosh(\bar{y}+\D_{\bar{\bx}}[\ff](\bx))-\cos(\bar{x})\right]},
\end{multline*}
and using on it some trigonometric identities, we obtain
\begin{multline*}
\left(\Psi_1[\ff+\tau\hh]-\Psi_1[\ff]\right)(\bx,\bar{\bx})=\left[\Psi_2[\ff](\bx,\bar{\bx})-\Psi_1^2[\ff](\bx,\bar{\bx})\cosh\left(\tau\D_{\bar{\bx}}[\hh](\bx)\right)\right]\sinh\left(\tau\D_{\bar{\bx}}[\hh](\bx)\right)\\
+\Psi_1^2[\ff](\bx,\bar{\bx}) \frac{\cosh(\bar{y}+\D_{\bar{\bx}}[\ff+\tau\hh](\bx)-\cosh(\bar{y}+\D_{\bar{\bx}}[\ff](\bx)}{\cosh(\bar{y}+\D_{\bar{\bx}}[\ff+\tau\hh](\bx)-\cos(\bar{x})}\sinh\left(\tau\D_{\bar{\bx}}[\hh](\bx)\right)\cosh\left(\tau\D_{\bar{\bx}}[\hh](\bx)\right).
\end{multline*}
Now, dividing the above expression by $\tau$ and subtracting the term $(\Psi_2[\ff]-\Psi_1^2[\ff])(\bx,\bar{\bx})\D_{\bar{\bx}}[\hh](\bx)$ we obtain that
\begin{multline*}
\frac{(\Psi_1[\ff+\tau\hh]-\Psi_1[\ff])(\bx,\bar{\bx})}{\tau}-(\Psi_2[\ff]-\Psi_1^2[\ff])(\bx,\bar{\bx})\D_{\bar{\bx}}[\hh](\bx)\\
=\Psi_2[\ff](\bx,\bar{\bx})\left(\frac{\sinh(\tau\D_{\bar{\bx}}[\hh](\bx))}{\tau} - \D_{\bar{\bx}}[\hh](\bx)\right)\\
-\Psi_1^2[\ff](\bx,\bar{\bx})\left(\cosh(\tau\D_{\bar{\bx}}[\hh](\bx))\frac{\sinh(\tau\D_{\bar{\bx}}[\hh](\bx))}{\tau} - \D_{\bar{\bx}}[\hh](\bx)\right)\\
+\Psi_1^2[\ff](\bx,\bar{\bx}) \frac{\cosh(\bar{y}+\D_{\bar{\bx}}[\ff+\tau\hh](\bx)-\cosh(\bar{y}+\D_{\bar{\bx}}[\ff](\bx)}{\cosh(\bar{y}+\D_{\bar{\bx}}[\ff+\tau\hh](\bx)-\cos(\bar{x})}\sinh\left(\tau\D_{\bar{\bx}}[\hh](\bx)\right)\cosh\left(\tau\D_{\bar{\bx}}[\hh](\bx)\right).
\end{multline*}
Finally, using the Taylor expansion of the trigonometric functions $\sinh(\cdot), \cosh(\cdot)$ and  \eqref{bound_f'-f''} give us
the desired inequality
\begin{equation}\label{integralA}
\sup_{\bx\in D_{\ep}} \int_{D_{\ep}(y)}\left|\mathsf{A}[\ff,\hh]\right|^2\mbox{d}\bar{\bx}\leq   C(\e)\tau^2.
\end{equation}
\end{proof}

\begin{proof}[Proof of iii)]
Recalling \eqref{A+B}, we have
\begin{equation}\label{dxA+dxB}
\p_{\bx}^2\left\lbrace\frac{\cK[\ff+\tau \hh,\ff](\bx,\bar{\bx})}{\tau}-\Psi_1[\ff](\bx,\bar{\bx})\D_{\bar{\bx}}[\hh](\bx)\right\rbrace= \p_{\bx}\mathsf{A}[\ff,\hh](\bx,\bar{\bx})+\p_{\bx}\mathsf{B}[\ff,\hh](\bx,\bar{\bx}),
\end{equation}
where, using  relation $\p_{\bx}\Psi_1[g](\bx,\bar{\bx})=\tilde{\II}_1[g](\bx,\bar{\bx})\D_{\bar{\bx}}[\p_{\bx}g](\bx)$ whit $\tilde{\II}_1[g]=\left(\Psi_2[g]-(\Psi_1[g])^2\right)$, we get
\begin{multline*}
\p_{\bx}\mathsf{A}[\ff,\hh](\bx,\bar{\bx})= (\tilde{\II}_1[\ff+\tau\hh]-\tilde{\II}_1[\ff])(\bx,\bar{\bx})\D_{\bar{\bx}}[\p_{\bx}\ff](\bx)\D_{\bar{\bx}}[\p_{\bx}\hh](\bx)\\
 +\frac{(\Psi_1[\ff+\tau\hh]-\Psi_1[\ff])(\bx,\bar{\bx})}{\tau}\D_{\bar{\bx}}[\p_{\bx}^2\ff](\bx)-\tilde{\II}_1[\ff](\bx,\bar{\bx})\D_{\bar{\bx}}[\p_{\bx}^2\ff]\D_{\bar{\bx}}[\hh](\bx)\\
+\frac{(\tilde{\II}_1[\ff+\tau\hh]-\tilde{\II}_1[\ff])(\bx,\bar{\bx})}{\tau}\left(\D_{\bar{\bx}}[\p_{\bx}\ff](\bx)\right)^2 -\left(\p_{\bx}\Psi_2[\ff]-2\Psi_1[\ff]\p_{\bx}\Psi_1[\ff]\right)(\bx,\bar{\bx})\D_{\bar{\bx}}[\p_{\bx}\ff](\bx)\D_{\bar{\bx}}[\hh](\bx),
\end{multline*}
and
\begin{multline*}
\p_{\bx}\mathsf{B}[\ff,\hh](\bx,\bar{\bx})=(\tilde{\II}_1[\ff+\tau\hh]-\tilde{\II}_1[\ff])(\bx,\bar{\bx}) \D_{\bar{\bx}}[\p_{\bx}\ff](\bx)\D_{\bar{\bx}}[\p_{\bx}\hh](\bx)\\
+(\Psi_1[\ff+\tau\hh]-\Psi_1[\ff])(\bx,\bar{\bx})\D_{\bar{\bx}}[\p_{\bx}^2 \hh](\bx)\\
+ \tau \tilde{\II}_1[\ff+\tau\hh](\bx,\bar{\bx})\left(\D_{\bar{\bx}}[\p_{\bx} \hh](\bx)\right)^2.
\end{multline*}
By definition of $\tilde{\II}_1[\cdot]$ and relations \eqref{DPsi1} and \eqref{DPsi2}, the previous expressions can be rewritten in a more manageable way as
\begin{align}
\p_{\bx}\mathsf{A}[\ff,\hh](\bx,\bar{\bx})&=(\mathsf{A}_1'+\mathsf{A}_2'+\mathsf{A}_3'+\mathsf{A}_4'+\mathsf{A}_5')[\ff,\hh](\bx,\bar{\bx}), \label{dxA}\\
\p_{\bx}\mathsf{B}[\ff,\hh](\bx,\bar{\bx})&=(\mathsf{B}_1'+\mathsf{B}_2'+\mathsf{B}_3'+\mathsf{B}_4')[\ff,\hh](\bx,\bar{\bx}),\label{dxB}
\end{align}
with
\begin{align*}
\mathsf{A}_1'[\ff,\hh](\bx,\bar{\bx})&:=\left[\frac{(\Psi_1[\ff+\tau\hh]-\Psi_1[\ff])(\bx,\bar{\bx})}{\tau}-(\Psi_2[\ff]-\Psi_1^2[\ff])(\bx,\bar{\bx})\D_{\bar{\bx}}[\hh](\bx)\right]\D_{\bar{\bx}}[\p_{\bx}^2\ff](\bx),\\
\mathsf{A}_2'[\ff,\hh](\bx,\bar{\bx})&:=(\Psi_2[\ff+\tau\hh]-\Psi_2[\ff])(\bx,\bar{\bx})\D_{\bar{\bx}}[\p_{\bx}\ff](\bx)\D_{\bar{\bx}}[\p_{\bx}\hh](\bx),\\
\mathsf{A}_3'[\ff,\hh](\bx,\bar{\bx})&:=-(\Psi_1[\ff+\tau\hh]-\Psi_1[\ff])(\Psi_1[\ff+\tau\hh]+\Psi_1[\ff])(\bx,\bar{\bx})\D_{\bar{\bx}}[\p_{\bx}\ff](\bx)\D_{\bar{\bx}}[\p_{\bx}\hh](\bx),\\
\mathsf{A}_4'[\ff,\hh](\bx,\bar{\bx})&:=\left[\frac{(\Psi_2[\ff+\tau\hh]-\Psi_2[\ff])(\bx,\bar{\bx})}{\tau} -\left(\Psi_1[\ff]-\Psi_1[\ff]\Psi_2[\ff]\right)(\bx,\bar{\bx})\D_{\bar{\bx}}[\hh](\bx)\right]\left(\D_{\bar{\bx}}[\p_{\bx}\ff](\bx)\right)^2,
\end{align*}
\begin{multline*}
\mathsf{A}_5'[\ff,\hh](\bx,\bar{\bx}):=\bigg[-(\Psi_1[\ff+\tau\hh]+\Psi_1[\ff])(\bx,\bar{\bx})\frac{(\Psi_1[\ff+\tau\hh]-\Psi_1[\ff])(\bx,\bar{\bx})}{\tau}\\
+2\Psi_1[\ff](\bx,\bar{\bx}) (\Psi_2[\ff]-\Psi_1^2[\ff])(\bx,\bar{\bx})\D_{\bar{\bx}}[\hh](\bx)\bigg]\left(\D_{\bar{\bx}}[\p_{\bx}\ff](\bx)\right)^2,
\end{multline*}
and
\begin{align*}
\mathsf{B}_1'[\ff,\hh](\bx,\bar{\bx})&:=(\Psi_2[\ff+\tau\hh]-\Psi_2[\ff])(\bx,\bar{\bx})\D_{\bar{\bx}}[\p_{\bx}\ff](\bx)\D_{\bar{\bx}}[\p_{\bx}\hh](\bx),\\
\mathsf{B}_2'[\ff,\hh](\bx,\bar{\bx})&:=-(\Psi_1[\ff+\tau\hh]-\Psi_1[\ff])(\Psi_1[\ff+\tau\hh]+\Psi_1[\ff])(\bx,\bar{\bx})\D_{\bar{\bx}}[\p_{\bx}\ff](\bx)\D_{\bar{\bx}}[\p_{\bx}\hh](\bx),\\
\mathsf{B}_3'[\ff,\hh](\bx,\bar{\bx})&:= \tau (\Psi_2[\ff+\tau\hh]-\Psi_1^2[\ff+\tau\hh])(\bx,\bar{\bx})\left(\D_{\bar{\bx}}[\p_{\bx} \hh](\bx)\right)^2,\\
\mathsf{B}_4'[\ff,\hh](\bx,\bar{\bx})&:=(\Psi_1[\ff+\tau\hh]-\Psi_1[\ff])(\bx,\bar{\bx})\D_{\bar{\bx}}[\p_{\bx}^2 \hh](\bx).
\end{align*}
Note that the only new type of term that appear is $\mathsf{A}_4'[\ff,\hh].$ The rest of terms can be easily handle using auxiliary Lemma \ref{l:Psi_1[f']-Psi_1[f'']}, Corollary \ref{boundsPSI} and \eqref{bound_f'-f''}.
To sum up, our proof reduces to check that
\begin{equation}\label{A4'}
\int_{D_{\ep}}\left(\int_{D_\ep(y)}\left|\mathsf{A}_4'[\ff,\hh](\bx,\bar{\bx})\right|^2 |\bar{\bx}|^{2\g}\mbox{d}\bar{\bx} \right)\mbox{d}\bx\leq C(\ep,\|\ff\|_{H^{3}(D_{\ep})})\tau^2.
\end{equation}
By \eqref{bound_f'-f''}, as $\ff\in H^{4,3}(D_{\ep})$ we have that $\p_{\bx}\ff\in H^{3,2}(D_{\ep})\subset C^{1}(D_{\ep})$ and we just need  to prove the same type of estimate for the term
$$ \int_{D_{\ep}}\left(\int_{D_\ep(y)}\left|\frac{(\Psi_2[\ff+\tau\hh]-\Psi_2[\ff])(\bx,\bar{\bx})}{\tau} -\left(\Psi_1[\ff]-\Psi_1[\ff]\Psi_2[\ff]\right)(\bx,\bar{\bx})\D_{\bar{\bx}}[\hh](\bx)\right|^2 |\bar{\bx}|^{4+2\g}\mbox{d}\bar{\bx} \right)\mbox{d}\bx. $$
Thanks to the definition of $\Psi_2[\cdot](\bx,\bar{\bx})$, adding and subtracting some appropriate term we have
\begin{multline*}
\left(\Psi_2[\ff+\tau\hh]-\Psi_2[\ff]\right)(\bx,\bar{\bx})=\frac{\cosh(\bar{y}+\D_{\bar{\bx}}[\ff+\tau\hh](\bx))-\cosh(\bar{y}+\D_{\bar{\bx}}[\ff](\bx))}{\cosh(\bar{y}+\D_{\bar{\bx}}[\ff](\bx))-\cos(\bar{x})}\\
-\cosh(\bar{y}+\D_{\bar{\bx}}[\ff+\tau\hh](\bx))\frac{\cosh(\bar{y}+\D_{\bar{\bx}}[\ff+\tau\hh](\bx))-\cosh(\bar{y}+\D_{\bar{\bx}}[\ff](\bx))}{\left[\cosh(\bar{y}+\D_{\bar{\bx}}[\ff+\tau\hh](\bx))-\cos(\bar{x})\right]\left[\cosh(\bar{y}+\D_{\bar{\bx}}[\ff](\bx))-\cos(\bar{x})\right]},
\end{multline*}
and using on it some trigonometric identities, we obtain
\begin{multline*}
\left(\Psi_2[\ff+\tau\hh]-\Psi_2[\ff]\right)(\bx,\bar{\bx})=\left[\Psi_1[\ff](\bx,\bar{\bx})-\Psi_1[\ff]\Psi_2[\ff](\bx,\bar{\bx})\cosh\left(\tau\D_{\bar{\bx}}[\hh](\bx)\right)\right]\sinh\left(\tau\D_{\bar{\bx}}[\hh](\bx)\right)\\
-\Psi_1^2[\ff](\bx,\bar{\bx})\sinh^2(\tau\D_{\bar{\bx}}[\hh](\bx)) +2\Psi_1[\ff]\Psi_2[\ff](\bx,\bar{\bx})\sinh(\tau\D_{\bar{\bx}}[\hh](\bx)) \sinh^2\left(\tfrac{\tau}{2}\D_{\bar{\bx}}[\hh](\bx)\right)\\
+2\Psi_2[\ff](\bx,\bar{\bx})\left[1-\Psi_2[\ff](\bx,\bar{\bx})\cosh(\tau\D_{\bar{\bx}}[\hh](\bx))\right] \sinh^2\left(\tfrac{\tau}{2}\D_{\bar{\bx}}[\hh](\bx)\right)\\
+\Psi_1[\ff+\tau\hh](\bx,\bar{\bx})\left[\frac{\cosh(\bar{y}+\D_{\bar{\bx}}[\ff+\tau\hh](\bx))-\cosh(\bar{y}+\D_{\bar{\bx}}[\ff](\bx))}{\cosh(\bar{y}+\D_{\bar{\bx}}[\ff](\bx))-\cos(\bar{x})}\right]^2.
\end{multline*}
Now, dividing the above expression by $\tau$ and subtracting the term $(\Psi_1[\ff]-\Psi_1[\ff]\Psi_2[\ff])(\bx,\bar{\bx})\D_{\bar{\bx}}[\hh](\bx)$ we obtain that
\begin{multline*}
\frac{(\Psi_2[\ff+\tau\hh]-\Psi_2[\ff])(\bx,\bar{\bx})}{\tau} -\left(\Psi_1[\ff]-\Psi_1[\ff]\Psi_2[\ff]\right)(\bx,\bar{\bx})\D_{\bar{\bx}}[\hh](\bx)\\
=\Psi_1[\ff](\bx,\bar{\bx})\left(\frac{\sinh\left(\tau\D_{\bar{\bx}}[\hh](\bx)\right)}{\tau}-\D_{\bar{\bx}}[\hh](\bx) \right)\\
-\Psi_1[\ff]\Psi_2[\ff](\bx,\bar{\bx})\left(\cosh\left(\tau\D_{\bar{\bx}}[\hh](\bx)\right)\frac{\sinh\left(\tau\D_{\bar{\bx}}[\hh](\bx)\right)}{\tau}-\D_{\bar{\bx}}[\hh](\bx) \right)\\
-\Psi_1^2[\ff](\bx,\bar{\bx})\sinh^2(\tau\D_{\bar{\bx}}[\hh](\bx)) +2\Psi_1[\ff]\Psi_2[\ff](\bx,\bar{\bx})\sinh(\tau\D_{\bar{\bx}}[\hh](\bx)) \sinh^2\left(\tfrac{\tau}{2}\D_{\bar{\bx}}[\hh](\bx)\right)\\
+2\Psi_2[\ff](\bx,\bar{\bx})\left[1-\Psi_2[\ff](\bx,\bar{\bx})\cosh(\tau\D_{\bar{\bx}}[\hh](\bx))\right] \sinh^2\left(\tfrac{\tau}{2}\D_{\bar{\bx}}[\hh](\bx)\right)\\
+\Psi_1[\ff+\tau\hh](\bx,\bar{\bx})\left[\frac{\cosh(\bar{y}+\D_{\bar{\bx}}[\ff+\tau\hh](\bx))-\cosh(\bar{y}+\D_{\bar{\bx}}[\ff](\bx))}{\cosh(\bar{y}+\D_{\bar{\bx}}[\ff](\bx))-\cos(\bar{x})}\right]^2.
\end{multline*}
Finally, using the Taylor expansion of the trigonometric functions $\sinh(\cdot), \cosh(\cdot)$ and  \eqref{bound_f'-f''} give us
the desired inequality \eqref{A4'}.
\end{proof}

\begin{proof}[Proof of iv)]
Recalling \eqref{dxA+dxB} together with \eqref{dxA} and \eqref{dxB}, we have
\begin{equation}\label{dxA'+dxB'}
\p_{\bx}^3\left\lbrace\frac{\cK[\ff+\tau \hh,\ff](\bx,\bar{\bx})}{\tau}-\Psi_1[\ff](\bx,\bar{\bx})\D_{\bar{\bx}}[\hh](\bx)\right\rbrace= \sum_{i=1}^{5}\p_{\bx}\mathsf{A}_i'[\ff,\hh](\bx,\bar{\bx})+\sum_{i=1}^{4}\p_{\bx}\mathsf{B}_i'[\ff,\hh](\bx,\bar{\bx}),
\end{equation}
where,
\begin{multline*}
\p_{\bx}\mathsf{A}_1'[\ff,\hh](\bx,\bar{\bx})=\p_{\bx}\left\lbrace\frac{(\Psi_1[\ff+\tau\hh]-\Psi_1[\ff])(\bx,\bar{\bx})}{\tau}-(\Psi_2[\ff]-\Psi_1^2[\ff])(\bx,\bar{\bx})\D_{\bar{\bx}}[\hh](\bx)\right\rbrace\D_{\bar{\bx}}[\p_{\bx}^2\ff](\bx)\\
+\left[\frac{(\Psi_1[\ff+\tau\hh]-\Psi_1[\ff])(\bx,\bar{\bx})}{\tau}-(\Psi_2[\ff]-\Psi_1^2[\ff])(\bx,\bar{\bx})\D_{\bar{\bx}}[\hh](\bx)\right]\D_{\bar{\bx}}[\p_{\bx}^3\ff](\bx)\\
:=\mathsf{A}_{11}''[\ff,\hh](\bx,\bar{\bx})+\mathsf{A}_{12}''[\ff,\hh](\bx,\bar{\bx}),
\end{multline*}
\begin{multline*}
\p_{\bx}\mathsf{A}_2'[\ff,\hh](\bx,\bar{\bx})=(\p_{\bx}\Psi_2[\ff+\tau\hh]-\p_{\bx}\Psi_2[\ff])(\bx,\bar{\bx})\D_{\bar{\bx}}[\p_{\bx}\ff](\bx)\D_{\bar{\bx}}[\p_{\bx}\hh](\bx)\\
+(\Psi_2[\ff+\tau\hh]-\Psi_2[\ff])(\bx,\bar{\bx})\D_{\bar{\bx}}[\p_{\bx}^2\ff](\bx)\D_{\bar{\bx}}[\p_{\bx}\hh](\bx)\\
+(\Psi_2[\ff+\tau\hh]-\Psi_2[\ff])(\bx,\bar{\bx})\D_{\bar{\bx}}[\p_{\bx}\ff](\bx)\D_{\bar{\bx}}[\p_{\bx}^2\hh](\bx)\\
:=\mathsf{A}_{21}''[\ff,\hh](\bx,\bar{\bx})+\mathsf{A}_{22}''[\ff,\hh](\bx,\bar{\bx})+\mathsf{A}_{23}''[\ff,\hh](\bx,\bar{\bx}),
\end{multline*}
\begin{multline*}
\p_{\bx}\mathsf{A}_3'[\ff,\hh](\bx,\bar{\bx})=-(\p_{\bx}\Psi_1[\ff+\tau\hh]-\p_{\bx}\Psi_1[\ff])(\Psi_1[\ff+\tau\hh]+\Psi_1[\ff])(\bx,\bar{\bx})\D_{\bar{\bx}}[\p_{\bx}\ff](\bx)\D_{\bar{\bx}}[\p_{\bx}\hh](\bx)\\
 -(\Psi_1[\ff+\tau\hh]-\Psi_1[\ff])(\p_{\bx}\Psi_1[\ff+\tau\hh]+\p_{\bx}\Psi_1[\ff])(\bx,\bar{\bx})\D_{\bar{\bx}}[\p_{\bx}\ff](\bx)\D_{\bar{\bx}}[\p_{\bx}\hh](\bx)\\
-(\Psi_1[\ff+\tau\hh]-\Psi_1[\ff])(\Psi_1[\ff+\tau\hh]+\Psi_1[\ff])(\bx,\bar{\bx})\D_{\bar{\bx}}[\p_{\bx}^2\ff](\bx)\D_{\bar{\bx}}[\p_{\bx}\hh](\bx)\\
 -(\Psi_1[\ff+\tau\hh]-\Psi_1[\ff])(\Psi_1[\ff+\tau\hh]+\Psi_1[\ff])(\bx,\bar{\bx})\D_{\bar{\bx}}[\p_{\bx}\ff](\bx)\D_{\bar{\bx}}[\p_{\bx}^2\hh](\bx)\\
:=\mathsf{A}_{31}''[\ff,\hh](\bx,\bar{\bx})+\mathsf{A}_{32}''[\ff,\hh](\bx,\bar{\bx})+\mathsf{A}_{33}''[\ff,\hh](\bx,\bar{\bx})+\mathsf{A}_{34}''[\ff,\hh](\bx,\bar{\bx}),
\end{multline*}
\begin{multline*}
\p_{\bx}\mathsf{A}_4'[\ff,\hh](\bx,\bar{\bx})\\=\p_{\bx}\left\lbrace\frac{(\Psi_2[\ff+\tau\hh]-\Psi_2[\ff])(\bx,\bar{\bx})}{\tau}-(\Psi_1[\ff]-\Psi_1[\ff]\Psi_2[\ff])(\bx,\bar{\bx})\D_{\bar{\bx}}[\hh](\bx)\right\rbrace \left(\D_{\bar{\bx}}[\p_{\bx}\ff](\bx)\right)^2\\
\quad +\left[\frac{(\Psi_2[\ff+\tau\hh]-\Psi_2[\ff])(\bx,\bar{\bx})}{\tau}-(\Psi_1[\ff]-\Psi_1[\ff]\Psi_2[\ff])(\bx,\bar{\bx})\D_{\bar{\bx}}[\hh](\bx)\right]2 \D_{\bar{\bx}}[\p_{\bx}\ff](\bx) \D_{\bar{\bx}}[\p_{\bx}^2\ff](\bx)\\
:=\mathsf{A}_{41}''[\ff,\hh](\bx,\bar{\bx})+\mathsf{A}_{42}''[\ff,\hh](\bx,\bar{\bx}),
\end{multline*}
\begin{multline*}
\p_{\bx}\mathsf{A}_5'[\ff,\hh](\bx,\bar{\bx})=\p_{\bx}\bigg \lbrace -(\Psi_1[\ff+\tau\hh]+\Psi_1[\ff])(\bx,\bar{\bx})\frac{(\Psi_1[\ff+\tau\hh]-\Psi_1[\ff])(\bx,\bar{\bx})}{\tau}\\
 +2\Psi_1[\ff] (\Psi_2[\ff]-\Psi_1^2[\ff])(\bx,\bar{\bx})\D_{\bar{\bx}}[\hh](\bx) \bigg \rbrace\left(\D_{\bar{\bx}}[\p_{\bx}\ff](\bx)\right)^2\\
+\bigg[-(\Psi_1[\ff+\tau\hh]+\Psi_1[\ff])(\bx,\bar{\bx})\frac{(\Psi_1[\ff+\tau\hh]-\Psi_1[\ff])(\bx,\bar{\bx})}{\tau}\\
+2\Psi_1[\ff] (\Psi_2[\ff]-\Psi_1^2[\ff])(\bx,\bar{\bx})\D_{\bar{\bx}}[\hh](\bx)\bigg]2 \D_{\bar{\bx}}[\p_{\bx}\ff](\bx) \D_{\bar{\bx}}[\p_{\bx}^2\ff](\bx)\\
:=\mathsf{A}_{51}''[\ff,\hh](\bx,\bar{\bx})+\mathsf{A}_{52}''[\ff,\hh](\bx,\bar{\bx}),
\end{multline*}

and
\begin{multline*}
\p_{\bx}\mathsf{B}_1'[\ff,\hh](\bx,\bar{\bx})=(\p_{\bx}\Psi_2[\ff+\tau\hh]-\p_{\bx}\Psi_2[\ff])(\bx,\bar{\bx})\D_{\bar{\bx}}[\p_{\bx}\ff](\bx)\D_{\bar{\bx}}[\p_{\bx}\hh](\bx)\\
 +(\Psi_2[\ff+\tau\hh]-\Psi_2[\ff])(\bx,\bar{\bx})\left[\D_{\bar{\bx}}[\p_{\bx}^2\ff](\bx)\D_{\bar{\bx}}[\p_{\bx}\hh](\bx)+\D_{\bar{\bx}}[\p_{\bx}\ff](\bx)\D_{\bar{\bx}}[\p_{\bx}^2\hh](\bx)\right]\\
=:\mathsf{B}_{11}''[\ff,\hh](\bx,\bar{\bx})+\mathsf{B}_{12}''[\ff,\hh](\bx,\bar{\bx}),
\end{multline*}
\begin{multline*}
\p_{\bx}\mathsf{B}_2'[\ff,\hh](\bx,\bar{\bx})=-(\p_{\bx}\Psi_1[\ff+\tau\hh]-\p_{\bx}\Psi_1[\ff])(\Psi_1[\ff+\tau\hh]+\Psi_1[\ff])(\bx,\bar{\bx})\D_{\bar{\bx}}[\p_{\bx}\ff](\bx)\D_{\bar{\bx}}[\p_{\bx}\hh](\bx)\\
-(\Psi_1[\ff+\tau\hh]-\Psi_1[\ff])(\p_{\bx}\Psi_1[\ff+\tau\hh]+\p_{\bx}\Psi_1[\ff])(\bx,\bar{\bx})\D_{\bar{\bx}}[\p_{\bx}\ff](\bx)\D_{\bar{\bx}}[\p_{\bx}\hh](\bx)\\
-(\Psi_1[\ff+\tau\hh]-\Psi_1[\ff])(\Psi_1[\ff+\tau\hh]+\Psi_1[\ff])(\bx,\bar{\bx})\D_{\bar{\bx}}[\p_{\bx}^2\ff](\bx)\D_{\bar{\bx}}[\p_{\bx}\hh](\bx)\\
-(\Psi_1[\ff+\tau\hh]-\Psi_1[\ff])(\Psi_1[\ff+\tau\hh]+\Psi_1[\ff])(\bx,\bar{\bx})\D_{\bar{\bx}}[\p_{\bx}\ff](\bx)\D_{\bar{\bx}}[\p_{\bx}^2\hh](\bx)\\
=:\mathsf{B}_{21}''[\ff,\hh](\bx,\bar{\bx})+\mathsf{B}_{22}''[\ff,\hh](\bx,\bar{\bx})+\mathsf{B}_{23}''[\ff,\hh](\bx,\bar{\bx})+\mathsf{B}_{24}''[\ff,\hh](\bx,\bar{\bx}),
\end{multline*}
\begin{multline*}
\p_{\bx}\mathsf{B}_3'[\ff,\hh](\bx,\bar{\bx})= \tau (\p_{\bx}\Psi_2[\ff+\tau\hh]-2\Psi_1[\ff+\tau\hh]\p_{\bx}\Psi_1[\ff+\tau\hh])(\bx,\bar{\bx})\left(\D_{\bar{\bx}}[\p_{\bx} \hh](\bx)\right)^2\\
 +\tau (\Psi_2[\ff+\tau\hh]-\Psi_1^2[\ff+\tau\hh])(\bx,\bar{\bx})2 \D_{\bar{\bx}}[\p_{\bx} \hh](\bx)\D_{\bar{\bx}}[\p_{\bx}^2 \hh](\bx)\\
=:\mathsf{B}_{31}''[\ff,\hh](\bx,\bar{\bx})+\mathsf{B}_{32}''[\ff,\hh](\bx,\bar{\bx}),
\end{multline*}
\begin{multline*}
\p_{\bx}\mathsf{B}_4'[\ff,\hh](\bx,\bar{\bx})= (\p_{\bx}\Psi_1[\ff+\tau\hh]-\p_{\bx}\Psi_1[\ff])(\bx,\bar{\bx})\D_{\bar{\bx}}[\p_{\bx}^2 \hh](\bx)\\
+(\Psi_1[\ff+\tau\hh]-\Psi_1[\ff])(\bx,\bar{\bx})\D_{\bar{\bx}}[\p_{\bx}^3 \hh](\bx)\\
=:\mathsf{B}_{41}''[\ff,\hh](\bx,\bar{\bx})+\mathsf{B}_{42}''[\ff,\hh](\bx,\bar{\bx}).
\end{multline*}
Taking into account all the above we have proved that
\begin{multline*}
\int_{D_{\ep}}\left(\int_{D_\ep(y)}\left|\p_{\bx}^3\left\lbrace \frac{\cK[\ff+\tau \hh,\ff](\bx,\bar{\bx})}{\tau}-\Psi_1[\ff](\bx,\bar{\bx})\D_{\bar{\bx}}[\hh](\bx)\right\rbrace\right|^2 |\bar{\bx}|^{2}\mbox{d}\bar{\bx} \right)\mbox{d}\bx \\
\lesssim \sum_{i=1}^{5} \int_{D_{\ep}}\left(\int_{D_\ep(y)}\left|\p_{\bx}\mathsf{A}_i'[\ff,\hh](\bx,\bar{\bx})\right|^2 |\bar{\bx}|^{2}\mbox{d}\bar{\bx} \right)\mbox{d}\bx + \sum_{i=1}^{4} \int_{D_{\ep}}\left(\int_{D_\ep(y)}\left|\p_{\bx}\mathsf{B}_i'[\ff,\hh](\bx,\bar{\bx})\right|^2 |\bar{\bx}|^{2}\mbox{d}\bar{\bx} \right)\mbox{d}\bx,
\end{multline*}
where each of the terms that appear can be decomposed as follows:
\begin{align*}
\p_{\bx}\mathsf{A}_1'[\ff,\hh](\bx,\bar{\bx})&=\mathsf{A}_{11}''[\ff,\hh](\bx,\bar{\bx})+\mathsf{A}_{12}''[\ff,\hh](\bx,\bar{\bx}),\\
\p_{\bx}\mathsf{A}_2'[\ff,\hh](\bx,\bar{\bx})&=\mathsf{A}_{21}''[\ff,\hh](\bx,\bar{\bx})+\mathsf{A}_{22}''[\ff,\hh](\bx,\bar{\bx})+\mathsf{A}_{23}''[\ff,\hh](\bx,\bar{\bx}),\\
\p_{\bx}\mathsf{A}_3'[\ff,\hh](\bx,\bar{\bx})&=\mathsf{A}_{31}''[\ff,\hh](\bx,\bar{\bx})+\mathsf{A}_{32}''[\ff,\hh](\bx,\bar{\bx})+\mathsf{A}_{33}''[\ff,\hh](\bx,\bar{\bx})+\mathsf{A}_{34}''[\ff,\hh](\bx,\bar{\bx}),\\
\p_{\bx}\mathsf{A}_4'[\ff,\hh](\bx,\bar{\bx})&=\mathsf{A}_{41}''[\ff,\hh](\bx,\bar{\bx})+\mathsf{A}_{42}''[\ff,\hh](\bx,\bar{\bx}),\\
\p_{\bx}\mathsf{A}_5'[\ff,\hh](\bx,\bar{\bx})&=\mathsf{A}_{51}''[\ff,\hh](\bx,\bar{\bx})+\mathsf{A}_{52}''[\ff,\hh](\bx,\bar{\bx}),
\end{align*}
and
\begin{align*}
\p_{\bx}\mathsf{B}_1'[\ff,\hh](\bx,\bar{\bx})&=\mathsf{B}_{11}''[\ff,\hh](\bx,\bar{\bx})+\mathsf{B}_{12}''[\ff,\hh](\bx,\bar{\bx}),\\
\p_{\bx}\mathsf{B}_2'[\ff,\hh](\bx,\bar{\bx})&=\mathsf{B}_{21}''[\ff,\hh](\bx,\bar{\bx})+\mathsf{B}_{22}''[\ff,\hh](\bx,\bar{\bx})+\mathsf{B}_{23}''[\ff,\hh](\bx,\bar{\bx})+\mathsf{B}_{24}''[\ff,\hh](\bx,\bar{\bx}),\\
\p_{\bx}\mathsf{B}_3'[\ff,\hh](\bx,\bar{\bx})&=\mathsf{B}_{31}''[\ff,\hh](\bx,\bar{\bx})+\mathsf{B}_{32}''[\ff,\hh](\bx,\bar{\bx}),\\
\p_{\bx}\mathsf{B}_4'[\ff,\hh](\bx,\bar{\bx})&=\mathsf{B}_{41}''[\ff,\hh](\bx,\bar{\bx})+\mathsf{B}_{42}''[\ff,\hh](\bx,\bar{\bx}).
\end{align*}
For the sake of brevity,  we shall present here the complete details for the term $\mathsf{A}_{51}''[\ff,\hh](\bx,\bar{\bx})$ and the other terms can be dealt using the previous lemmas via straightforward variations. To sum up, our goal reduces to check that
\begin{equation}\label{A51''}
\int_{D_{\ep}}\left(\int_{D_\ep(y)}\left|\mathsf{A}_{51}''[\ff,\hh](\bx,\bar{\bx})\right|^2 |\bar{\bx}|^{2}\mbox{d}\bar{\bx} \right)\mbox{d}\bx\leq C(\ep)\tau^2.
\end{equation}
Notice that $\mathsf{A}_{51}''[\ff,\hh]$ can be written  in a more manageable way  (adding and subtracting terms) as
\begin{multline*}
\mathsf{A}_{51}''[\ff,\hh](\bx,\bar{\bx})
=2(\Psi_1[\ff]-\Psi_1[\ff+\tau\hh])(\bx,\bar{\bx}) \p_{\bx}\left\lbrace \frac{(\Psi_1[\ff+\tau\hh]-\Psi_1[\ff])(\bx,\bar{\bx})}{\tau} \right\rbrace \left(\D_{\bar{\bx}}[\p_{\bx}\ff](\bx)\right)^2\\
-2 \p_{\bx}\Psi_1[\ff](\bx,\bar{\bx}) \left\lbrace \frac{(\Psi_1[\ff+\tau\hh]-\Psi_1[\ff])(\bx,\bar{\bx})}{\tau} - (\Psi_2[\ff]-\Psi_1^2[\ff])(\bx,\bar{\bx})\D_{\bar{\bx}}[\hh](\bx)\right\rbrace  \left(\D_{\bar{\bx}}[\p_{\bx}\ff](\bx)\right)^2\\
-2 \Psi_1[\ff](\bx,\bar{\bx})\p_{\bx}\left\lbrace \frac{(\Psi_1[\ff+\tau\hh]-\Psi_1[\ff])(\bx,\bar{\bx})}{\tau}-(\Psi_2[\ff]-\Psi_1^2[\ff])(\bx,\bar{\bx})\D_{\bar{\bx}}[\hh](\bx)\right\rbrace \left(\D_{\bar{\bx}}[\p_{\bx}\ff](\bx)\right)^2.
\end{multline*}
Now, each one of the above terms can be easily handled. Using \eqref{bound_f'-f''} and applying repeatedly Lemma \ref{l:Psi_1[f']-Psi_1[f'']} we get that the first term is bounded as required. For the second term, we only need to note that it can be written, remembering \eqref{defA}, as
\begin{multline*}
2 \p_{\bx}\Psi_1[\ff](\bx,\bar{\bx}) \left\lbrace \frac{(\Psi_1[\ff+\tau\hh]-\Psi_1[\ff])(\bx,\bar{\bx})}{\tau} - (\Psi_2[\ff]-\Psi_1^2[\ff])(\bx,\bar{\bx})\D_{\bar{\bx}}[\hh](\bx)\right\rbrace  \left(\D_{\bar{\bx}}[\p_{\bx}\ff](\bx)\right)^2\\
=2 \p_{\bx}\Psi_1[\ff](\bx,\bar{\bx}) \mathsf{A}[\ff,\hh](\bx,\bar{\bx}) \D_{\bar{\bx}}[\p_{\bx}\ff](\bx).
\end{multline*}
Combining \eqref{DPsi1} with Corollary \ref{boundstilde} and \eqref{bound_f'-f''} we obtain the bound $2 |\p_{\bx}\Psi_1[\ff](\bx,\bar{\bx})  \D_{\bar{\bx}}[\p_{\bx}\ff](\bx)|\leq C(\e).$ After that, as an immediate consequence of \eqref{integralA} we get the required bound for the second term. For the latter, we proceed in the same spirit as before, trying to split that term into previously studied terms. Notice that
\begin{multline*}
2 \Psi_1[\ff](\bx,\bar{\bx})\p_{\bx}\left\lbrace \frac{(\Psi_1[\ff+\tau\hh]-\Psi_1[\ff])(\bx,\bar{\bx})}{\tau}-(\Psi_2[\ff]-\Psi_1^2[\ff])(\bx,\bar{\bx})\D_{\bar{\bx}}[\hh](\bx)\right\rbrace \left(\D_{\bar{\bx}}[\p_{\bx}\ff](\bx)\right)^2\\
=2 \Psi_1[\ff](\bx,\bar{\bx})\p_{\bx}\mathsf{A}[\ff,\hh](\bx,\bar{\bx}) \D_{\bar{\bx}}[\p_{\bx}\ff](\bx)-2 \Psi_1[\ff](\bx,\bar{\bx})\mathsf{A}[\ff,\hh](\bx,\bar{\bx}) \D_{\bar{\bx}}[\p_{\bx}^2\ff](\bx)\\
=2 \Psi_1[\ff](\bx,\bar{\bx}) \left\lbrace \mathsf{A}_1'[\ff,\hh](\bx,\bar{\bx})+\mathsf{A}_4'[\ff,\hh](\bx,\bar{\bx})+\mathsf{A}_5'[\ff,\hh](\bx,\bar{\bx})\right\rbrace \D_{\bar{\bx}}[\p_{\bx}\ff](\bx).
\end{multline*}
Using Corollary \ref{boundsPSI} and \eqref{bound_f'-f''} we obtain the bound $2 |\Psi_1[\ff](\bx,\bar{\bx})\D_{\bar{\bx}}[\p_{\bx}\ff](\bx)|\leq C(\e)$. Finally, working as we did to get inequality \eqref{A4'}  we get the required bound for each of the above terms. Therefore, combining all we have proved our goal.
\end{proof}

\textbf{Acknowledgments:}  AC is supported by the Spanish Ministry of Science and Innovation, through
the Severo Ochoa Programme for Centers of Excellence in R\&D (CEX2019-000904-S),
and by grants  Europa Excelencia program ERC2018-092824 and  RED2018-102650-T funded by MCIN/AEI/10.13039/501100011033. AC and DL are partially supported by grants MTM2017-89976-P and  PID2020-114703GB-I00 funded by MCIN.

\end{document}